\documentclass[reqno]{amsart}
\usepackage{graphicx} % Required for inserting images
\usepackage{amsmath}
\usepackage{amsthm,amssymb,mathrsfs}
\usepackage{bm}
\usepackage{enumerate}
\usepackage{xcolor}
\usepackage{dsfont}
\usepackage{anyfontsize}

\newcommand{\f}{\bm f}
\newcommand{\e}{\bm e}
\newcommand{\uu}{\bm u}

\newcommand{\ww}{\bm w}
\newcommand{\T}{\bm T}
\newcommand{\R}{\mathbb{R}}
\newcommand{\N}{\mathbb{N}}
\newcommand{\norm}[1]{\left\Vert#1\right\Vert}

\def\calA{{\mathcal{A}}}
\def\calB{{\mathcal{B}}}
\def\calL{{\mathcal{L}}}

\newtheorem{theorem}{Theorem}[section]
\newtheorem{corollary}[theorem]{Corollary}
\newtheorem{lemma}[theorem]{Lemma}
\newtheorem{proposition}[theorem]{Proposition}

\newtheorem{definition}[theorem]{Definition}

\newtheorem{hyp}[theorem]{Hypotheses}
\newtheorem{remark}[theorem]{Remark}

\allowdisplaybreaks

\numberwithin{equation}{section}

\title[Kernel estimates for parabolic systems of PDEs]{Kernel estimates for parabolic systems of partial differential equations with unbounded coefficients}

\author[D. Addona]{Davide Addona}
\author[L. Lorenzi]{Luca Lorenzi}
\author[M. Porfido]{Marianna Porfido}

\address{D.Addona, L.Lorenzi: Plesso di Matematica, Dipartimento di Scienze Matematiche, Fisiche e Informatiche, Università di Parma, Parco Area delle Scienze 53/A, I-43124 Parma, Italy}
\email{davide.addona@unipr.it, luca.lorenzi@unipr.it}

\address{M. Porfido: Technische Universit\"at Bergakademie Freiberg,
Faculty of Mathematics and Computer Science
Institute of Applied Analysis,
D-09596 Freiberg, Germany}
\email{Marianna.Porfido@math.tu-freiberg.de}

\thanks{The authors are members of G.N.A.M.P.A. of the Italian Istituto Nazionale di Alta Matematica (INdAM). The authors have been partially funded by the project INdAM - G.N.A.M.P.A. ``Operatori ellittici vettoriali a coefficienti illimitati in spazi $L^p$'', CUP$\_$E53C22001930001.}


\date{\today}

\begin{document}

\begin{abstract}
We provide pointwise upper bounds for the transition kernels of semigroups associated with a class of systems of nondegenerate elliptic partial differential equations with unbounded coefficients with possibly unbounded diffusion coefficients, which may vary equation by equation.
\end{abstract}

\maketitle

\section{Introduction}
The aim of this paper is to provide pointwise upper estimates for the kernels of semigroups associated with vector-valued elliptic operators $\bm\calA$ in divergence form and defined on smooth functions $\f:\R^d\to \R^m$ as
\begin{align}
(\boldsymbol{\calA} \f)_h
= \mathrm{div} (Q^h\nabla  f_h) +\langle b^h,\nabla f_h\rangle -(V{\boldsymbol f})_h
=\sum_{i,j=1}^dD_{i} (q_{ij}^hD_j f_h)+\sum_{i=1}^db_i^hD_if_h-\sum_{k=1}^mv_{hk}f_k
\label{operatore}
\end{align}
for $h=1,\ldots,m$, where $Q^h:\R^d\to \R^{d\times d}$, $b^h:\R^d\to \R^d$ for every $h=1,\ldots,m$ and $V:\R^d\to \R^{m\times m}$ are smooth enough functions.

In the scalar case, nowadays estimates from above for the transition kernels of semigroups associated with elliptic operators with unbounded and smooth enough coefficients are well-established and available in the literature
(see e.g., \cite{ALR,ALM,AMP08,LaiMetPalRha11,MetPalRha10,MS,OU,S} for the case of bounded diffusion coefficients and \cite{BRT19,CRT17, CT16,CKPR, DRT16,KLR, KLR1,KPR, LR,MR,MS1} for classes of equations with unbounded diffusion coefficients). 
More recently, kernel estimates have been proved also in some cases where the diffusion coefficients are unbounded in a neighborhood of the origin but bounded at infinity (see, \cite{CMNS,MNS,MSS}). Such classes of operators contain Schr\"odinger operators with inverse square potential. 
We also point out that kernel estimates have been obtained also for some classes of fractional Kolmogorov equations (see e.g., \cite{Bak12,PorRhaTac24}).
Finally, we mention that estimates of the invariant measures associated to scalar semigroups (which are, more generally speaking, solutions to Fokker-Planck-Kolmogorov equations) are the elliptic counterpart of the estimates of the transition kernels and have been widely studied with different approaches (see e.g., \cite{BKR,BRS,FFMP,MPR,Sh}). 

Differently from the scalar setting, in the vector-valued case, to the best of our knowledge, only in \cite{ALMR,MR} the authors prove estimates for the transition kernels of semigroups associated to systems of elliptic partial differential equations with unbounded coefficients. In  the quoted papers, the elliptic equations have all the same diffusion part and form methods and Davies' trick are used to get the desired estimates. In particular, in \cite{ALMR}, where, differently from \cite{MR}, the diffusion coefficients are allowed to be unbounded over $\R^d$, the kernel estimates are obtained, in terms of a distance associated to the diffusion coefficients, for a scalar semigroup $(T(t))$, associated to a second-order elliptic operator, which pointwise controls the vector valued semigroup $(\T(t))$ (i.e. $|\T(t)\f|\le T(t)|\f|$ for every $t\in (0,\infty)$ and $\f\in C_b(\R^d;\R^m)$). It is worth noticing that such a distance, associated to the diffusion terms, is equivalent to the Euclidean one if and only if the diffusion coefficients are bounded and uniformly elliptic.

The arguments exploited in the quoted papers seem not to be generalized to the case when the diffusion coefficients vary from equation to equation. In particular, in such a situation there exists no scalar semigroup that controls the vector-valued one. Hence, the problem of determining estimates for the transition kernels of the vector-valued semigroup cannot be reduced, by comparison, to the problem of determining kernel estimates for a scalar semigroup.

In this paper, we prove pointwise kernel estimates, adapting and generalizing to our setting the  techniques exploited in the scalar case in \cite{ALR,CKPR,KLR,KLR1,KPR,LaiMetPalRha11,MetPalRha10,S}, based on time-dependent Lyapunov functions. We also take advantage and extend some results proved in \cite{AddLor23,AngLor20} on weakly coupled systems of parabolic equations with unbounded coefficients and possible different diffusion terms equation by equation. We stress that our approach strongly relies on the validity of a generalized version of the classical maximum principle.
Hence, unfortunately, it cannot be directly extended to more general situations where a maximum principle cannot hold, such as the case of systems of elliptic equation coupled at the level of the second-order derivatives. However, as our two examples show, if compared with those in \cite{CKPR,KPR}, then our kernel estimates reduce to those of the scalar case if $m=1$. 

The paper is organized as follows.
In Section \ref{sec:prliminaries}, we introduce our standing assumptions, some results from \cite{AddLor23, AngLor20} and some properties of the transition kernels of the semigroups $(\T(t))$ and $(\T^P(t))$, associated in $C_b(\R^d)$ to the operators $\bm{\calA}$ and $\bm \calA^{P}$, respectively. Here, $\bm\calA^P$ is the elliptic operator with the same diffusion and drift coefficients as $\bm\calA$ and the potential matrix $V$ replaced by the potential matrix $V^P$, which differs from $V$ just in the elements outside the main diagonal (more precisely, $v_{hk}$ is replaced by $-|v_{hk}|$ if $h\neq k$). In Section \ref{sec:lyapunov}, we introduce time-dependent Lyapunov functions, which are unbounded functions, for the operator $D_t+\bm{\calA}^P$ and prove that the action of $\T^P(t)$ on such functions is well-defined for every $t>0$. We stress that, due to the presence of coupling unbounded potential terms, the computations which we perform are more involved and delicate than the analogous of the scalar case.

The results in Sections \ref{sec:prliminaries} and \ref{sec:lyapunov} are crucial to prove the main results of this paper (i.e., the pointwise kernel estimates) which are the content of Section \ref{sec:kernels}. We underline that, differently from \cite{CKPR,KLR,KLR1,KPR}, the coupling at zero order does not allow us to {\it a-priori} prove  that the kernels associated to the vector-valued semigroup are bounded when the diffusion coefficients are bounded. However, this property is crucial to apply the procedure developed in \cite{KLR}. Hence, instead of approximating the unbounded diffusion coefficients by means of bounded ones, which is the technique exploited in the quoted papers, here, for every $n\in\N$, we prove pointwise estimates for the kernels of the semigroup $(\T^{\mathcal D,P,n}(t))$, associated to the realization of the operator $\bm\calA^P$ in $C_b(B(n);\R^m)$ with homogeneous Dirichlet boundary conditions, showing that such estimates are independent on $n\in\N$. Therefore, letting $n$ tend to infinity and exploiting the convergence of $(\T^{\mathcal D,P,n}(t))$ to $(\T^{P}(t))$ and the link between $(\T^{P}(t))$ and $(\T(t))$, we deduce pointwise estimates for the kernels of the semigroup $(\T(t))$.
We conclude Section \ref{sec:kernels} by showing that, under suitable assumptions, analogous results can be obtained for the kernels of the semigroup associated to the (formal) adjoint of the operator $\bm\calA^P$.

Finally, in Section \ref{sec:examples} we provide two different classes of systems of elliptic operators with unbounded coefficients to which our results apply. In particular, under our assumptions we are able to consider both polynomially and exponentially growing coefficients. 

\smallskip

\paragraph{\bf Notation.} Here, we introduce some notation used in the paper.

{\it General notation.}
By $B(r)$, we denote the open ball of $\R^d$ centered at the origin with radius $r$, and by $\bm0$ the null vector of $\R^d$. For every $k=1,\ldots,m$, ${\bm e}_k$ is the $k$-th vector of the canonical basis of $\R^m$. Moreover, for every $a,b\in\R$, with $a<b$, we set $R_n(a,b)=(a,b)\times B(n)$ $(n\in\N)$ and $R(a,b)=(a,b)\times\R^d$.
Given an open set $A\subseteq \R^d$, a function $\bm{f}:A\to\R^m$ and $\sigma>0$, we denote by $f_1,\ldots,f_m$ its components and by $\bm{f}\wedge \sigma$ the function whose $j$-th component is $f_j\wedge\sigma$, $j=1,\ldots,m$. Further, we denote by $|\bm f|$ the vector-valued function whose $j$-th component is $|f_j|$ for every $j=1,\ldots,m$. If $\bm g:A\to\R^m$ is another function, then $\f\geq \bm g$ means that $f_j\geq g_j$ on $A$ for every $j=1,\ldots,m$.

If $u: J\times A\to \R$, where $J\subseteq [0,\infty)$ is an
interval, then we use the following notation:
$D_t u =\frac{\partial u}{\partial t}$, $D_iu=\frac{\partial u}{\partial x_i}$, $D_{ij}u=\frac{\partial^2u}{\partial x_i\partial x_j}$ and $\nabla u=(D_1u, \dots, D_du)$. Moreover, $|\nabla u(t,x)|$ and $|D^2u(t,x)|$ denote, respectively, the Euclidean norm of $\nabla u(t,x)$ and the norm of the spatial Hessian matrix $D^2u(t,x)$.
For $F:A\to \R^d$, we set $\mathrm{div}(F)=\sum_{i=1}^dD_iF_i$.

% {\it Function spaces of scalar functions.} We assume that the reader is familiar with the spaces $C^k(A)$, when $k\in [0,\infty)\cup\{\infty\}$. We use the subscript ``$b$'', to stress that all the functions that we are considering are bounded together with all the existing derivatives. We use the  subscript ``$c$'' to stress that the functions that we are considering have compact support in $A$.
% We denote by $C_0(\R^d)$ the space of continuous functions in $\R^d$ vanishing as $|x|$ tends to $\infty$.
% $B_b(A)$ is the space of bounded measurable functions defined in $A\subseteq\R^d$.

{\it Spaces of functions.} We assume that the reader is familiar with the spaces $C^k(A)$, when $k\in [0,\infty)\cup\{\infty\}$.
We use the subscript ``$b$'', to stress that all the real-valued functions that we are considering are bounded together with all the existing derivatives. We use the  subscript ``$c$'' to stress that the functions that we are considering have compact support in $A$, whereas
the subscript ``$0$'' is used, when $A=\R$ and $k=0$, for functions vanishing at $\infty$.
$B_b(A)$ is the space of bounded Borel measurable functions defined in $A\subseteq\R^d$.  

We also assume that the reader is familiar with the parabolic spaces $C^{1,2}(E)$, $C^{\frac{\alpha}{2} ,\alpha}(E)$ and $C^{1+\frac{\alpha}{2} ,2+\alpha}(E)$, ($\alpha\in (0,1)$), where $E$ is a subset of $[0,\infty)\times\R^d$.
Also for parabolic spaces, we use the subscript ``$c$''. 
Sometimes, we will use the local H\"older space $C_{\rm loc}^{\frac{\alpha}{2} ,\alpha}(E)$, whose elements belong to $C^{\frac{\alpha}{2} ,\alpha}(B)$ for every compact subset $B$ of $E$.

For every open set $A\subseteq\R^d$, $0\le a<b<\infty$ and $p\in [1,\infty)\cup\{\infty\}$, we will consider the usual parabolic Sobolev spaces $W^{0,1}_p((a,b)\times A)$,
$W^{1,2}_p((a,b)\times A)$ and the subset of $W^{0,1}_p((a,b)\times A)$, denoted by $\mathcal{H}^{p,1}((a,b)\times A)$, of functions whose distributional time derivative belongs to $(W_{p'}^{0,1}((a,b)\times A))'$, the dual space of $W_{p'}^{0,1}((a,b)\times A)$, where $1/p+1/p'=1$. This space is endowed with the norm
$\|u\|_{\mathcal{H}^{p,1}((a,b)\times A)}= \|D_t u\|_{(W_{p'}^{0,1}((a,b)\times A))'}+ \|u\|_{W_p^{0,1}((a,b)\times A)}$.

Since no confusion may arise, we use the previous notation also for spaces of vector-valued functions. For instance, $\f\in C^k(A)$ means that all the components of $\f$ are  $k$-times continuously differentiable in $A$.

% {\it Function spaces of vector-valued functions.}
% We denote by $K(A)$ when $K=C_b$, $C^\alpha$, $C_b^k$, $C^\infty_c$, $C_0$, $B_b$, the space of functions $\bm{f}:A\to\R^m$ such that $f_j$ belongs to $K(A)$ for every $j=1,\dots,m$.
% Similarly, 
% we denote by $K(E)$ when $K=C^{1,2},C^{1,2}_c$, $C^{\frac{\alpha}{2},\alpha}$, $C^{1+\frac{\alpha}{2},2+\alpha}$, $C_b^{1+\frac{\alpha}{2},2+\alpha}$, $C^{\frac{\alpha}{2},\alpha}_{\rm loc}$, the space of functions $\bm{f}:A\to\R^m$ such that $f_j$ belongs to $K(A)$ for every $j=1,\dots,m$. Finally, the sup-norm of a function $\f:A\to\R^m$ is the sup norm of the function $x\mapsto\left (\sum_{j=1}^m|f_j(x)|^2\right )^{1/2}$.

\section{Main assumptions and preliminary results}
\label{sec:prliminaries}
Throughout the paper, if not otherwise specified, we consider the following assumptions. We recall that $\bm\calA$ is the operator defined by \eqref{operatore} and $\bm\calA^P$ is defined as the operator $\bm\calA$, with the matrix-valued potential $V=(v_{hk})_{h,k=1}^m$ replaced by the matrix-valued potential $V^P=(v^P_{hk})_{h,k=1}^m$, whose entries are given by $v_{hh}^P=v_{hh}$ and $v_{hk}^P=-|v_{hk}|$ for every $h,k=1,\ldots,m$, with $h\neq k$.

\begin{hyp}
\label{hyp-base}
\begin{enumerate}[\rm (i)]
\item
For every $i,j=1,\ldots,d$ and $h,k=1,\ldots,m$, the coefficients $q^h_{ij}=q^h_{ji}$ belong to $C^{1+\alpha}_{\rm loc}(\R^d)$, whereas $b_i^h$ and $v_{hk}$ belong to $C^{\alpha}_{\rm loc}(\R^d)$ for some $\alpha\in (0,1)$;
\item
the infimum $\eta_h^0$ over $\R^d$ of the minimum eigenvalue $\eta_h(x)$ of the matrix $Q^h(x)=(q^h_{ij}(x))_{i,j=1}^d$, $x\in\R^d$, is positive for every $h=1,\ldots,m$;
\item
there exist a positive function $\varphi\in C^2(\R^d)$, blowing up as $|x|$ tends to $\infty$, and $\lambda\ge 0$
such that ${{\bm{\mathcal A}}^P}\boldsymbol\varphi\le\lambda\boldsymbol\varphi$ in $\R^d$, where $\varphi_j=\varphi$ for every $j=1,\ldots,m$;
\item
the sum of the elements of each row of $V^P$ is a bounded from below function on $\R^d$, i.e., there exists $M\in\R$ such that
\begin{align}
\label{segno_V_P}
\sum_{k=1}^mv_{hk}^P(x)\geq M, \qquad\;\, x\in\R^d, \;\, h=1,\ldots,m.    
\end{align}
\end{enumerate}
\end{hyp}

\begin{definition}
A function $\bm \varphi$ satisfying Hypothesis $\ref{hyp-base}(iii)$ is called Lyapunov function for the operator $\bm{\calA}^P$.    
\end{definition}

\subsection{The semigroups $(\T(t))$ and $(\T^P(t))$ and basic properties of their kernels}
From \cite[Proposition 2.4]{AddLor23}, we infer that, under Hypotheses \ref{hyp-base}, for every $\f\in C_b(\R^d)$ the Cauchy problems
\begin{align}
(i)\,\left\{
\begin{array}{ll}
D_t{\uu}=\boldsymbol{\calA}\uu,   & {\rm in}~ (0,\infty)\times \R^d, \\[1mm]
\uu(0,\cdot)=\f,     & {\rm in}~\R^d, 
\end{array}
\right.
\qquad\;\,
(ii)\,\left\{
\begin{array}{ll}
D_t{\uu}=\boldsymbol{\calA}^P\uu,   & {\rm in}~ (0,\infty)\times \R^d, \\[1mm]
\uu(0,\cdot)=\f,     & {\rm in}~\R^d, 
\end{array}
\right.
\label{prob_cauchy_1}
\end{align}
are uniquely solvable with solutions, denoted by $\uu$ and $\uu^P$, respectively, which are bounded in each strip $[0,T]\times\R^d$, belong to $C([0,\infty)\times \R^d)\cap C^{1+\frac{\alpha}{2},2+\alpha}_{\rm loc}((0,\infty)\times \R^d)$ and, for every $t\geq0$, fulfill the estimates 
\begin{align}
\label{stima_sol_Cauchy_1}
\|\uu(t,\cdot)\|_{\infty}\leq \sqrt me^{-Mt}\max_{k=1,\ldots,m}\|f_k\|_{\infty},\qquad\;\,
\|\uu^P(t,\cdot)\|_\infty\leq \sqrt{m}e^{-Mt}\max_{k=1,\ldots,m}\|f_k\|_{\infty}.
\end{align}

By means of the solution to the corresponding Cauchy problem, we can associate a semigroup of bounded linear operators in $C_b(\R^d)$ to $\boldsymbol{\calA}$ (resp. $\boldsymbol{\calA}^P$) as follows:
for every $\f\in C_b(\R^d)$ we set $\T(t)\f=\uu(t,\cdot)$ for every $t\in [0,\infty)$ (resp. $\T^P(t)\f=\uu^P(t,\cdot)$ for every $t\in [0,\infty)$).

\begin{proposition}
\label{prop:prop_semigruppi}
The following properties are satisfied.
\begin{enumerate}[\rm (i)]
\item 
There exist families of Borel measures $\{p_{hk}(t,x,dy):t>0,\ x\in \R^d, \ h,k=1,\ldots,m\}$ and $\{p^P_{hk}(t,x,dy):t>0,\ x\in \R^d, \ h,k=1,\ldots,m\}$ such that
\begin{align}
(\T(t)\f(x))_h=\sum_{k=1}^m\int_{\R^d} f_k(y)p_{hk}(t,x,dy), \qquad\;\, \f\in C_b(\R^d), \label{int_form_smgr_1}\\  
(\T^P(t)\f(x))_h=\sum_{k=1}^m\int_{\R^d} f_k(y)p^P_{hk}(t,x,dy), \qquad\;\, \f\in C_b(\R^d). \label{int_form_smgr_2}
\end{align}
\item 
For every $t>0$, $ x\in\R^d$ and $h,k\in\{1,\ldots,m\}$, each measure $p_{hk}(t,x,dy)$ $($resp. $p_{hk}^P(t,x,dy))$
is absolutely continuous $($resp. nonnegative and absolutely continuous$)$ with respect to the Lebesgue measure $\mathcal L_d$ on $\R^d$. 
\end{enumerate}
\noindent In particular, the semigroups $(\T(t))$ and $(\T^P(t))$ extend to $B_b(\R^d)$ with strong Feller semigroups, still defined by
formulae \eqref{int_form_smgr_1} and \eqref{int_form_smgr_2}.
\end{proposition}

\begin{proof}
The first two statements follow adapting the proofs of \cite[Proposition 3.2 \& Theorem 3.3]{AddAngLorTes17}.   

From property (ii), it is clear the semigroups
$(\T(t))$ and $(\T^P(t))$ can be extended to $B_b(\R^d)$ through formulae \eqref{int_form_smgr_1} and \eqref{int_form_smgr_2}. Moreover,
if $(\f_n)\subset C_b(\R^d)$ is a bounded sequence which pointwise converges to $\f\in B_b(\R^d)$ almost everywhere in $\R^d$, then the sequences $(\T(t)\f_n)$ and $(\T^P(t)\f_n)$ converge, respectively, to $\T(t)\f$ and $\T^P(t)\f$ pointwise in $\R^d$ for every $t\in (0,\infty)$.

Actually, the interior Schauder estimates in 
\cite[Theorem A.2]{AngLor20} imply that, for every $t>0$, the sequences $(\T(t)\f_n)$ and $(\T^P(t)\f_n)$ are bounded in $C^2(K)$ for every compact set $K\subset\R^d$. Hence, by Ascoli-Arzel\`a theorem, it follows that the sequences $(\T(t)\f_n)$ and $(\T^P(t)\f_n)$ converge to bounded and continuous functions, which, by uniqueness, coincide with $\T(t)\f$ and $\T^P(t)\f$, respectively. 

Finally, we observe that if $\f\in B_b(\R^d)$ and $(\f_n)\subset C_b(\R^d)$ is a bounded sequence which converges almost everywhere in $\R^d$ to $\f$, then
$\T(t+s)\f_n=\T(t)\T(s)\f_n$ in $\R^d$ for every $s,t\in (0,\infty)$ and $n\in\N$. The left-hand side of the previous formula converges to $\T(t+s)\f$ pointwise in $\R^d$. Moreover, the sequence $(\T(s)\f_n)$ converges to $\T(s)\f$ pointwise in $\R^d$ and $\|\T(s)\f_n\|_{\infty}\le\sqrt{m}e^{-Ms}\sup_{n\in\N}\max_{k=1,\ldots,m}\|(f_n)_k\|_{\infty}$ (see \eqref{stima_sol_Cauchy_1}) for every $n\in\N$. Hence, we can apply
\cite[Proposition 2.7]{AddLor23} to infer that the sequence $(\T(t)\T(s)\f_n)$ pointwise converges in $\R^d$ to
$\T(t)\T(s)\f$ as $n$ tends to $\infty$. This shows that $(\T(t))$ is a semigroup of bounded linear operators in $B_b(\R^d)$. The same arguments and \cite[Theorem 2.7]{AngLor20} can be used to prove that $(\T^P(t))$ is a semigroup of bounded operators in $B_b(\R^d)$ as well. 

Summing up, we have proved that $(\T(t))$ and $(\T^P(t))$ are strong Feller semigroups.
\end{proof}

\begin{remark}
%\label{rmk-2.4}
{\rm In view of Proposition \ref{prop:prop_semigruppi}, 
for every $h,k=1,\ldots,m$, we can determine two scalar functions $p_{hk}$ and $p_{hk}^P$ defined in $(0,\infty)\times\R^d\times\R^d$ such that
$p_{hk}(t,x,dy)=p_{hk}(t,x,y)dy$ and
$p_{hk}^P(t,x,dy)=p_{hk}^P(t,x,y)dy$ for every $(t,x)\in (0,\infty)\times\R^d$. In particular, $p_{hk}(t,x,\cdot)$ and $p^P_{hk}(t,x,\cdot)$ belong to $L^1(\R^d)$ for every $(t,x)\in (0,\infty)\times\R^d$ and their $L^1$-norms can be bounded from above by $\sqrt{m}e^{-Mt}$.}
\end{remark}

\subsection{Approximation results and consequences}
In this subsection, we approximate the semigroups $(\T(t))$ and $(\T^P(t))$ by means of semigroups associated to elliptic operators %with smooth and bounded coefficients 
which are defined in a ball of $\R^d$ or in the whole $\R^d$. Through these approximations, we then prove some results which will play a crucial role in the proof of the main results of the paper.

We denote by $(\T^{\mathcal D,n}(t))$ (resp. $(\T^{\mathcal D,P,n}(t))$) the semigroup generated by the realization of the operator $\boldsymbol{\calA}$ (resp. $\boldsymbol{\calA}^P$) in $C_b(B(n))$ with homogeneous Dirichlet boundary conditions. For every $t>0$, let $\{p_{hk}^{\mathcal D,n}(t,x,y): x,y\in B(n),\ h,k=1,\ldots,m\} $ (resp. $\{p_{hk}^{\mathcal D,P,n}(t,x,y):\ x,y\in B(n),\ h,k=1,\ldots,m\} $) be the kernels of the operator $\T^{\mathcal D,n}(t)$ (resp. $\T^{\mathcal D,P,n}(t)$), i.e.,
\begin{align}
(\T^{\mathcal D,n}(t)\f(x))_h=\sum_{k=1}^m\int_{B(n)} f_k(y)p^{\mathcal D,n}_{hk}(t,x,y)dy,\qquad\;\, \f\in C_b(B(n)) 
\label{int_form_smgr_dirich}
\end{align}
and
\begin{align}
(\T^{\mathcal D,P,n}(t)\f(x))_h=\sum_{k=1}^m\int_{B(n)} f_k(y)p^{\mathcal D,P,n}_{hk}(t,x,y)dy, \qquad\;\, \f\in C_b(B(n))  
\label{int_form_smgr_dirich_pos}
\end{align}
for every $t>0$ and $x\in B(n)$. It is well-known that, for every $t>0$ and $n\in\N$, the functions $(x,y)\mapsto p_{kh}^{\mathcal D,P,n}(t,x,y)$ are positive almost everywhere in $B(n)\times B(n)$. This implies that for every $t\geq0$ the operator $\T^{\mathcal D,P,n}(t)$ maps the cone of componentwise nonnegative functions of $C_b(B(n))$ into itself.

Arguing as in the proof of \cite[Proposition 2.3]{AddLor23} (which deals with realization of $\bm\calA$ and $\bm\calA^P$ in $C_b(B(n))$ with homogeneous Neumann boundary conditions), it follows that
\begin{align}
\label{dis_smgr_dirichlet}
|((\T^{\mathcal D,n}(t)\f)(x))_h|
\leq ((\T^{\mathcal D,P,n}(t)|\f|)(x))_h,\qquad\;\,t\in (0,\infty),\;\,x\in B(n),
\end{align}
for every $n\in\N$, $\f\in B_b(\R^d)$ and $h=1,\ldots,m$. In particular, from \eqref{int_form_smgr_dirich}, \eqref{int_form_smgr_dirich_pos} and \eqref{dis_smgr_dirichlet} we infer that
$|p_{hk}^{\mathcal D,n}(t,x,y)|\leq p_{hk}^{\mathcal D,P,n}(t,x,y)$ for every $n\in\N$, $t>0$, $x,y\in B(n)$ and  $h,k=1,\ldots,m$.

From \cite[Theorem 2.4]{AddLor23} (whose proof holds true also if one considers approximations by means of solutions to Cauchy problems on balls with homogeneous Dirichlet boundary conditions), it follows that
\begin{align}
\label{conv_smgr_Dirich}
\lim_{n\to \infty}\T^{\mathcal D,n}(\cdot)\f=\T(\cdot)\f \qquad \textrm{and }\qquad
\lim_{n\to\infty}\T^{\mathcal D,P,n}(\cdot)|\f|= \T^P(\cdot)|\f| 
\end{align}
in $C^{1+\frac{\alpha}{2},2+\alpha}(E)$ for every compact set  $E\subset (0,\infty)\times \R^d$ and every $\f\in C_b(\R^d)$. From \eqref{dis_smgr_dirichlet} and \eqref{conv_smgr_Dirich} we deduce that
\begin{align}
\label{stima_smgr_comp_k}
|(\T(t)\f)_h(x)|
\leq (\T^P(t)|\f|)_h(x)
\end{align}
for every $\f\in C_b(\R^d)$, $t>0$, $x\in \R^d$ and  $h=1,\ldots,m$. 
\begin{proposition} 
\label{prop:mono_nuclei}
For every $j,n\in\N$ with $j<n$ and every $(t,x)\in(0,\infty)\times B(j)$, it holds that $(\T^{\mathcal D,P,j}(t)\f)(x)\leq (\T^{\mathcal D,P,n}(t)\f)(x)$ for every $\f\in C_b(\R^d)$ with $\f\geq \bm0$. \\
As a byproduct, $p^{\mathcal D,P,j}_{hk}(t,x,y)\leq p^{\mathcal D,P,n}_{hk}(t,x,y)$ for every $t\in (0,\infty)$ and $x,y\in B(j)$. Further, for every $\f\in C_b(\R^d)$, with $\f\geq \bm0$, and every $(t,x)\in(0,\infty)\times \R^d$, the sequence $((\T^{\mathcal D,P,n}(t)\f)(x))$ monotonically componentwise converges to $(\T^P(t)\f)(x)$ as $n$ tends to $\infty$.
\end{proposition}

\begin{proof}
Fix $j$, $n$ and $\f$ as in the statement and  set ${\ww}_{j,n}:=\T^{\mathcal D,P,j}(\cdot)\f-\T^{\mathcal D,P,n}(\cdot)\f$. A simple computation shows that  $D_t\ww_{j,n}-\boldsymbol{\calA}\ww_{j,n} =\bm{0}$ on $(0,\infty)\times B(j)$, $\ww_{j,n}\leq \bm{0}$ on $(0,\infty)\times \partial B(j)$ and $\ww_{j,n}(0,\cdot)=\bm0$ on $B(j)$. From the maximum principle in \cite[Theorem 3.13]{ProWin67} we deduce that $\ww_{j,n}\leq\bm{0}$ on $[0,\infty)\times B(j)$,  i.e., $\T^{\mathcal D,P,j}(t)\f\leq \T^{\mathcal D,P,n}(t)\f$ in $B(j)$ for every $t\geq0$.

We now take $\f=f \e_k$, where $f\in C_b(\R^d)$  is nonnegative and has support contained in $\overline{B(j)}$, $k\in\{1,\ldots,m\}$ and $\e_k$ is the $k$-th component of the canonical basis of $\R^d$, and obtain that
\begin{align*}
\int_{B(j)}p_{hk}^{\mathcal D,P,j}(t,x,y)f(y)dy\leq  \int_{B(j)}p_{hk}^{\mathcal D,P,n}(t,x,y)f(y)dy  
\end{align*}
for every $(t,x)\in (0,\infty)\times B(j)$ and $h=1,\ldots,m$.
From the arbitrariness of $f$ and $k$, we deduce that $p_{hk}^{\mathcal D,P,j}\leq p_{hk}^{\mathcal D,P,n}$ on $(0,\infty)\times B(j)\times B(j)$. 

The last part of the statement follows from the first part and \eqref{conv_smgr_Dirich}.
\end{proof}

We notice that, from \eqref{int_form_smgr_1}, \eqref{int_form_smgr_2} and \eqref {stima_smgr_comp_k}, 
\begin{align*}
\left|\int_a^bdt\int_{\R^d} p_{hk}(t,x,y)f(t,y)dy\right|
 = & \left|\int_a^b((\bm T(t)(f(t,\cdot)\bm e_k))(x))_hdt\right|
\! \leq \! \int_a^b|((\bm T(t)(f(t,\cdot)\bm e_k))(x))_h|dt \\
\leq & \int_a^b((\bm T^P(t)(f(t,\cdot)\bm e_k))(x))_hdt
=\int_a^b dt\int_{\R^d} p_{hk}^P(t,x,y)f(t,y)dy
\end{align*}
for every $a,b\in\R$, with $0\leq a<b$, and every nonnegative function  $f\in C_b([a,b]\times \R^d)$.
This implies that, for every $a,b\in\R$, with $0\leq a<b$, every $x\in\R^d$ and $h,k=1,\ldots,m$, 
\begin{align}
|p_{hk}(t,x,y)|\leq p_{hk}^P(t,x,y) \qquad \textrm{for a.e. }(t,y)\in (a,b)\times \R^d.
\label{stima_nuclei_completa}
\end{align}

\begin{remark}
{\rm 
We stress that the results quoted from \cite{AddLor23,AngLor20} have been proved assuming additionally that a nontrivial set $E\subset \{1, \ldots, m\}$, such that $v_{hk}$ identically vanishes on $\R^d$ for every $h\in E$ and $k\notin E$, does not exist. Actually this condition, which implies that the system \eqref{prob_cauchy_1}(i) cannot be decoupled,
is used only to prove that all the measures
$p_{hk}^P(t,x,dy)$ are equivalent to the Lebesgue measure (see \cite[Proposition 2.8]{AngLor20}). As we will see in Proposition \ref{prop:gen_eq_meas}, without this assumption, for every $(t,x)\in(0,\infty)\times \R^d$ and every $h,k=1,\ldots,m$, all the measures $p_{hk}^P(t,x,dy)$, that are not trivial, are equivalent to the Lebesgue measure.} \end{remark}

% We now recall the following maximum principle.

% \begin{proposition}[Theorem 3.13 of \cite{ProWin67}]
% \label{prop:max_princ_palla_1}
% Assume that Hypotheses $\ref{hyp-base}(i)$-$(ii)$ are satisfied and fix $n\in\N$. If $\uu\in C_b((0,T]\times \overline{B(n)})\cap C_b([0,T]\times B(n))\cap C^{1,2}((0,T]\times B(n))$ satisfies the system
% \begin{align*}
% \left\{
% \begin{array}{ll}
% D_t{\uu}(t,x)-(\boldsymbol{\calA}^P\uu)(t,x)\leq \bm 0,   & (t,x)\in(0,T]\times B(n), \\[1mm]
% \uu(t,x)\leq \bm 0, & t\in(0,T]\times \partial B(n), \\[1mm]
% \uu(0,x)\leq \bm 0,     & x\in B(n), 
% \end{array}
% \right.
% \end{align*}
% then $\uu\leq \bm0$ in $[0,T]\times B(n)$.
% \end{proposition}

% As a consequence of the maximum principle in Proposition \ref{prop:max_princ_palla_1} we get the following result.

\begin{remark}
\label{rmk:analiticita}
{\rm
For every $\ell\in\{1,\ldots,m\}$ and $n\in\N$, we introduce the operator $\calA^{\ell}$, defined on scalar smooth functions $u:B(n)\to \R$ by
$\calA^{\ell} u
= {\rm div}(Q^{\ell}\nabla u)+\langle b^{\ell},\nabla u\rangle -v_{\ell \ell}u$.
For every $n\in\N$, $\bm\calA$ can be seen as a perturbation of the diagonal operator $\bm A$, whose components are the operators $\calA^1,\ldots,\calA^m$, by means of the operator $\bm B=(V-{\rm diag}(v_{11},\ldots,v_{mm}))$ in $C_b(B(n))$. Since the $k$-th component of $\bm A$, with domain $\{f\in\bigcap_{p<\infty}W^{2,p}(B(n)) :f\equiv0 \textrm{ on }\partial B(n), \mathcal A^kf\in C_b(B(n))\}$, $k=1,\ldots,m$, is a sectorial operator in $C_b(B(n))$ and $\bm B$ is a bounded linear operator in $(C_b(B(n)))$, from \cite[Proposition 2.4.1]{Lun10} we infer that $\bm A-\bm B=\bm\calA$ is a sectorial operator in $C_b(B(n))$ with domain $\{\f\in\bigcap_{p<\infty}W^{2,p}(B(n)):f_k\equiv0\textrm{ on }\partial B(n) \textrm{ for every }k=1,\ldots,m, \ \bm\calA\bm f\in C_b(B(n)))$.

Similar arguments show that the realization of the operator $\bm\calA^P$ in $C_b(B(n))$, with domain $\{\f\in \bigcap_{p<\infty}W^{2,p}(B(n)):f_k\equiv0\textrm{ on }\partial B(n) \textrm{ for every }k=1,\ldots,m, \ \bm\calA^P\bm f\in C_b(B(n))\}$ is sectorial. Finally, also the realization of the operator $\bm\calA^P_n$ ($n\in\N$) in $C_b(\R^d)$, with domain $\{\bm f\in \bigcap_{p<\infty}W^{2,p}_{\rm loc}(\R^d):\bm\calA^P_n\bm f\in C_b(\R^d))$ is sectorial. Here,
\begin{align}
\label{op_vett_tronc}
({\bm\calA}^P_n \f)_k
= & \vartheta_n{\rm div}(Q^k\nabla f_k)+\eta^0(1-\vartheta_n)\Delta f_k+ \vartheta_n\langle b^k,\nabla f_k\rangle -\vartheta_n(V^P\bm f)_k, \quad k=1,\ldots,m,
\end{align}
where $\eta^0=\min\{\eta^0_k:k=1,\ldots,m\}$ (see Hypothesis \ref{hyp-base}(ii)) and $\vartheta_n\in C^\infty_c(\R^d)$ satisfies the condition $\chi_{B(n)}\leq \vartheta_n\leq \chi_{B(2n)}$.}
\end{remark}

\begin{remark}
\label{rmk:limit_kernel}
{\rm As a consequence of Proposition \ref{prop:mono_nuclei}, it follows that for every $(t,x)\in(0,\infty)\times \R^d$ and every $h,k=1,\ldots,m$, there exist the limit
\begin{align*}
\lim_{n\to\infty} p_{hk}^{\mathcal D,P,n}(t,x,y)=:r^P_{hk}(t,x,y)\in[0,\infty)\cup\{\infty\},    \qquad y\in \R^d,
\end{align*}
and a subset $N_{t,x}\subseteq \R^d$ of null Lebesgue measure such that $r^P_{hk}(t,x,y)<\infty$ for every $y\in \R^d\setminus N_{t,x}$. Further, from \eqref{int_form_smgr_2} and \eqref{conv_smgr_Dirich} we infer that, for every $(t,x)\in(0,\infty)\times \R^d$ and every $h,k=1,\ldots,m$, we get $r^{P}_{hk}(t,x,y)=p^P_{hk}(t,x,y)$ for a.e. $y\in\R^d$.

Hereafter, for every $(t,x)\in(0,\infty)\times \R^d$ and every $h,k\in\{1,\ldots,m\}$, we always consider the version of $p_{hk}^P(t,x,\cdot)$ given by $r^P_{hk}(t,x,\cdot)$.}
\end{remark}

In the following proposition, the complement of a subset of $\{1,\ldots,m\}$ is always meant with respect to the universe set $\{1,\ldots,m\}$. Moreover, for every $k\in\{1,\ldots,m\}$ and $i\in\N$, we introduce the sets
\begin{align*}
H_k^0
= & \left\{h\in\{1,\ldots,m\}\setminus\{k\}:v_{hk}\not\equiv 0 \textrm{ on }\R^d\right\}; \\   H_k^i= & \bigg\{h\in\{1,\ldots,m\}\setminus\bigg(\{k\}\cup\bigcup_{r=0}^{i-1}H_k^r\bigg):\exists \ell\in H_k^{i-1} \textrm{ such that }v_{h\ell}\not\equiv 0 \textrm{ on }\R^d\bigg\}.
\end{align*}
Since the cardinality of $H_k^{i}$ is strictly less than $m-i$, it follows immediately that if $i> m-1$ then $H_k^i=\varnothing$. Finally, we set $F_k=\{k\}\cup\bigcup_{i=0}^{m-1}H^i_k$.

\begin{proposition}
\label{prop:gen_eq_meas}
For every $h,k=1,\ldots,m$, the measure $p_{hk}^P(t,x,y)dy$ is either equivalent to the Lebesgue measure for every $(t,x)\in (0,\infty)\times\R^d$ or it is the null measure for every $(t,x)\in(0,\infty)\times \R^d$. This second case occurs if and only if $F_k\neq \{1,\ldots,m\}$ and $h\in F_k^c$.
% and only if there exists a nontrivial subset $E\subset \{1, \ldots, m\}$ such that $h\in E$, $k\in E^c$ and $v_{rs}$ identically vanishes on $\R^d$ for every $r\in E$ and $s\notin E$.
\end{proposition}

\begin{proof}
We fix $k\in\{1,\ldots,m\}$ and split the proof into some steps.

{\em Step 1}. Here, we assume that there exists a nontrivial set $E\subset \{1,\ldots,m\}$ such that $h\in E$, $k\in E^c$ and $v_{rs}\equiv0$ on $\R^d$ for every $r\in E$ and $s\in E^c$, and prove that $p_{hk}(t,x,y)dy$ is the null measure for every $(t,x)\in (0,\infty)\times\R^d$.

We denote by $i_1,\ldots,i_{\overline \ell}$ the elements of $E$, ordered from the smallest to the largest, fix a function $f\in C_b(\R^d)$ and set $\f=f\bm{e}_k$. Since $v_{i_js}=0$ on $\R^d$ for every $j=1,\ldots,\overline \ell$ and every $s\not\in E$, the function $\uu:\R^d\to\R^{\overline \ell}$, where $u_j=(\T^P(\cdot)\f)_{i_j}$ for every $j=1,\ldots,\overline \ell$,  is bounded in each strip $[0,T]\times\R^d$, belongs to $C([0,\infty)\times \R^d)\cap C^{1+\frac{\alpha}{2},2+\alpha}_{\rm loc}((0,\infty)\times \R^d)$ and solves the Cauchy problem
\begin{align}
\left\{
\begin{array}{ll}
D_t\bm{v}=\bm\calB\bm v,  &\textrm{ in }(0,\infty)\times \R^d,\\[1mm]
\bm v(0,\cdot)=\bm{0}, & \textrm{ on }\R^d,
\end{array}
\right.
\label{cauchy_prob_null_in_datum}
\end{align}
where $\bm\calB$ is the operator defined on smooth functions $\f:\R^d\to\R^{\overline \ell}$ as
\begin{align*}
(\bm\calB\f)_{i_j}=\mathrm{div} (Q^{i_j}\nabla  f_{i_j}) +\langle b^{i_j},\nabla f_{i_j}\rangle -(V^P{\f})_{i_j}, \qquad j=1,\ldots,\overline \ell.    
\end{align*}
Since $\bm\calB$ satisfies Hypotheses \ref{hyp-base}, from
\cite[Theorem 2.4]{AddLor23} we infer that problem \eqref{cauchy_prob_null_in_datum} admits a unique solution which is bounded in each strip $[0,T]\times\R^d$ and belongs to $C([0,\infty)\times \R^d)\cap C^{1+\frac{\alpha}{2},2+\alpha}_{\rm loc}((0,\infty)\times \R^d)$. Therefore, $\uu\equiv \bm 0$ or, equivalently,
\begin{align*}
0=(\T^P(t)\f)_{i_j}(x)=\int_{\R^d}f(y)p_{i_j k}(t,x,y)dy, \qquad (t,x)\in (0,\infty)\times \R^d, \ j=1,\ldots,\overline \ell.
\end{align*}
Due to arbitrariness of $f\in C_b(\R^d)$ and recalling that $h=i_j$ for some $j\in\{1,\ldots,\overline \ell\}$, we deduce that $p_{hk}(t,x,y)dy=0$ for every $(t,x)\in(0,\infty)\times \R^d$. 

{\em Step 2}. In this step we prove that, if $\bm0\leq\f\in C_b(\R^d)$ and $f_k$ does not identically vanish on $\R^d$, then $(\T^P(t)\f)_h$ is strictly positive in $\R^d$ for every $t>0$ and $h\in F_k$. 
For this purpose, we fix $n\in\N$. From the variation-of-constants formula, we infer that
\begin{align}
\label{var_cons_form_diri}
(\T^{\mathcal D,P,n}(t)\f)_\ell
= & T^{\mathcal D,n}_\ell(t)f_\ell-\sum_{r=1, r\neq \ell}^m\int_0^tT^{\mathcal D,n}_\ell(t-s)(v_{\ell r}^P(\T^{\mathcal D,P,n}(s)\f)_r)ds 
\end{align}
for every $t>0$, $\ell=1,\ldots,m$, and $n\in\N$, where, for every $n\in\N$, $(T^{\mathcal D,n}_\ell(t))$ is the scalar semigroup associated to the realization of the elliptic operator $\calA_{\ell}:=\mathrm{div} (Q^\ell\nabla ) -\langle b^\ell,\nabla \rangle -v_{\ell\ell}$ in $C_b(B(n))$ with homogeneous Dirichlet boundary conditions. Since $(T^{\mathcal D,n}_\ell(t))$ is irreducible (see \cite[Proposition 1.2.13]{book}), $v_{\ell r}^P\leq0$, for every $\ell=1,\ldots,m$ and $r\in\{1,\ldots,m\}\setminus\{\ell\}$, and $\T^{\mathcal D,P,n}(\cdot)\f\ge \bm{0}$ in $(0,\infty)\times\R^d$, from \eqref{var_cons_form_diri} it follows that $(\T^{\mathcal D,P,n}(t)\f)_\ell\ge T^{\mathcal D,n}_\ell(t)f_\ell>0$ on $B(n)$ for every $t>0$, if $f_\ell$ does not identically vanish on $\R^d$. Recalling that $(T^{\mathcal D,n}_\ell(t)f_\ell)$ is a pointwise increasing and convergent sequence in $\R^d$ (see for instance \cite[Remark 1.2.2]{book}), it follows easily that $(\T^P(t)\f)_\ell$ is positive in $\R^d$ for every $t>0$ if $f_\ell$ does not identically vanish on $\R^d$. In particular, $(\T^{P}(t)\f)_k>0$ on $\R^d$ for every $t>0$. Hence, if $f_h$ does not identically vanish on $\R^d$, then we are done. If $f_h\equiv 0$ on $\R^d$, then from \eqref{var_cons_form_diri} we deduce that
\begin{align*}
%\label{var_cont_form_2}
(\T^{\mathcal D,P,n}(t)\f)_h= -\sum_{r=1, r\neq h}^m\int_0^tT^{\mathcal D,n}_h(t-s)(v_{hr}^P(\T^{\mathcal D,P,n}(s)\f)_r)ds, \qquad t>0. 
\end{align*}
Let $\overline i\in\{0,\ldots,m-1\}$ be the unique index such that $h\in H^{\overline i}_k$. 
We set $j_{\overline i}=h$ and, if $\overline i\neq 0$, then for every $i=0,\ldots,\overline i-1$, we denote by $j_i\in H^i_k$ any index such that $v_{j_{i+1}j_i}\neq0$ on $\R^d$ (the existence of such an index is guaranteed by the definition of $H^{i+1}_k$). From  \eqref{var_cons_form_diri} with $\ell=j_0$ and recalling that $f_{j_0}\geq 0$ and that $v_{rs}^P=-|v_{rs}|$ for every $r,s=1,\ldots,m$ with $r\neq s$, we infer that
\begin{align*}
(\T^{\mathcal D,P,n}(t)\f)_{j_0}
\geq & \int_0^tT^{\mathcal D,n}_{j_0}(t-s)(|v_{j_0k}|(\T^{\mathcal D,P,n}(s)\f)_k)ds,\qquad\;\,t\in (0,\infty).    
\end{align*}
Since $v_{j_0k}\neq0 $ and $(\T^{\mathcal D,P,n}(t)\f)_k>0$, from the irreducibility of $(T^{\mathcal D,n}_{j_0}(t))$  we deduce that $(\T^{\mathcal D,P,n}(\cdot)\f)_{j_0}>0$ in $(0,\infty)\times B(n)$. Letting $n$ tend to $\infty$, from Proposition \ref{prop:mono_nuclei} we infer that $(\T^P(\cdot)\f)_{j_0}>0$ in $(0,\infty)\times\R^d$. 

If $\overline i=0$ then we are done. Otherwise, we observe that 
\begin{align*}
(\T^{\mathcal D,P,n}(t)\f)_{j_1}
\geq & \int_0^tT^{\mathcal D,n}_{j_1}(t-s)(|v_{j_1j_0}|(\T^{\mathcal D,P,n}(s)\f)_{j_0})ds
, \qquad t\in (0,\infty),   
\end{align*}
so that also $(\T^{\mathcal D,P,n}(t)\f)_{j_1}>0
$ on $B(n)$. Iterating this argument, in a finite number of steps we show that $(\T^{\mathcal D,P,n}(t)\f)_{h}=(\T^{\mathcal D,P,n}(t)\f)_{j_{\overline i}}>0$ in $B(n)$. Finally, letting $n$ tend to $\infty$ and taking Proposition \ref{prop:mono_nuclei} into account,  we conclude that $(\T^P(t)\f)_h>0$ in $\R^d$.

{\em Step 3}.
Now, we prove that, for every $h\in F_k$, the measure $p_{hk}^P(t,x,y)dy$ is equivalent to the Lebesgue measure $\mathcal L_d$  on $\R^d$ for every $(t,x)\in (0,\infty)\times\R^d$. In view of Proposition \ref{prop:prop_semigruppi}, it is enough to show that, if $p^P_{hk}(t,x,A)=0$, then $\mathcal L_d(A)=0$. Assume by contradiction that $\mathcal L_d(A)>0$ and notice that, from \eqref{var_cons_form_diri}, Proposition \ref{prop:prop_semigruppi} and the strong Feller and irreducibility properties of $(T^{\mathcal D,n}_k(t))$, it follows that
$(\T^{\mathcal D,P,n}(t)(\chi_A{\bm e}_k))_k\geq T_k^{\mathcal D,n}(t)\chi_A>0$ on $\R^d$ for every $t>0$ and every $n\in\N$. Letting $n$ tend  to infinity, Propositions \ref{prop:prop_semigruppi} and \ref{prop:mono_nuclei} (which implies that the sequence
$(\T^{\mathcal D,P,n}(t)(\chi_A{\bm e}_k))_k$ is increasing)
show that
$(\T^{P}(t)(\chi_A{\bm e}_k))_k>0$ on $\R^d$ for every $t>0$. Since $\T^{P}(t)(\chi_A\e_k)=\T^{P}(t-\varepsilon)(\T^{P}(\varepsilon)(\chi_A\e_k))$, $\T^{P}(\varepsilon)(\chi_A\e_k)$ belongs to $C_b(\R^d)$ (see Proposition \ref{prop:prop_semigruppi}) and it is nonnegative in $\R^d$ for every $\varepsilon\in(0,t)$, from Step 2 it follows that $p_{hk}^P(t,x,A)=(\T^{P}(t)(\chi_A\e_k))_h>\bm 0$ on $\R^d$ for every $h\in F_k$. This yields a contradiction. Hence, $\mathcal L_d(A)=0$. 

{\em Step 4}. It remains to show that, if $h\in F_k^c$ then $p_{hk}^P(t,x,y)dy$ is the null measure for every $(t,x)\in (0,\infty)\times\R^d$. For this purpose, we observe that, for every $r\in F_k^c$, it follows that $r\neq k$ and $r\notin H_k^i$ for every $i\in\{0,\ldots,m-1\}$. This implies that $v_{rs}\equiv 0$ whenever $s=k$ or $s\in H_k^i$, for every $i\in\{0,\ldots,m-1\}$, i.e., whenever $s\in F_k$. From Step 1, with $E=F_k^c$, it follows that, for every $h\in F_k^c$,  $p^P_{hk}(t,x,y)dy$ is the trivial measure for every $(t,x)\in (0,\infty)\times\R^d$. 

The proof is complete.
\end{proof}
\begin{remark}
{\rm 
Fix $k\in\{1,\ldots,m\}$ and suppose that $E$ is a nontrivial subset of $\{1, \ldots, m\}$ such that $k\in E^c$ and $v_{rs}$ identically vanishes on $\R^d$ for every $r\in E$ and $s\notin E$. Then, $E\subseteq F_k^c$.
In particular, if this latter set is empty then no set $E$ with the above properties exists.
Indeed, suppose that $E$ is a set with the above properties. If $h\in E$ then $h\neq k$ from the definition of $E$. Further, $h\notin H^0_k$ since $v_{hk}\equiv 0$ on $\R^d$. We notice that $h\notin H^1_k$. Indeed, suppose by contradiction that $h\in H^1_k$. Then, there exists an index $j_1\in H^0_k$ such that $v_{j_1k}\not\equiv0$ and $v_{hj_1}\not\equiv0$ on $\R^d$. This implies that $j_1\notin E$, since $v_{j_1k}\not\equiv 0$ on $\R^d$, and $j_1\notin E^c$, since $v_{hj_1}\not\equiv0$ on $\R^d$. Iterating this argument, we conclude that $h\notin H^i_k$ for every $i\in\{0,\ldots,m-1\}$.}     
\end{remark}

We denote by $({\bm T}_n^P(t))$ the analytic semigroup in $C_b(\R^d)$ associated to the operator ${\bm\calA}^P_n$, defined in \eqref{op_vett_tronc}.
\begin{proposition}
\label{prop:conv_SMGR}
For every ${\bm{f}}\in C_b(\R^d)$ and every compact set $E\subseteq (0,\infty)\times \R^d$, the sequence $(\bm{T}^P_n(\cdot)\f)$ converges to $\bm{T}^P(\cdot)\f$ in $C^{1,2}(E)$. If $\f\in C_c(\R^d)$, then we can also consider compact sets $E\subseteq [0,\infty)\times \R^d$.   
\end{proposition}

\begin{proof}
Fix $\f\in C_b(\R^d)$. From interior Schauder estimates (see \cite[Theorem 7.2]{AngLor20}), for every compact set $E\subseteq(0,\infty)\times \R^d$ there exists a positive constant $K$, which depends on $E$, such that
\begin{align*}
\|\bm{T}^P_n(\cdot)\f\|_{C^{1+\frac{\alpha}{2},2+\alpha}(E)}\leq K\|\bm{T}^P_n(\cdot)\f\|_\infty\leq \widetilde K\max_{k=1,\ldots,m}\|f_k\|_\infty,  \qquad n\in\N,   
\end{align*}
where $\widetilde K=\sqrt{m}K\max\{e^{-MT},1\}$ and $T>0$ satisfies $E\subseteq (0,T]\times \R^d$. From Ascoli-Arzel\`a theorem, it follows that there exist a subsequence $(n_k)$ and a function ${\bm u}\in C^{1+\frac{\alpha}{2},2+\alpha}_{\rm loc}((0,\infty)\times \R^d)$ such that $({\bm T}^P_{n_k}(\cdot)\f)$ converges to $\bm u$ in $C^{1,2}(E)$ for every compact set $E\subseteq(0,\infty)\times \R^d$. Further, for every $T>0$ we have $\|\uu(t,\cdot)\|_\infty\leq \sqrt m\max\{e^{-Mt},1\}\max_{k=1,\ldots,m}\|f_k\|_\infty$ for every $t\in(0,T]$ and $D_t\uu=\bm\calA^P\uu$ on $(0,\infty)\times \R^d$, since $D_t\bm{T}^P_{n}(\cdot)\f=\bm\calA^P\bm{T}^P_n(\cdot)\f$ on $E$ if $n$ is large enough. 

To prove that $\uu={\bm T}^P(\cdot)\f$, it remains to show that $\uu$ can be continuously extended up to $t=0$ setting $\uu(0,\cdot)=\f$. Indeed, as we have already recalled, problem \eqref{prob_cauchy_1}(ii) admits a unique solution $\uu\in C([0,\infty)\times\R^d)\cap C^{1,2}((0,\infty)\times\R^d)$, which is bounded in every strip $[0,T]\times\R^d$. 

At first, we consider the case $\f\in C^{2+\alpha}_c(\R^d)$. From the Schauder estimates in Theorem \ref{teo-Schauder} we infer that, for every $R,T>0$, 
\begin{align*}
\|\bm{T}^P_n(\cdot)\f\|_{C^{1+\frac{\alpha}{2},2+\alpha}([0,T]\times B(R))}\leq K_{R,T}\|\f\|_{C^{2+\alpha}_b(\R^d)}, \qquad n\in\N,
\end{align*}
where $K_{R,T}$ is a positive constant which depends on $R$ and $T$ but is independent on $n\in\N$. Again Ascoli-Arzel\`a theorem implies that $\uu\in C^{1+\frac{\alpha}{2},2+\alpha}([0,T]\times B(R))$ for every $R>0$ and, consequently, $\uu={\bm T}^P(\cdot)\f$ and
$(\T_{n_k}^P(\cdot)\f)$ converges to  $\T^P(\cdot)\f$ in $C^{1,2}([0,T]\times B(R))$ for every $R,T>0$.
In particular, we have proved that
any subsequence of $(\T_n^P(t)\f)$, which converges in $C^{1,2}(E)$ for every compact set $E\subseteq (0,\infty)\times\R^d$, actually converges to $\T^P(\cdot)\f$, so that all the sequence
$(\T_n^P(t)\f)$ converges to $\T^P(\cdot)\f$ as $n$ tends to $\infty$ in $C^{1,2}(E)$ for every compact set $E\subseteq [0,\infty)\times \R^d$.

If $\f\in C_0(\R^d)$, then there exists a sequence $(\f_j)\subseteq C_c^{2+\alpha}(\R^d)$ which converges to $\f$ in $C_b(\R^d)$. Since, for every $T>0$ and every $h,k\in\N$,
\begin{eqnarray*}
\|\T_{n_k}(\cdot)\f_h-\T_{n_k}(\cdot)\f\|_{C_b([0,T]\times\R^d)}\le \sqrt m\max\{e^{-MT},1\}
\|\f_h-\f\|_{\infty},
\end{eqnarray*}
letting $k$ tend to $\infty$, we conclude that
\begin{eqnarray*}
\|\T(\cdot)\f_h-\uu\|_{C_b((0,T]\times\R^d)}\le \sqrt m\max\{e^{-MT},1\}
\|\f_h-\f\|_{\infty},
\end{eqnarray*}
so that $(\T(\cdot)\f_h)$ converges to $\uu$, uniformly in $[0,T]\times\R^d$ as $h$ tends to $\infty$. This shows that $\uu$ can be extended by continuity up to $t=0$, by setting $\uu(0,\cdot)=\f$. In particular, $\uu=\T^P(\cdot)\f$ and the whole sequence $(\T_n(\cdot)\f)$ converges to $\T^P(\cdot)\f$.

It remains to deal with the general case when $\f\in C_b(\R^d)$. We fix $R,T>0$, a function $\varphi\in C_c^\infty(\R^d)$ such that $\chi_{B(R)}\le\varphi\le\chi_{B(2R)}$, and split
\begin{align}
\label{split_TPmf}
{\bm T}^P_n(\cdot)\f={\bm T}^P_n(\cdot)(\varphi\f)+{\bm T}^P_n(\cdot)((1-\varphi)\f), \qquad n\in\N.   
\end{align}
We introduce the function 
$\bm{w}_1:={\bm T}^P_n(\cdot)((1-\varphi)\f)-\|\f\|_{\infty}{\bm T}^P_n(\cdot)((1-\varphi)\mathds{1})$,
which satisfies the conditions
$D_t\bm{w}_1-\bm\calA^P\bm{w}_1=\bm0$ in $(0,T]\times \R^d$ and $\bm{w}_1(0,\cdot)\leq \bm 0$ in $\R^d$.    
The maximum principle in \cite[Theorem 2.3]{AngLor20}
(which is inspired by \cite[Chapter 3, Section 8]{ProWin67}), 
whose proof works whenever the off-diagonal entries of the matrix $V^P$ are nonpositive, even if the sum of the elements of each row of $V^P$ is not nonnegative,\footnote{Notice that, in our operator, the sign of the potential term is the opposite of the sign of the potential term in the quoted paper.}
implies that $\bm{w}_1\leq \bm0$ on $[0,T]\times \R^d$.
Similar arguments applied to the function $\bm{w}_2=-{\bm T}^P_n(\cdot)((1-\varphi)\f)-\|\f\|_{\infty}{\bm T}^P_n(\cdot)((1-\varphi)\mathds{1})$, 
give $\bm{w}_2\leq\bm0$ on $[0,T]\times \R^d$. Hence, we infer that
\begin{align}
\label{stima_TP1-fif}
|({\bm T}^P_n(\cdot)((1-\varphi)\f))_h|
\le\|\f\|_{\infty}({\bm T}^P_n(\cdot)((1-\varphi)\mathds{1}))_h, \qquad h=1,\ldots,m, \ n\in\N,
\end{align}
in $[0,T]\times \R^d$. From \eqref{split_TPmf} and \eqref{stima_TP1-fif}, it follows that
\begin{align}
|(\bm{T}^P_n(t)\f)_h-f_h|
\leq & |(\bm{T}^P_n(t)(\varphi\f))_h-f_h|
+\|\f\|_{\infty}({\bm T}^P_n(t)((1-\varphi)\mathds{1}))_h \notag \\
= & |(\bm{T}^P_n(t)(\varphi\f))_h-f_h|
+\|\f\|_{\infty}({\bm T}^P_n(t)\mathds{1})_h-\|\f\|_{\infty}({\bm T}^P_n(t)(\varphi\mathds{1}))_h
\label{stima_cont_TPmf}
\end{align}
for every $t\in[0,T]$, $h=1,\ldots,m$ and $n\in\N$. Let us apply once again the maximum principle in \cite[Theorem 2.3]{AngLor20},  with $\bm\calA^P$ replaced by $\bm\calA^P_n+\widetilde MI$, where $\widetilde M=\min\{0,M\}$, to the function $\bm{w}_3$, defined as
$\bm{w}_3(t,\cdot):=e^{\widetilde Mt}{\bm T}^P_n(t)\mathds{1}-\mathds{1}$ for every $t\geq0$.
Since, from \eqref{segno_V_P}, 
\begin{align*}
(D_t\bm{w}_3-(\bm\calA^P_n+\widetilde MI)\bm{w_3})_h
=-\vartheta_n\sum_{k=1}^m v_{hk}^P+\widetilde M\leq\bm 0 \ \  \textrm{on }(0,T]\times \R^d, \qquad \bm{w_3}(0,\cdot)=\bm0 \ \ \textrm{on }\R^d,
\end{align*}
we infer that $\bm{w_3}\leq\bm0$ on $[0,T]\times \R^d$. Replacing in \eqref{stima_cont_TPmf}, it follows that
\begin{align}
\label{stima_cont_TPf_2}
|(\bm{T}^P_n(t)\f)_h-f_h|
\leq |(\bm{T}^P_n(t)(\varphi\f))_h-f_h|
+\|\f\|_{\infty}e^{-\widetilde Mt}-\|\f\|_{\infty}({\bm T}^P_n(t)(\varphi\mathds{1}))_h   
\end{align}
for every $t\in[0,T]$, $h=1,\ldots,m$ and $n\in\N$. Recalling that $\varphi\f,\varphi\mathds1\in C_0(\R^d)$, replacing $n$ with $n_k$ in \eqref{stima_cont_TPf_2} and letting $k$ tend to infinity, we deduce that
\begin{align*}
|u_h(t,\cdot)-f_h|
\leq |(\bm{T}^P(t)(\varphi\f))_h-f_h|
+\|\f\|_{\infty}e^{-\widetilde Mt}-\|\f\|_{\infty}({\bm T}^P(t)(\varphi\mathds{1}))_h,\qquad\;\,t\in (0,T]
\end{align*}
for every $h=1,\ldots,m$. Since $\varphi\f=\f$ and $\varphi\mathds1=\mathds1$ on $B(R)$, the continuity of the functions $\bm{T}^P(\cdot)(\varphi\f)$ and $ \bm{T}^P(\cdot)(\varphi\mathds1)$ in $[0,\infty)\times \R^d$ implies that $\uu$ can be continuously extended up to $t=0$ setting $\uu(0,\cdot)=\f$. This concludes the proof.  
\end{proof}

\begin{lemma}
%\label{lmma:ug_Q_com}
Let $0\leq a<b$ and $h,k\in\{1,\ldots,m\}$. Then,
\begin{align}
&\int_{R(a,b)} (D_t \varphi(t,y) + \calA_k \varphi (t,y))p_{hk}(t,x,y)dtdy
-\sum_{j=1,j\neq k}^m\int_{R(a,b)}v_{jk}(y)\varphi(t,y)p_{hj}(t,x,y)dtdy \notag \\
=& \int_{\R^d} (p_{hk}(b,x,y)\varphi(b,y)-p_{hk}(a,x,y)\varphi(a,y))dy
\label{int_parti_nuclei_com_2}
\end{align}
for every $(t,x)\in (a,b)\times \R^d$,  $h,k=1,\dots, m$, $\varphi\in C^{1,2}_c(\overline{R(a,b)})$.
Finally, the same holds true if we replace $\bm\calA$ with $\bm\calA^P$ and $p_{hk}$ with $p_{hk}^P$ for every $h,k\in\{1,\ldots,m\}$.
\end{lemma}
\begin{proof}
Let $\varphi\in C_c^{1,2}(\overline {R(a,b)})$, fix $k\in\{1,\ldots,m\}$ and observe that the function $\bm\varphi=\varphi{\bm e}_k$ belongs to
the domain of the realization of $\bm\calA$ in $C_b(B(n))$ with homogeneous Dirichlet boundary conditions for every 
$n>R$, where $R>0$ satisfies ${\rm supp}(\varphi(t,\cdot))\subseteq B(R)$ for every $t\in[a,b]$. Since $(\T^{\mathcal D,n}(t))$ is an analytic semigroup (see Remark \ref{rmk:analiticita}), it follows that 
\begin{align*}
D_t(\T^{\mathcal D,n}(t)\bm\varphi(t,\cdot))
= &\T^{\mathcal D,n}(t)(\bm\calA\bm\varphi(t,\cdot))+\T^{\mathcal D,n}(t)D_t\bm\varphi(t,\cdot)
\end{align*}
on $B(n)$ for every $n>R$ and $t\geq0$. From \eqref{conv_smgr_Dirich}, letting $n$ tend to infinity we infer that the function $t\mapsto (\T(t)\bm{\varphi}(t,\cdot))(x)$ is differentiable in $(0,\infty)$ for every $x\in\R^d$ and
\begin{align*}
D_t(\T(t)\bm\varphi(t,\cdot))
= \T(t)(\bm\calA\bm\varphi(t,\cdot))+\T(t)D_t\bm\varphi(t,\cdot)
\end{align*}
in $\R^d$ for every $t\in (0,\infty)$. In particular, the function $(t,x)\mapsto D_t(\T(t)\bm{\varphi}(t,\cdot))(x)$ is continuous in $[0,\infty)\times\R^d$ and
\begin{align*}
(D_t{\bm T}(\cdot)(\varphi{\bm e}_k))_h
= & \int_{\R^d}\bigg (D_t\varphi(\cdot,y)p_{hk}(\cdot,\cdot,y) + \sum_{j=1}^m(\bm{\calA}(\varphi {\bm e}_k))_j(\cdot,y)p_{hj}(\cdot,\cdot,y)\bigg )dy
\end{align*}
in $(a,b)\times\R^d$, for every $h=1,\ldots,m$. Integrating both the sides of the above equality with respect to $t$ between $a$ and $b$ yields the assertion.
\end{proof}

\begin{remark}
\label{rmk_kernels}
{\rm The previous proof shows that formula \eqref{int_parti_nuclei_com_2} is satisfied also by the kernels $p_{hk}^{\mathcal{D},n}$ and $p_{hk}^{\mathcal{D},P,n}$ with $R(a,b)$ being replaced by $R_n(a,b)$.}
\end{remark}

Now, we prove the following regularity result for the kernels $p_{hk}$ and $p_{hk}^P$, with $h,k\in\{1,\ldots,m\}$. 

\begin{proposition}
\label{prop 2.3.1 tesi}
Fix $x\in\R^d$ and $0\leq a_0<b_0\leq T$. Suppose that $p_{hk}(\cdot,x,\cdot)$ $($resp. $p_{hk}^P(\cdot,x,\cdot))$ $\in L^r_{\rm loc}(R(a_0,b_0))$ for some $r\in(1,\infty)\cup\{\infty\}$, some $h=1,\ldots,m$ and every $k=1,\ldots,m$. Then, for every $a_0<a<b<b_0$, every $x\in\R^d$ and every $n\in\N$, the function $p_{hk}(\cdot,x,\cdot)$ $($resp. $p_{hk}^P(\cdot,x,\cdot))$ belongs to $\mathcal H^{r,1}(R_n(a,b))$, if $r\in\R$, and 
to $\mathcal H^{q,1}(R_n(a,b))$ for every $q\in (1,\infty)$, if $r=\infty$, for every $k=1,\ldots,m$. In particular, if $r>d+2$, then $p_{hk}(\cdot,x,\cdot)$ $($resp. $p_{hk}^P(\cdot,x,\cdot))$ belongs to $C([a,b];C(\overline{B(n)}))$ for every $k=1,\ldots,m$ and every $n\in\N$.
\end{proposition}

\begin{proof}
The proof of the smoothness of the functions $p_{hk}(\cdot,x,\cdot)$ and
$p_{hk}^P(\cdot,x,\cdot)$ can be obtained by arguing as in the proof of \cite[Lemma 3.1]{MetPalRha10}. We limit ourselves to showing that, in our situation, we gain the estimate which is crucial in the proof of the quoted result. For simplicity, we consider the case $r\in (1,\infty)$, since the case $r=\infty$ can be proved by means of similar arguments by replacing $r$ with $q\in(1,\infty)$.

We fix $n\in\N$, $x\in\R^d$ and prove  that the kernel $p_{hk}(\cdot,x,\cdot)$ belongs to ${\mathcal H}^{r,1}(R_n(a,b))$ for every $k=1,\ldots,m$. Let $\vartheta\in C_c^\infty(\R)$ be such that $\chi_{[a,b]}\le\vartheta\le\chi_{[a_0,b_0]}$ and let $\vartheta_n\in C^\infty_c(\R^d)$ be such that $\chi_{B(n)}\leq \vartheta_n\leq \chi_{B(2n)}$ for every $n\in\N$. Further, we set $\mathfrak p_{hk}^n(t,x,y):=\vartheta_n(y)p_{hk}(t,x,y)$ for every $t\in(0,T)$, $n\in\N$ and $y\in \R^d$.

Applying \eqref{int_parti_nuclei_com_2} to the function $\vartheta\vartheta_n\varphi$, with $\varphi\in C^{\infty}_c(\overline{R(a_0,b_0)})$, we deduce that
\begin{align*}
& \int_{R(0,T)} p_{hk}(t,x,y)(\vartheta_n(y)D_t(\vartheta\varphi(\cdot,y))(t)+\vartheta(t)(\calA_k(\vartheta_n\varphi(t,\cdot)))(y)dtdy \notag \\
& -\sum_{j=1,j\neq k}^m\int_{R(0,T)}v_{jk}(y)\vartheta(t)\vartheta_n(y)\varphi(t,y)p_{hj}(t,x,y)dtdy=0, 
\end{align*}
which gives
\begin{align*}
& \int_{R(0,T)}\vartheta(t)\mathfrak{p}_{hk}^n(t,x,y)\big (D_t\varphi(t,y)+\calA^k_1\varphi(t,y)\big )dtdy \\
=& -  \int_{R(0,T)}\vartheta(t)\mathfrak{p}_{hk}^n(t,x,y)\big(\langle G^k(y)+ b^k(y),\nabla \varphi(t,y)\rangle
-v_{kk}(y)\varphi(t,y)\big)dtdy \\
& -\int_{R(0,T)}\vartheta'(t)\varphi(t,y)\mathfrak{p}_{hk}^{n}(t,x,y)dtdy \\
& - \int_{R(0,T)}\vartheta(t)p_{hk}(t,x,y)\big(2\langle Q^k(y)\nabla\varphi(t,y),\nabla \vartheta_n(y)\rangle+\varphi(t,y)\calA_1^k\vartheta_n(t,y) \\
&%\qquad\qquad\qquad\qquad
\ \ \ \qquad\quad\,+\varphi(t,y)\langle G^k(y)+b^k(y),\nabla \vartheta_n(t,y)\rangle\big)dtdy \\
&+\sum_{j=1,j\neq k}^m\int_{R(0,T)}v_{jk}(y)\mathfrak{p}_{hj}^{n}(t,x,y)\vartheta(t)\varphi(t,y)dtdy,
\end{align*}
where
$\calA^k_1:=\sum_{i,j=1}^dq_{ij}^kD^2_{ij}$ and $(G^k)_j=\sum_{i=1}^d D_iq_{ij}^k$ for every $j=1,\ldots,d$. From H\"older's inequality and the local boundedness of the coefficients of the operator $\bm\calA$, it follows that 
\begin{align*}
\bigg |\int_{R(0,T)}\!\vartheta \mathfrak{p}^n_{hk}(t,x,y)(D_t\varphi(t,y)\!+\!\calA^k_1\varphi(t,y))dtdy\bigg |   \leq c_n\sum_{j=1}^m\|p_{hj}(\cdot,x,\cdot)\|_{L^r(R_{2n}(a_0,b_0))}\|\varphi\|_{W^{0,1}_{r'}(R(0,T))} 
\end{align*}
for some positive constant $c_n$, independent of $\varphi$.

From now on, the proof follows the same lines as those of \cite[Lemma 3.1]{MetPalRha10} and allows us to prove that $\mathfrak p_{hk}^n(\cdot,x,\cdot)\in\mathcal H^{r,1}(R(a,b))$ for every $n\in\N$ and every $k=1,\ldots,m$. Since $\mathfrak p_{hk}^n(\cdot,x,\cdot)=p_{hk}(\cdot,x,\cdot)$ on $(0,\infty)\times B(n)$, we get the assertion. To prove the smoothness of $p_{hk}(\cdot,x,\cdot)$, it is enough to notice that, if $r>d+2$, then the same computations as in \cite [Proposition 3.3]{MetPalRha10} (see also \cite[Corollary 7.5]{Kry00}) give $\mathfrak p_{hk}^n(\cdot,x,\cdot)\in C^{\nu}([a,b];C^\theta_b(\R^d))$ for some $\nu,\theta>0$.

The result for the kernel $p_{hk}^P(\cdot,x,\cdot)$ can be proved arguing in the same way.
\end{proof}

\section{Integrability of Lyapunov functions}
\label{sec:lyapunov}

In this section we prove that Lyapunov functions for $\bm\calA^P$ and time-dependent Lyapunov functions for the operator $D_t+\bm\calA^P$ (see Definition \ref{def:time_dep_Lyap}) are integrable with respect to the measure $p_{hk}^P(t,x,y)dy$ for every $t\in(0,T]$, $x\in\R^d$ and $h,k=1,\ldots,m$.

\begin{definition}
\label{def:time_dep_Lyap}
Given $T>0$ and a function $\nu\in C^{1,2}((0,T)\times \R^d)\cap C([0,T]\times \R^d)$, we say that the $\R^m$-valued function ${\bm\nu}=(\nu,\ldots,\nu)$ is a time-dependent Lyapunov function for the operator ${\bm \calL}:=D_t+{\bm\calA}^P$ $($with respect to $\varphi$ and $g)$ if  $(i)$ $\lim_{|x|\to\infty}\nu(t,x)=\infty$, uniformly with respect to $t$ in compact subsets of $(0,T]$, $(ii)$ $\nu\leq \varphi$, and $(iii)$ there exists a function $g\in L^1(0,T)$ such that ${\bm\calL}{\bm \nu}(t,\cdot)\leq g(t)\bm{\nu}(t,\cdot)$ in $(0,T)\times \R^d$. 
\end{definition}

\begin{theorem}
\label{thm:int_time_lyap_fct}
If ${\bm\nu}$ is a time-dependent Lyapunov function for the operator ${\bm \calL}$ with respect to $\varphi$ and $g$, then the functions ${\bm T}^{\mathcal{D},P,n}(t){\bm\nu(t,\cdot)}$ $(n\in\N)$ and
${\bm T}^P(t){\bm\nu(t,\cdot)}$ are well-defined for every $t\in[0,T]$ and can be estimated from above by the function 
$t\mapsto e^{G(t)}\bm{\nu}(0,\cdot)$,
where $G(t)=\displaystyle\int_0^tg(s)ds$ for every $t\in [0,T]$.
\end{theorem}

\begin{proof}
We split the proof into three steps. In the first two steps we prove the assertion for the function $\T^P(\cdot)\bm\nu$, assuming first that the constant $M$ in \eqref{segno_V_P} is nonnegative (Step 1) and then addressing the general case (Step 2). Finally, in Step 3, we prove the assertion for the function $\T^{\mathcal{D},P,n}(\cdot)\bm\nu$.

{\it Step 1}. We fix $\sigma>0$ and, for every $\varepsilon\in(0,1)$, let $\psi_{\sigma,\varepsilon}\in C^{\infty}([0,\infty))$ be such that $\psi_{\sigma,\varepsilon}(t)=t$ for every $t\in [0,\sigma]$, $\psi_{\sigma,\varepsilon}=\sigma+\frac{1}{2}\varepsilon$ on $(\sigma+\varepsilon,\infty)$, $0\le\psi_{\sigma,\varepsilon}'\le 1$ and $\psi''_{\sigma,\varepsilon}\leq0$ in $[0,\infty)$. 
Further, for every $n\in\N$ we consider the operator ${\bm\calA}^P_n$, defined on smooth functions $\f$ by \eqref{op_vett_tronc}, where $\vartheta_n\in C^\infty_c(\R^d)$ satisfies $\chi_{B(n)}\leq \vartheta_n\leq \chi_{B(2n)}$.

We now fix $k\in\{1,\ldots,m\}$ and observe that the function ${\bm \psi}_{k,\sigma,\varepsilon}=(\psi_{\sigma,\varepsilon}\circ\nu){\bm e}_k$ belongs to $C^{1,2}_b((0,T)\times\R^d)$. In particular, for every $t\in (0,T)$, the function ${\bm \psi}_{k,\sigma,\varepsilon}(t,\cdot)$ belongs to the domain of $\bm\calA^P_n$, which generates the analytic semigroup
$({\bm T}^P_n(t))$ (see Remark \ref{rmk:analiticita}). Hence, the function 
${\bm T}^P_n(\cdot){\bm \psi}_{k,\sigma,\varepsilon}$ is differentiable with respect to time in $[0,T]\times\R^d$ and
$D_t{\bm T}^P_n(\cdot){\bm \psi}_{k,\sigma,\varepsilon}
={\bm T_n^P(\cdot)}{\bm \calA}^P_n{\bm \psi}_{k,\sigma,\varepsilon}
+{\bm T_n^P(\cdot)}D_t{\bm \psi}_{k,\sigma,\varepsilon}$
in $[0,T]\times\R^d$.

Note that $({\bm\calA}^P_n{\bm \psi}_{k,\sigma,\varepsilon})_{\ell}
= -\vartheta_nv_{\ell k}^P(\psi_{\sigma,\varepsilon}\circ\nu)$, if
$\ell\neq k$ and
\begin{align*}
&({\bm\calA}^P_n{\bm \psi}_{k,\sigma,\varepsilon})_k
= (\psi_{\sigma,\varepsilon}'\circ\nu)\calA^k_n\nu+\vartheta_nv_{kk}[\nu(\psi_{\sigma,\varepsilon}'\circ \nu)-\psi_{\sigma,\varepsilon}\circ\nu] \\ 
&\phantom{({\bm\calA}^P_n{\bm \psi}_{k,\sigma,\varepsilon})_k}\;\;\;\;\; +(\psi_{\sigma,\varepsilon}''\circ\nu)[ \vartheta_n\langle Q^k\nabla \nu,\nabla\nu\rangle  +\eta^0(1-\vartheta_n)|\nabla \nu|^2]\\
&\phantom{({\bm\calA}^P_n{\bm \psi}_{k,\sigma,\varepsilon})_k}\le (\psi_{\sigma,\varepsilon}'\circ\nu) \calA^k_n\nu-\vartheta_nv_{kk}(\psi_{\sigma,\varepsilon}\circ\nu)\chi_{\{\nu>\sigma+\varepsilon\}}.
\end{align*}
Here, $\calA^\ell_n$, $\ell=1,\ldots,m$, is the operator defined on smooth functions $u:\R^d\to\R$ by
\begin{align*}
(\calA^\ell_nu)
:=\vartheta_n{\rm div}(Q^\ell\nabla u)+\eta^0(1-\vartheta_n)\Delta u+\vartheta_n\langle b^k,\nabla u\rangle-\vartheta_nv_{\ell\ell}u,
\end{align*}
with $\eta^0=\min\{\eta_k^0:k=1,\ldots,m\}$ (see Hypothesis \ref{hyp-base}(ii)). Note that the last inequality in the previous chain of inequalities follows from the fact that $\psi_{\sigma,\varepsilon}''(\xi)\leq 0$ and $\xi\psi_{\sigma,\varepsilon}'(\xi)-\psi_{\sigma,\varepsilon}(\xi)\leq 0$ for every $\xi\geq0$, and $\psi_{\sigma,\varepsilon}'=0$ on $(\sigma+\varepsilon,\infty)$.

Therefore, for every $h\in\{1,\ldots,m\}$ we can estimate
\begin{align}
D_t({\bm T}^P_n(\cdot){\bm \psi}_{k,\sigma,\varepsilon})_h
\le &-\sum_{\ell=1,\ \ell\neq k}^m\int_{\R^d}\vartheta_nv^P_{\ell k}(y)
\psi_{\sigma,\varepsilon}(\nu(\cdot,y))p_{h\ell}^{P,n}(\cdot,\cdot,y)dy\notag\\
& +\int_{\R^d}\psi_{\sigma,\varepsilon}'(\nu(\cdot,y))(\calA^k_n\nu)(y)p_{hk}^{P,n}(\cdot,\cdot,y)dy\notag\\
&-\int_{\{\nu>\sigma+\varepsilon\}}\vartheta_n(y)v_{kk}(y)\psi_{\sigma,\varepsilon}(\nu(\cdot,y))p_{hk}^{P,n}(\cdot,\cdot,y)dy\notag\\
& +\int_{\{\nu\leq \sigma+\varepsilon\}} \psi'_{\sigma,\varepsilon}(\nu(\cdot,y))D_t\nu(\cdot,y)p_{hk}^{P,n}(\cdot,\cdot,y)dy.
\label{caldo}
\end{align}
where $\{p_{hk}^{P,n}:h,k=1,\ldots,m\}$ is the family of kernels associated to $(\T^{P}_n(t))$.

Note that ${\bm\psi}_{k,\sigma,\varepsilon}$ converges to $(\nu\wedge\sigma) {\bm e}_k$ in a dominated way in $\R^d$ as $\varepsilon$ tends to $0$, so that $\bm{T}^P_n(\cdot){\bm\psi}_{k,\sigma,\varepsilon}$ converges 
pointwise in $(0,T)\times\R^d$ to $\bm{T}^P_n(\cdot)((\nu\wedge\sigma) {\bm e}_k)$ as $\varepsilon$ tends to $0$. Moreover,
the interior Schauder estimates in \cite{AngLor20} show that the family of functions $\{\bm{T}^P_n(\cdot){\bm\psi}_{k,\sigma,\varepsilon}: \varepsilon\in (0,1]\}$ is bounded in $C^{1+\frac{\alpha}{2},2+\alpha}(E)$ for every compact set 
$E\subseteq (0,T)\times\R^d$. This is enough to conclude that
$D_t\bm{T}^P_n(\cdot){\bm\psi}_{k,\sigma,\varepsilon}$ converges 
locally uniformly in $(0,T)\times\R^d$ to $D_t\bm{T}^P_n(\cdot)((\nu\wedge\sigma) {\bm e}_k)$ as $\varepsilon$ tends to $0$.
Moreover, it is easy to check that the right-hand side of \eqref{caldo} converges pointwise in $(0,T)\times\R^d$. In particular, letting $\varepsilon$ tend to zero in both the sides of \eqref{caldo}, we infer that
\begin{align*}
D_t({\bm T}_n^P(\cdot)((\nu\wedge\sigma){\bm e}_k)_h
\le &-\sum_{\ell=1,\ \ell\neq k}^m\int_{\R^d}\vartheta_nv^P_{\ell k}(y)
(\nu(\cdot,y)\wedge\sigma)p_{h\ell}^{P,n}(\cdot,\cdot,y)dy\notag\\
& +\int_{\{\nu\le\sigma\}}(\calA^k_n\nu)(y)p_{hk}^{P,n}(\cdot,\cdot,y)dy-\sigma\int_{\{\nu>\sigma\}}\vartheta_n(y)v_{kk}(y)p_{hk}^{P,n}(\cdot,\cdot,y)dy\notag\\
& +\int_{\{\nu\leq \sigma\}}D_t\nu(\cdot,y)p_{hk}^{P,n}(\cdot,\cdot,y)dy \notag \\
=&-\sigma\sum_{\ell=1}^m\int_{\{\nu>\sigma\}}\vartheta_n(y) v_{\ell k}^P(y)p_{h\ell}^{P,n}(\cdot,\cdot,y)dy
+\int_{\{\nu\le\sigma\}}({\mathcal A}^k_n\nu)(y)p_{hk}^{P,n}(\cdot,\cdot,y)dy\notag\\
&-\sum_{\ell=1,\ \ell\neq k}^m\int_{\{\nu\le\sigma\}}\vartheta_n(y) v_{\ell k}^P(y)\nu(\cdot,y)p_{h\ell}^{P,n}(\cdot,\cdot,y)dy\\
& +\int_{\{\nu\leq \sigma\}}D_t\nu(\cdot,y)p_{hk}^{P,n}(\cdot,\cdot,y)dy.
\end{align*}

We sum up the previous formula over $k\in\{1,\ldots,m\}$ and obtain that
\begin{align*}
&D_t({\bm T}^P_n(\cdot)({\bm\nu}\wedge\sigma))_h\\
\le &-\sigma\sum_{\ell=1}^m\int_{\{\nu>\sigma\}}\vartheta_n(y)\sum_{k=1}^m v_{\ell k}^P(y)p_{h\ell}^{P,n}(\cdot,\cdot,y)dy
+\int_{\{\nu\le\sigma\}}\sum_{k=1}^m({\mathcal A}^k_n\nu)(y)p_{hk}^{P,n}(\cdot,\cdot,y)dy\\
&-\sum_{k=1}^m\sum_{\ell=1,\ \ell\neq k}^m\int_{\{\nu\le\sigma\}}\vartheta_n(y) v_{\ell k}^P(y)\nu(\cdot,y)p_{h\ell}^{P,n}(\cdot,\cdot,y)dy+\sum_{k=1}^m\int_{\{\nu\leq \sigma\}}D_t\nu(\cdot,y)p_{hk}^{P,n}(\cdot,\cdot,y)dy\\
\le &\int_{\{\nu\le\sigma\}}\sum_{k=1}^m({\mathcal A}^k_n\nu)(y)p_{hk}^{P,n}(\cdot,\cdot,y)dy-\sum_{k=1}^m\sum_{\ell=1,\ \ell\neq k}^m\int_{\{\nu\le\sigma\}}\vartheta_n(y) v_{\ell k}^P(y)\nu(\cdot,y)p_{h\ell}^{P,n}(\cdot,\cdot,y)dy\\
&+\sum_{k=1}^m\int_{\{\nu\leq \sigma\}}D_t\nu(\cdot,y)p_{hk}^{P,n}(\cdot,\cdot,y)dy\\
=&\int_{\{\nu\le\sigma\}}\sum_{\ell=1}^m\bigg (
\calA^\ell_n\nu(\cdot,y)-\sum_{k=1,k\neq \ell}^m\vartheta_n(y)v_{\ell k}^P(y)\nu(\cdot,y)\bigg )p_{h\ell}^{P,n}(\cdot,\cdot,y)dy\\
&+\sum_{k=1}^m\int_{\{\nu\leq \sigma\}}D_t\nu(\cdot,y)p_{hk}^{P,n}(\cdot,\cdot,y)dy
\end{align*}
since, by assumptions, the sum of the elements of every row of $V^P$ is a nonnegative function.

Next, we notice that
\begin{align*}
\calA^\ell_n\nu(\cdot,y)-\sum_{k=1,k\neq \ell}^m\vartheta_n(y)v_{\ell k}^P(y)\nu(\cdot,y) 
= ({\bm \calA}^P_n{\bm \nu}(\cdot,y))_\ell, \qquad y\in\R^d,\;\,\ell=1,\ldots,m.
% \label{racc_lyap_int}
\end{align*}
Therefore,
\begin{align*}
D_t({\bm T}^P_n(\cdot)({\bm \nu}\wedge \sigma))_h
\leq & \sum_{\ell=1}^m\int_{\{\nu\leq \sigma\}}({\bm \calA}^P_n{\bm \nu}(\cdot,y))_\ell p_{h\ell}^{P,n}(\cdot,\cdot,y)dy
+\sum_{k=1}^m\int_{\{\nu\leq \sigma\}}D_t\nu(\cdot,y)p_{hk}^{P,n}(\cdot,\cdot,y)dy\\
= &({\bm T}^P_n(\cdot)(\chi_{\{\nu\leq \sigma\}}(D_t\bm\nu+{\bm \calA}^P_n{\bm \nu})))_h.
\end{align*} 

Let us observe that, from Hypothesis \ref{hyp-base}(iii), it follows that
\begin{align*}
\chi_{\{\nu\leq \sigma\}}(D_t\bm\nu+{\bm \calA}^P_n{\bm \nu})
= & \chi_{\{\nu\le\sigma\}}((D_t\bm\nu+\vartheta_n\bm{\calA}^P\bm{\nu}+(1-\vartheta_n)\eta^0\Delta\bm{\nu}) \\
\le & g\bm{\nu}\chi_{\{\nu\le\sigma\}}+\chi_{\{\nu\le\sigma\}}(1-\vartheta_n)(D_t\bm\nu+\eta^0\Delta\bm{\nu})
=g\bm{\nu}\chi_{\{\nu\le\sigma\}}
\end{align*}
if $n\ge R$, where $R$ is such that
$\{\nu\le\sigma\}\subset (0,T)\times B(R)$.
Hence,  
\begin{align*}
%\label{stima_der_smgr_lyap}
D_t({\bm T}^P_n(\cdot)({\bm \nu}\wedge \sigma))\le {\bm T}^P_n(\cdot)(\chi_{\{\nu\leq \sigma\}}(D_t\bm\nu+{\bm \calA}^P_n{\bm \nu}))
\le
g {\bm T}^P_n(\cdot)(\bm{\nu}\wedge\sigma),\qquad\;\,n>R.
\end{align*}
Letting $n$ tend to infinity and taking Proposition \ref{prop:conv_SMGR} into account, it follows that 
$D_t({\bm T}^P(\cdot)({\bm \nu}\wedge \sigma))
\le g {\bm T}^P(\cdot)({\bm \nu}\wedge \sigma)$
in $(0,T]\times\R^d$. 
Applying Gronwall's Lemma we infer that
${\bm T}^P(t)({\bm \nu(t,\cdot)\wedge \sigma})\leq e^{G(t)}({\bm\nu(0,\cdot)}\wedge \sigma)$ in $[0,T]\times \R^d$.
Finally, letting $\sigma$ tend to infinity, the assertion follows by monotone convergence.

{\it Step 2}. If $M<0$ in \eqref{segno_V_P}, then we introduce the operator ${\bm \calA}^P_M={\bm\calA}^P+MI$ and observe that $\bm\nu$ is a time-dependent Lyapunov function for the operator $D_t+\bm\calA^P_M$, with the function $g$ replaced by $g+M$ and $M$ replaced by $0$ in \eqref{segno_V_P}. It follows that we can associate a semigroup $({\bm T}^P_M(t))$ in $C_b(\R^d)$ to ${\bm\calA}^P_M$ by means of the unique solution to the associated Cauchy problem. Further, ${\bm T}^P_M(t)\f=e^{Mt}{\bm T}^P(t)\f$ for every $\f\in C_b(\R^d)$, $t\in [0,\infty)$,
and the sum of the elements of each row of the potential matrix in ${\bm\calA}^P_M$ is nonnegative. Hence, from Step 1 we deduce that 
we can compute the operator ${\bm T}^P_M(t)$ on the function $\bm{\nu}$ and $({\bm T}^P_M(t){\bm \nu}(t,\cdot))_k\leq e^{G(t)+Mt}\nu(0,\cdot)$ for every $t\in [0,T]$ and $k=1,\ldots,m$. 
Thus, we get
\begin{align*}
e^{Mt}{\bm T}^P(t)({\bm\nu(t,\cdot)}\wedge n)
= {\bm T}^P_M(t)({\bm\nu(t,\cdot)}\wedge n)\le {\bm T}^P_M(t){\bm\nu(t,\cdot)}
\leq e^{G(t)+Mt}{\bm\nu(0,\cdot)}
\end{align*}
for every $t\in [0,T]$ and $n\in\N$.
Letting $n$ tend to infinity, by the monotone convergence theorem we infer that ${\bm T}^P(t){\bm\nu}(t,\cdot)$ is well-defined and ${\bm T}^P(t){\bm \nu}(t,\cdot)\leq e^{G(t)}{\bm \nu}(0,\cdot)$ for every $t\in [0,T]$.

{\em Step 3}. We fix $\sigma>0$ and recall that the sequence
$(\T^{\mathcal{D},P,n}(\cdot)(\bm{\nu}\wedge\sigma))$ converges to $\T^{P}(\cdot)(\bm{\nu}\wedge\sigma)$ in a monotone way in $[0,T]\times\R^d$. The positivity of the semigroup $(\T^P(t))$ implies that
$\T^P(\cdot)(\bm{\nu}\wedge\sigma)\le\T^P(\cdot){\bm{\nu}}$ for every $\sigma>0$. Putting everything together, we conclude that $\T^{\mathcal{D},P,n}(t)(\bm{\nu}(t,\cdot)\wedge\sigma)\le e^{G(t)}\bm{\nu}(0,\cdot)$ for every $t\in [0,T]$. Letting $\sigma$ tend to $\infty$, from the monotone convergence theorem we deduce that
$\T^{\mathcal{D},P,n}(t)\bm{\nu}(t,\cdot)\le e^{G(t)}\bm{\nu}(0,\cdot)$ for every $t\in [0,T]$ and every $n\in\N$.
\end{proof}

\section{Kernel estimates}
\label{sec:kernels}

In this section, we provide pointwise estimates for the kernels of the semigroups $(\T(t))$ and $(\T^P(t))$. For this purpose, we need some preliminary results. 
% The following lemma can be obtained arguing as in the proof of Lemma \ref{lmma:ug_Q_com}.

% \begin{lemma}
% \label{lmma:ug_Q_n}
% Let $0\leq a<b$ and $\varphi\in C_c^{1,2}([a,b]\times B(n))$. Then,
% \begin{align*}
% &\int_{R_n(a,b)} (D_t \varphi(t,y) + \calA_k \varphi (t,y))p_{hk}^{\mathcal D,P,n}(t,x,y)dtdy \notag \\
% &-\sum_{j=1,j\neq k}^m\int_{R_n(a,b)}v_{jk}^P(y)\varphi(t,y)p_{hj}^{\mathcal D,P,n}(t,x,y)dtdy \notag \\
% =& \int_{B(n)} (p_{hk}^{\mathcal D,P,n}(b,x,y)\varphi(b,y)-p_{hk}^{\mathcal D,P,n}(a,x,y)\varphi(a,y))dy
% \end{align*}
% for every $x\in B(n)$,  $h,k=1,\dots, m$, where, for every $k=1,\ldots,m$,
% $\calA_k:=\mathrm{div} (Q^k\nabla ) -\langle b^k,\nabla \rangle -v_{kk}$. 
% \end{lemma}

In the following lemma, we consider a matrix-valued function $Q:\R^d\to \R^{d\times d}$ such that $Q(x)=(q_{ij}(x))_{i,j=1}^d$ is symmetric for every $x\in \R^d$, its entries belong to 
$C^\alpha_{\rm loc}(\R^d)$ for some $\alpha\in (0,1)$ and there exists a positive constant $\eta$ such that $\langle Q(x) \xi,\xi\rangle\geq \eta|\xi|^2$ for every $x,\xi\in\R^d$. 

\begin{lemma}\label{Thm: stima norma infinito}
Fix $T>0$, $0\leq a_0<b_0\leq T$, $n\in\N$, $r>d+2$ and consider two functions $f\in L^\frac{r}{2}(R_n(a_0,b_0))$ and ${\bm h}\in L^r(R_n(a_0,b_0))$. Further, assume that $u\in \mathcal{H}^{p,1}(R_n(a_0,b_0))\cap C(\overline{R_n(a_0,b_0)})$ for some $p>1$. \\
If there exists $\Omega\Subset B(n)$ such that $f,{\bm h}$ and $u$ vanish on $(a_0,b_0)\times \overline \Omega^c$ and
\begin{equation}\label{eq: Thm 3.7 KunzeLorenziRhandi_n}
\int_{R_n(a_0,b_0)}(\langle Q\nabla u,\nabla\psi\rangle+\psi D_tu)dtdx
= \int_{R_n(a_0,b_0)} f\psi dtdx
+\int_{R_n(a_0,b_0)} \langle {\bm h},\nabla \psi\rangle dtdx
\end{equation}
for every $\psi\in C_c^\infty([a_0,b_0]\times B(n))$, then %$u$ is bounded on $R_n(a_0,b_0)$ and 
there exists a constant $C>0$, depending only on $\eta,d$ and $r$, such that
\begin{equation*}
\norm{u}_{L^\infty(R_n(a_0,b_0))}\leq C(\norm{u}_{L^2(R_n(a_0,b_0))}+\norm{f}_{L^{\frac{r}{2}}(R_n(a_0,b_0))}+\norm{\bm h}_{L^{r}(R_n(a_0,b_0))}).
%\label{diluvio}
\end{equation*}
\end{lemma}

\begin{proof}
The main step consists in showing that
\begin{align}
\int_{R(a_0,b_0)}(\langle Q_n\nabla \overline u,\nabla\psi\rangle+\psi D_t \overline u) dtdx= \int_{R(a_0,b_0)} \overline f\psi \,dt\,dx
+\int_{R(a_0,b_0)} \langle \overline {\bm h},\nabla \psi\rangle dtdx
\label{diluvio-1}
\end{align}
for every $\psi\in C_c^\infty(\overline{R(a_0,b_0)})$.
Here, $\overline v$ denotes the trivial extension of $v$, defined on $R_n(a_0,b_0)$, to the whole $R(a_0,b_0)$) and $Q_n=\varphi_n Q+\eta(1-\varphi_n)I$, where $\varphi_n$ is a smooth function satisfying $\chi_{B(n)}\leq \varphi_n\leq \chi_{B(2n)}$.
Once this formula is proved, we apply \cite[Theorem 3.7]{KLR}, with $Q$ replaced by $Q_n$, to get the assertion.

To begin with, we notice that $\overline f\in L^{\frac{r}{2}}(R(a_0,b_0))$, $\overline{\bm h}\in L^r(R(a_0,b_0))$ and  $\overline u$ belongs to $\in L^\infty((a_0,b_0);L^2(\R^d))\cap\mathcal H^{p,1}(R(a_0,b_0))\cap C_b(\overline {R(a_0,b_0)})$.
Further, we fix $0<\varepsilon<\frac12d(\Omega,B(n)^c)$ and let $\varphi\in C_c^\infty([a_0,b_0]\times B(n))$ be such that $\varphi\equiv1$ on $(a_0,b_0)\times \Omega_{\varepsilon}$, where $\Omega_\varepsilon=\{x+y\in \R^d:x\in \Omega, \ y\in B(\varepsilon)\}\Subset B(n)$. 
Finally, we fix $\psi\in C_c^\infty(\overline{R(a_0,b_0)})$. The function $\psi\varphi$ belongs to $C_c^\infty([a_0,b_0]\times B(n))$ and, from \eqref{eq: Thm 3.7 KunzeLorenziRhandi_n}, we infer that
\begin{align*}
\int_{R_n(a_0,b_0)}(\langle Q\nabla u,\nabla(\psi\varphi)\rangle+\psi\varphi D_t u)dtdx 
= \int_{R_n(a_0,b_0)} f\psi\varphi dtdx
+\int_{R_n(a_0,b_0)} \langle {\bm h},\nabla (\psi\varphi)\rangle dtdx.
\end{align*}
Since $u,f,\bm h$ have support in $(a_0,b_0)\times \overline \Omega$ and $\varphi\equiv1$ on $(a_0,b_0)\times \Omega_\varepsilon$, it follows that
\begin{align*}
\int_{(a_0,b_0)\times \Omega_\varepsilon}(\langle Q\nabla u,\nabla\psi\rangle+\psi D_t u)dtdx
= \int_{(a_0,b_0)\times \Omega_\varepsilon} f\psi dtdx
+\int_{(a_0,b_0)\times \Omega_\varepsilon} \langle {\bm h},\nabla \psi\rangle dtdx.
\end{align*}
From this formula, \eqref{diluvio-1} follows immediately.
\end{proof}

We are now ready to establish pointwise upper bounds for the kernels $p_{hk}^{\mathcal D,P,n}$ for all $h,k=1,\dots,m$ and $n\in\N$, under the following set of assumptions.

\begin{hyp}
\label{hyp-stime-nucleo}
Fix $T>0$, $x\in\R^d$ and $0<a_0<b_0<T$. Let ${\bm\nu_1}=(\nu_1,\dots,\nu_1)$ and ${\bm\nu_2}=(\nu_2,\dots,\nu_2)$ be two time-dependent Lyapunov functions for the operator ${\bm {\mathcal L}}:=D_t+{\bm\calA}^P$, with respect to $\varphi$ and $g_1$ and $g_2$, respectively, such that $\nu_1\leq\nu_2$. Let $1\leq w\in C^{1,2}((0,T)\times\R^d)$ be a weight function such that
there exist $s>d+2$ and constants $c_1,\ldots ,c_8$, possibly depending on the interval $(a_0,b_0)$, with
\begin{tabbing}
\= {\rm (i)} $w\le c_1^{\frac{s}{2}}\nu_1$,
\quad\qquad\qquad \qquad\quad\= {\rm (ii)} $|Q^h\nabla w|\le c_2w^{\frac{s-1}{s}}\nu_1^{\frac{1}{s}}$,\qquad\quad\={\rm (iii)} $|{\rm div}(Q^h\nabla w)|\le c_3w^{\frac{s-2}{s}}\nu_1^{\frac{2}{s}}$
\\[0.5em]
\> {\rm (iv)} $|D_t w|\le c_4w^{\frac{s-2}{s}}\nu_1^{\frac{2}{s}}$,
\>{\rm (v)} $|V^h|\le c_5 w^{-\frac{2}{s}}\nu_2^{\frac{2}{s}}$, \>{\rm (vi)} $|b^h|\leq c_6 w^{-\frac{1}{s}}\nu_2^{\frac{1}{s}}$,\\[0.5em]
\> {\rm (vii)} $|Q^h|\leq c_7 w^{-\frac{1}{s}}\nu_1^{\frac{1}{s}}$ , \> {\rm (viii)} $|R^h|\le c_8w^{-\frac{2}{s}}\nu_1^{\frac{2}{s}}$
\end{tabbing}
on $[a_0,b_0]\times \R^d$, for every $h=1,\dots,m$, where $V^h$ denotes the $h$-th column of the matrix $V$ and $R^h$ is the matrix whose entries are $D_iq^h_{ij}$ with $i,j=1,\ldots,d$.
\end{hyp}

Under Hypotheses \ref{hyp-stime-nucleo}, for $a$ and $b$ such that $a_0<a<b<b_0$, we introduce the function 
\begin{align}
\mathscr H_{a,b}= &\Big(c_1^\frac{s}{2}+\frac{c_1^\frac{s}{2}}{[(a-a_0)\wedge (b_0-b)]^\frac{s}{2}}+c_2^s+c_3^\frac{s}{2}+c_4^\frac{s}{2}+c_1^\frac{s}{4} c_2^\frac{s}{2}+c_1^\frac{s}{4}c_7^\frac{s}{2}+c_7^s+c_8^\frac{s}{2}\Big)\nu_1(0,\cdot)\int_{a_0}^{b_0} e^{G_1(t)} dt\notag\\
&\quad\;\,+\Big(c_1^\frac{s}{4}c_6^\frac{s}{2}+c_2^\frac{s}{2}c_6^\frac{s}{2}+c_5^\frac{s}{2}+c_6^s\Big) \nu_2(0,\cdot)\int_{a_0}^{b_0} e^{G_2(t)} dt, 
\label{funzione_stima}
\end{align}
where $G_i(s):=\displaystyle \int_{0}^{s}g_i(\sigma)d\sigma$ for every $s\in[0,T]$ and $i=1,2$.

\begin{theorem}\label{thm: stima-nucleo_n}
Under Hypotheses $\ref{hyp-base}$ and $\ref{hyp-stime-nucleo}$, the function $p_{hk}(\cdot,x,\cdot)$ belongs to ${\mathcal H}^{r,1}(R_n(a,b))$ for every $h,k=1,\ldots,m$, $n\in\N$, $r\in (1,\infty)$ and $0<a<b<T$. Moreover, there exists a positive constant $C$, depending only on $\eta^0_1,\ldots,\eta^0_m$, $d$ and $s$, such that, for every $a_0<a<b<b_0$,
\begin{align}
w(t,y)\sum_{k=1}^m |p_{hk}(t,x,y)|\le  C \mathscr H_{a,b}(x)
\label{eq: stima-nucleo}
\end{align}
for any $h=1,\dots, m$, $(t,y)\in (a,b)\times\R^d$ and $x\in \R^d$. 
\end{theorem}

\begin{proof}
Fix $x\in\R^d$, $a$, $b$, $a_0$ and $b_0$ as in the statement and $h,k\in\{1,\ldots,m\}$. The main step of the proof consists in proving \eqref{eq: stima-nucleo}, with $p_{hk}$ being replaced by $p_{hk}^{{\mathcal D},P,n}$ and with a positive constant $C$, independent of $n$. Indeed, once this property is proved, we can obtain \eqref{eq: stima-nucleo} and the smoothness property of $p_{hk}(\cdot,x,\cdot)$, recalling that the sequence $(p_{hk}^{\mathcal D,P,n})$ pointwise converges to $p_{hk}^P$ in $(0,\infty)\times\R^d\times\R^d$ (see Remark \ref{rmk:limit_kernel}). Hence, taking the limit as  $n$ tends to infinity in the estimate for $p_{hk}^{{\mathcal D},P,n}$,  we obtain that
$w(t,y)p_{hk}^P(t,x,y)\le C\mathscr{H}_{a,b}(x)$ for every $(t,y)\in (a,b)\times\R^d$. Next, from \eqref{stima_nuclei_completa} and the previous estimate, it follows immediately that 
\eqref{eq: stima-nucleo} holds true for almost every $(t,y)\in R(a,b)$. In particular, this implies that the function $p_{hk}(\cdot,x,\cdot)$ belongs to $L^{\infty}(R(a_0,b_0))$.
Hence, by Proposition \ref{prop 2.3.1 tesi}, 
$p_{hk}(\cdot,x,\cdot)$ belongs to ${\mathcal H}^{r,1}(R_n(a_1,b_1))$ for every $r\in (1,\infty)$ and $(a_1,b_1)\subset (a,b)$.
This allows us to extend the validity of \eqref{eq: stima-nucleo} to every $(t,y)\in R(a,b)$.

Based on the above remarks, we now prove that
$p_{hk}^{\mathcal D,P,n}(\cdot,x,\cdot)\le C{\mathscr H}_{a,b}(x)$ for every $n\in\N$ such that $x\in B(n)$.
For this purpose, we introduce a function 
$\vartheta\in C^\infty(\R)$ such that $\chi_{[a,b]}\le\vartheta\le\chi_{[a_0,b_0]}$ and $|\vartheta'|\leq \frac{2}{(a-a_0)\wedge (b_0-b)}$ and, for every $n\in\N$, the function $\vartheta_n$, defined by  
$\vartheta_n(y)=\vartheta_0(n^{-1}|y|)$ for every $y\in\R^d$, where  $\vartheta_0(\xi)=\chi_{\left [0,\frac{1}{2}\right )}(\xi)+
\exp\left(1-\frac{1}{1-(4\xi-2)^3}\right)\chi_{\left (\frac{1}{2},\frac{3}{4}\right )}(\xi)$ for every $\xi\in [0,\infty)$. Each function $\vartheta_n$ belongs to $C^2_c(\R^d)$, $\chi_{B(n/2)}\leq \vartheta_n\leq\chi_{B(3n/4)}$ and there exists a positive constant $c_0$ such that
\begin{align}
\label{stima_der_theta_n}
|D^2\vartheta_n(x)|+|\nabla \vartheta_n(x)|\leq c_0\sqrt{\vartheta_n(x)}, \qquad x\in\R^d, \ n\in\N.
\end{align}

Applying formula \eqref{int_parti_nuclei_com_2}, with
$p_{hk}$ and $R(a,b)$ being replaced by $p_{hk}^{\mathcal{D},P,n}$ and $R_n(a,b)$, respectively (see Remark \ref{rmk_kernels}), and
with $\varphi(t,y)=\vartheta^{\frac{s}{2}}(t)\vartheta_n^{\frac{s}{2}}(y)w(t,y)\psi(t,y)$ for every $(t,y)\in (0,\infty)\times B(n)$, where $\psi\in C^{1,2}_c([a_0,b_0]\times B(n))$, it follows that
\begin{align}
&\int_{R_n(a_0,b_0)} (D_t \varphi(t,y) + \calA_k \varphi (t,y))p_{hk}^{\mathcal D,P,n}(t,x,y)\, dt dy\notag\\
&-\sum_{j=1,j\neq k}^m\int_{R_n(a_0,b_0)}v_{jk}^P(y)\varphi(t,y)p_{hj}^{\mathcal D,P,n}(t,x,y)dtdy=0.
\label{milla}
\end{align}
By straightforward computations we obtain that
\begin{align*}
D_t \varphi=& \vartheta^\frac{s}{2}\vartheta_n^\frac{s}{2}\Big( \frac{s}{2} \vartheta^{-1}\vartheta'w+ D_t w\Big)\psi+\vartheta^\frac{s}{2}\vartheta_n^\frac{s}{2} w D_t \psi,\\[1mm]
\mathrm{div}(Q^k \nabla \varphi)=& \vartheta^\frac{s}{2}\vartheta_n^\frac{s}{2}w\, \mathrm{div}(Q^k \nabla \psi)+ \vartheta^\frac{s}{2}\vartheta_n^\frac{s}{2}\langle s \vartheta_n^{-1} w Q^k \nabla \vartheta_n+2 Q^k \nabla w, \nabla\psi\rangle\\
&+ \vartheta^\frac{s}{2}\vartheta_n^\frac{s}{2}\Big (\mathrm{div}(Q^k \nabla w) + s \vartheta_n^{-1} \langle Q^k \nabla w, \nabla \vartheta_n\rangle+ \frac{s}{2}\vartheta_n^{-1}w\,\mathrm{div}(Q^k\nabla \vartheta_n)\\
&\qquad\qquad\;\;+\frac{s(s-2)}{4}\vartheta_n^{-2} w\langle Q^k \nabla \vartheta_n, \nabla \vartheta_n\rangle\Big )\psi,\\[1mm]
\langle b^k, \nabla\varphi\rangle=& \vartheta^\frac{s}{2}\vartheta_n^\frac{s}{2} \Big (\langle b^k, \nabla w\rangle +\frac{s}{2}\vartheta_n^{-1} w \langle b^k, \nabla\vartheta_n\rangle \Big )\psi + \vartheta^\frac{s}{2}\vartheta_n^\frac{s}{2}  w \langle b^k, \nabla\psi\rangle,
\end{align*}
which, replaced in \eqref{milla}, give
\begin{align*}
&\int_{R_n(a_0,b_0)} w\mathfrak{r}^{nx}_{hk} (-D_t \psi-\mathrm{div}(Q^k \nabla \psi)) dtdy\\
= &\int_{R_n(a_0,b_0)} \bigg (\frac{s}{2} \vartheta^{-1}\vartheta'w\mathfrak{r}^{nx}_{hk}+ \mathfrak{r}^{nx}_{hk}D_t w+\mathfrak{r}^{nx}_{hk}\mathrm{div}(Q^k \nabla w) + s \vartheta_n^{-1}\mathfrak{r}^{nx}_{hk} \langle Q^k \nabla w, \nabla \vartheta_n\rangle\\
&\qquad\qquad\quad+ \frac{s}{2}\vartheta_n^{-1}w\mathfrak{r}^{nx}_{hk}\,\mathrm{div}(Q^k\nabla \vartheta_n)+\frac{s(s-2)}{4}\vartheta_n^{-2} w\mathfrak{r}^{nx}_{hk}\langle Q^k\nabla \vartheta_n, \nabla \vartheta_n\rangle \\
&\qquad\qquad\quad+\mathfrak{r}^{nx}_{hk}\langle b^k, \nabla w\rangle+\frac{s}{2}\vartheta_n^{-1} w\mathfrak{r}^{nx}_{hk} \langle b^k, \nabla\vartheta_n\rangle-w\sum_{j=1}^m v_{jk}^P \mathfrak{r}_{hj}^{nx}\bigg )\psi dtdy\\
&+ \int_{R_n(a_0,b_0)} \mathfrak{r}^{nx}_{hk} \langle s \vartheta_n^{-1} w Q^k \nabla \vartheta_n
+2 Q^k \nabla w + wb^k, \nabla\psi\rangle dtdy,
\end{align*}
where $\mathfrak{r}^{n,x}_{hk}(t,y)=\vartheta^\frac{s}{2}(t)\vartheta_n^\frac{s}{2}(y) p_{hk}^{\mathcal D,P,n}(t,x,y)$ for every $(t,y)\in (0,\infty)\times B(n)$.

Integrating by parts the left-hand side of the above equality, we find that formula \eqref{eq: Thm 3.7 KunzeLorenziRhandi_n} holds with the matrix $Q$ being replaced by $Q^k$, $u=w\mathfrak{r}^{nx}_{hk}$ and
\begin{align*}
f=& \frac{s}{2} \vartheta^{-1}\vartheta'w\mathfrak{r}^{nx}_{hk}+ \mathfrak{r}^{nx}_{hk}D_t w+\mathfrak{r}^{nx}_{hk}\mathrm{div}(Q^k \nabla w) + s \vartheta_n^{-1}\mathfrak{r}^{nx}_{hk} \langle Q^k \nabla w, \nabla \vartheta_n\rangle\\
&+\frac{s}{2}\vartheta_n^{-1}w\mathfrak{r}^{nx}_{hk}\,\mathrm{div}(Q^k \nabla \vartheta_n)+\frac{s(s-2)}{4}\vartheta_n^{-2} w\mathfrak{r}^{nx}_{hk}\langle Q^k\nabla \vartheta_n, \nabla \vartheta_n\rangle\\
&+\mathfrak{r}^{nx}_{hk}\langle b^k, \nabla w\rangle +\frac{s}{2}\vartheta_n^{-1} w\mathfrak{r}^{nx}_{hk} \langle b^k, \nabla\vartheta_n\rangle-w\sum_{j=1}^m v_{jk}^P \mathfrak{r}_{hj}^{nx};\\[1mm]
\bm{h}=& s \vartheta_n^{-1} w \mathfrak{r}^{nx}_{hk} Q^k \nabla \vartheta_n +2 \mathfrak{r}^{nx}_{hk} Q^k \nabla w + w\mathfrak{r}^{nx}_{hk} b^k.
\end{align*}

From \cite[Chapter IV, Section 2, Theorem 3.4]{Eid}, it follows that $p_{hk}^{\mathcal D,P,n}(\cdot,x,\cdot)\in C^{1+\frac{\alpha}{2},2+\alpha}((0,T)\times B(n))$ for every $x\in B(n)$. Since $w$, along with its first-order partial derivatives, is bounded on $B(n)$, we infer that $w\mathfrak{r}^{nx}_{hk}\in L^\infty(R_n(a_0,b_0))\cap \mathcal H^{\sigma,1}(R_{n}(a_0,b_0))$ for all $\sigma\in (1, s/2)$. Then, by Lemma \ref{Thm: stima norma infinito} (with $Q=Q^k$) there exists a positive constant $C$, depending only on $\eta^0_1,\ldots,\eta^0_m$, $d$ and $s$, such that
\begin{equation}\label{eq: stima u_infty}
\|u\|_{L^\infty(R_n(a_0,b_0))}\leq C(\|u\|_{L^2(R_n(a_0,b_0))}+\|f\|_{L^{\frac{s}{2}}(R_n(a_0,b_0))}+\|\bm h\|_{L^{s}(R_n(a_0,b_0))}).
\end{equation}

For notation convenience, we denote by $\|\cdot\|_p$ the $L^p$-norm over the cylinder $R_n(a_0,b_0)$ for $p\in [1,\infty)\cup\{\infty\}$. Moreover, we set
\begin{eqnarray*} 
M_{i,h,j}^x:=\int_{R_n(a_0,b_0)} \nu_i(t,y) p_{hj}^{\mathcal D,P,n}(t,x,y)dtdy,\qquad\;\, i=1,2,\;\, j=1,\ldots, m.
\end{eqnarray*}
We observe that $M_{i,h,j}^x<\infty$ by Theorem \ref{thm:int_time_lyap_fct}.
We now estimate the right-hand side of \eqref{eq: stima u_infty}. 
In the following computations, the constant $C$ may vary from line to line, but it is always independent of $\|Q^k\|_{\infty}$.
By applying Hypotheses \ref{hyp-stime-nucleo}, estimate \eqref{stima_der_theta_n} and setting $\eta_0=\min_{k=1,\ldots,m}\eta^0_k$, we get
\begin{align*}
&\|w\mathfrak{r}^{nx}_{hk}\|_2^2\leq c_1^\frac{s}{2}\|w\mathfrak{r}^{nx}_{hk}\|_{\infty} M_{1,h,k}^x;\\[1mm]
%%%2
&\|\vartheta^{-1}\vartheta'w\mathfrak{r}^{nx}_{hk}\|_{\frac{s}{2}}^{\frac{s}{2}}
\leq \frac{2^\frac{s}{2}}{[(a-a_0)\wedge (b_0-b)]^\frac{s}{2}} \|w\mathfrak{r}^{nx}_{hk}\|_{\infty}^\frac{s-2}{2} \int_{R_n(a_0,b_0)} \vartheta^{-\frac{s}{2}} w\mathfrak{r}^{nx}_{hk}dtdy\notag\\
&\phantom{\|\vartheta^{-1}\vartheta'w\mathfrak{r}^{nx}_{hk}\|_{\frac{s}{2}}^{\frac{s}{2}}}
\le\frac{2^\frac{s}{2}c_1^\frac{s}{2}}{[(a-a_0)\wedge (b_0-b)]^\frac{s}{2}} \|w\mathfrak{r}^{nx}_{hk}\|_{\infty}^\frac{s-2}{2}M_{1,h,k}^x;\\[1mm]
%%%3
&\|\mathfrak{r}^{nx}_{hk}D_t w\|_{\frac{s}{2}}^{\frac{s}{2}}
\leq c_4^\frac{s}{2} \int_{R_n(a_0,b_0)} (\mathfrak{r}^{nx}_{hk})^\frac{s}{2} w^\frac{s-2}{2} \nu_1 dtdy
\leq c_4^\frac{s}{2} \|w\mathfrak{r}^{nx}_{hk}\|_\infty^\frac{s-2}{2}M_{1,h,k}^x;\\[1mm]
%%%4
&\|\mathfrak{r}^{nx}_{kh}\mathrm{div}(Q^k \nabla w)\|_{\frac{s}{2}}^{\frac{s}{2}}
\leq c_3^\frac{s}{2}\int_{R_n(a_0,b_0)} (\mathfrak{r}^{nx}_{hk})^\frac{s}{2} w^\frac{s-2}{2} \nu_1 dtdy
\leq c_3^\frac{s}{2} \|w\mathfrak{r}^{nx}_{hk}\|_{\infty}^\frac{s-2}{2}M_{1,h,k}^x;\\[1mm]
%%%5
&\|\vartheta_n^{-1}\mathfrak{r}^{nx}_{hk} \langle Q^k \nabla w, \nabla \vartheta_n\rangle\|_{\frac{s}{2}}^{\frac{s}{2}}
\leq \int_{R_n(a_0,b_0)} \vartheta_n^{-\frac{s}{2}}(\mathfrak{r}^{nx}_{hk})^\frac{s}{2} |Q^k \nabla w|^\frac{s}{2} |\nabla \vartheta_n|^\frac{s}{2} dtdy\\
&\phantom{\norm{\vartheta_n^{-1}\mathfrak{r}^{nx}_{hk} \langle Q^k \nabla w, \nabla \vartheta_n\rangle}_{\frac{s}{2}}^{\frac{s}{2}}\;}\leq c_0^\frac{s}{2} c_2^\frac{s}{2}\int_{R_n(a_0,b_0)} \vartheta_n^{-\frac{s}{4}}(\mathfrak{r}^{nx}_{hk})^\frac{s}{2} w^\frac{s-1}{2} \nu_1^\frac{1}{2} dtdy\\
&\phantom{\|\vartheta_n^{-1}\mathfrak{r}^{nx}_{hk} \langle Q^k \nabla w, \nabla \vartheta_n\rangle\|_{\frac{s}{2}}^{\frac{s}{2}}\;}\leq c_0^\frac{s}{2} c_2^\frac{s}{2}\|w\mathfrak{r}^{nx}_{hk}\|_{\infty}^\frac{s-2}{2} \int_{R_n(a_0,b_0)} \vartheta_n^{-\frac{s}{4}} \mathfrak{r}^{nx}_{hk} w^{\frac{1}{2}}\nu_1^\frac{1}{2} dtdy\\
&\phantom{\|\vartheta_n^{-1}\mathfrak{r}^{nx}_{hk} \langle Q^h \nabla w, \nabla \vartheta_n\rangle\|_{\frac{s}{2}}^{\frac{s}{2}}\;}\leq c_0^\frac{s}{2} c_1^\frac{s}{4} c_2^\frac{s}{2}\|w\mathfrak{r}^{nx}_{hk}\|_\infty^\frac{s-2}{2} M_{1,h,k}^x;\\
%%%6
&\|\vartheta_n^{-1}w\mathfrak{r}^{nx}_{hk}\,\mathrm{div}(Q^k \nabla \vartheta_n)\|_{\frac{s}{2}}^{\frac{s}{2}}
\leq 2^{\frac{s}{2}-1} \int_{R_n(a_0,b_0)} \vartheta_n^{-\frac{s}{2}}(w\mathfrak{r}^{nx}_{hk})^\frac{s}{2} (d^{\frac{s}{4}}|R^k|^\frac{s}{2} |\nabla \vartheta_n|^\frac{s}{2} + |Q^k|^\frac{s}{2} |D^2 \vartheta_n|^\frac{s}{2}) dtdy\\
&\phantom{\|\vartheta_n^{-1}w\mathfrak{r}^{nx}_{hk}\,\mathrm{div}(Q^k \nabla \vartheta_n)\|_{\frac{s}{2}}^{\frac{s}{2}}\;}\leq 2^{\frac{s}{2}-1}d^{\frac{s}{4}}c_0^\frac{s}{2} c_8^\frac{s}{2}\|w\mathfrak{r}^{nx}_{hk}\|_\infty^\frac{s-2}{2}\int_{R_n(a_0,b_0)}\vartheta_n^{-\frac{s}{2}}\mathfrak{r}^{nx}_{hk}\nu_1 dtdy\\
&\phantom{\|\vartheta_n^{-1}w\mathfrak{r}^{nx}_{hk}\,\mathrm{div}(Q^k \nabla \vartheta_n)\leq \|_{\frac{s}{2}}^{\frac{s}{2}}\;} +2^{\frac{s}{2}-1}c_0^\frac{s}{2} c_7^\frac{s}{2}\norm{w\mathfrak{r}^{nx}_{hk}}_\infty^\frac{s-2}{2}\int_{R_n(a_0,b_0)}\vartheta_n^{-\frac{s}{2}}\mathfrak{r}^{nx}_{hk}w^{\frac{1}{2}}\nu_1^\frac{1}{2} dtdy \\
&\phantom{\|\vartheta_n^{-1}w\mathfrak{r}^{nx}_{hk}\,\mathrm{div}(Q^k \nabla \vartheta_n)\|_{\frac{s}{2}}^{\frac{s}{2}}\;}\leq 2^{\frac{s}{2}-1} c_0^\frac{s}{2} (d^{\frac{s}{4}}c_8^\frac{s}{2}+c_1^\frac{s}{4}c_7^\frac{s}{2})\norm{w\mathfrak{r}^{nx}_{hk}}_\infty^\frac{s-2}{2} M_{1,h,k}^x;\\[1mm]
%%%7
&\norm{\vartheta_n^{-2} w\mathfrak{r}^{nx}_{hk}\langle Q^h \nabla \vartheta_n, \nabla \vartheta_n\rangle}_{\frac{s}{2}}^{\frac{s}{2}}
% \leq \int_{R_n(a_0,b_0)}\vartheta_n^{-s} (w\mathfrak{r}^{nx}_{hk})^\frac{s}{2}|Q^h|^{\frac{s}{2}} |\nabla \vartheta_n|^s dtdy\\
% &\phantom{\norm{\vartheta_n^{-2} w\mathfrak{r}^{nx}_{hk}\langle Q^h \nabla \vartheta_n, \nabla \vartheta_n\rangle}_{\frac{s}{2}}^{\frac{s}{2}}\;}
\leq c_0^{s}
\int_{R_n(a_0,b_0)}\vartheta_n^{-\frac{s}{2}} (w\mathfrak{r}^{nx}_{hk})^\frac{s}{2}|Q^h|^{\frac{s}{2}} dtdy\\
&\phantom{\norm{\vartheta_n^{-2} w\mathfrak{r}^{nx}_{hk}\langle Q^h \nabla \vartheta_n, \nabla \vartheta_n\rangle}_{\frac{s}{2}}^{\frac{s}{2}}\;}\leq c_0^s c_7^{\frac{s}{2}}\norm{w\mathfrak{r}^{nx}_{hk}}_\infty^\frac{s-2}{2}\int_{R_n(a_0,b_0)}\vartheta_n^{-\frac{s}{2}}\mathfrak{r}^{nx}_{hk} w^{\frac{1}{2}}\nu_1^\frac12 dtdy\\
% &\phantom{\norm{\vartheta_n^{-2} w\mathfrak{r}^{nx}_{hk}\langle Q^h \nabla \vartheta_n, \nabla \vartheta_n\rangle}_{\frac{s}{2}}^{\frac{s}{2}}\;}\leq c_0^s c_1^{\frac{s}{4}}c_7^{\frac{s}{2}} \norm{w\mathfrak{r}^{nx}_{hk}}_\infty^\frac{s-2}{2}\int_{R_n(a_0,b_0)}\vartheta_n^{-\frac{s}{2}} \mathfrak{r}^{nx}_{hk} \nu_1 dtdy\\
&\phantom{\norm{\vartheta_n^{-2} w\mathfrak{r}^{nx}_{hk}\langle Q^h \nabla \vartheta_n, \nabla \vartheta_n\rangle}_{\frac{s}{2}}^{\frac{s}{2}}\;}\leq c_0^sc_1^{\frac s4}c_7^\frac s2
\norm{w\mathfrak{r}^{nx}_{hk}}_\infty^\frac{s-2}{2}M_{1,h,k}^x;\\[1mm]
%%%8
&\norm{\mathfrak{r}^{nx}_{hk}\langle b^h, \nabla w\rangle}_{\frac{s}{2}}^{\frac{s}{2}}
\leq \int_{R_n(a_0,b_0)}(\mathfrak{r}^{nx}_{hk})^\frac{s}{2}|b^h|^\frac{s}{2}|\nabla w|^\frac{s}{2} dtdy\\
&\phantom{\norm{\mathfrak{r}^{nx}_{hk}\langle b^h, \nabla w\rangle}_{\frac{s}{2}}^{\frac{s}{2}}\;}\leq c_6^\frac{s}{2}\int_{R_n(a_0,b_0)}(\mathfrak{r}^{nx}_{hk})^\frac{s}{2} w^{-\frac{1}{2}} \nu_2^\frac{1}{2} |\nabla w|^\frac{s}{2} dtdy\\
&\phantom{\norm{\mathfrak{r}^{nx}_{hk}\langle b^h, \nabla w\rangle}_{\frac{s}{2}}^{\frac{s}{2}}\;}\leq \eta_0^{-\frac{s}{2}} c_2^\frac{s}{2} c_6^\frac{s}{2}\int_{R_n(a_0,b_0)}(\mathfrak{r}^{nx}_{hk})^\frac{s}{2} w^{\frac{s-2}{2}} \nu_1^\frac{1}{2}\nu_2^\frac{1}{2} dtdy\\
&\phantom{\norm{\mathfrak{r}^{nx}_{hk}\langle b^h, \nabla w\rangle}_{\frac{s}{2}}^{\frac{s}{2}}\;}\leq \eta_0^{-\frac{s}{2}} c_2^\frac{s}{2} c_6^\frac{s}{2}\norm{w\mathfrak{r}^{nx}_{hk}}_\infty^\frac{s-2}{2} M_{2,h,k}^x;\\[1mm]
%%%9
&\norm{\vartheta_n^{-1} w\mathfrak{r}^{nx}_{hk} \langle b^h, \nabla\vartheta_n\rangle}_{\frac{s}{2}}^{\frac{s}{2}}
\leq c_0^\frac{s}{2}c_6^\frac{s}{2}\norm{w\mathfrak{r}^{nx}_{hk}}_\infty^\frac{s-2}{2}\int_{R_n(a_0,b_0)}\vartheta_n^{-\frac{s}{4}} w^\frac{1}{2}\mathfrak{r}^{nx}_{hk} \nu_2^\frac{1}{2} dtdy\\
&\phantom{\norm{\vartheta_n^{-1} w\mathfrak{r}^{nx}_{hk} \langle b^h, \nabla\vartheta_n\rangle}_{\frac{s}{2}}^{\frac{s}{2}}\;}\leq c_0^\frac{s}{2}c_1^\frac{s}{4} c_6^\frac{s}{2}\norm{w\mathfrak{r}^{nx}_{hk}}_\infty^\frac{s-2}{2} M_{2,h,k}^x;\\[1mm]
%%%10
&\bigg\|w\sum_{j=1}^m v_{jk}^P \mathfrak{r}_{hj}^{nx}\bigg\|_{\frac{s}{2}}^{\frac{s}{2}}
\leq \int_{R_n(a_0,b_0)}w^\frac{s}{2}|V^k|^\frac{s}{2}\bigg(\sum_{j=1}^m  \mathfrak{r}_{hj}^{nx}\bigg )^\frac{s}{2} dtdy
% &\phantom{\bigg\|w\sum_{j=1}^m v_{jk}^P \mathfrak{r}_{hj}^{nx}\bigg\|_{\frac{s}{2}}^{\frac{s}{2}}\;}\leq \bigg\|w\sum_{j=1}^m \mathfrak{r}_{hj}^{nx}\bigg\|_\infty^\frac{s-2}{2}\sum_{j=1}^m \int_{R_n(a_0,b_0)}w|V^k|^\frac{s}{2}\mathfrak{r}_{hj}^{nx} dtdy\\
%&\phantom{\bigg\|w\sum_{j=1}^m v_{jk}^P \mathfrak{r}_{hj}\bigg\|_{\frac{s}{2}}^{\frac{s}{2}}\;}
\leq c_5^\frac{s}{2} \bigg\|w\sum_{j=1}^m \mathfrak{r}_{hj}^{nx}\bigg\|_{\infty}^\frac{s-2}{2}\sum_{j=1}^m M_{2,h,j}^x;\\[1mm]
%%%11
&\norm{\vartheta_n^{-1} w \mathfrak{r}^{nx}_{hk} Q^h \nabla \vartheta_n}_{s}^{s}
\leq c_0^s c_7^s \norm{w\mathfrak{r}^{nx}_{hk}}_\infty^{s-1}\int_{R_n(a_0,b_0)} \vartheta_n^{-\frac{s}{2}}w \mathfrak{r}^{nx}_{hk} w^{-1}\nu_1 dtdy\\
&\phantom{\norm{\vartheta_n^{-1} w \mathfrak{r}^{nx}_{hk} Q^h \nabla \vartheta_n}_{s}^{s}\;}\leq c_0^s c_7^s \norm{w\mathfrak{r}^{nx}_{hk}}_\infty^{s-1} M_{1,h,k}^x;\\[1mm]
%%%12
&\norm{\mathfrak{r}^{nx}_{hk} Q^h \nabla w}_{s}^{s}
\leq c_2^s \int_{R_n(a_0,b_0)}(\mathfrak{r}^{nx}_{hk})^s w^{s-1} \nu_1 dtdy
\leq c_2^s \norm{w\mathfrak{r}^{nx}_{hk}}_\infty^{s-1} M_{1,h,k}^x;\\[1mm]
%%%13
&\norm{w\mathfrak{r}^{nx}_{hk} b^h}_{s}^{s}
\leq c_6^s \norm{w\mathfrak{r}^{nx}_{hk}}_\infty^{s-1} \int_{R_n(a_0,b_0)} \mathfrak{r}^{nx}_{hk}\nu_2 dtdy
\leq c_6^s \norm{w\mathfrak{r}^{nx}_{hk}}_\infty^{s-1} M_{2,h,k}^x.
\end{align*}

Combining all the above inequalities with \eqref{eq: stima u_infty}, we obtain that
\begin{align*}
\norm{w \mathfrak{r}^{nx}_{hk}}_\infty\leq & C c_1^\frac{s}{4}(M_{1,h,k}^x)^\frac{1}{2} \norm{w\mathfrak{r}^{nx}_{hk}}_\infty^\frac{1}{2}
+Cc_5 \sum_{j=1}^m(M_{2,h,j}^x)^{\frac{2}{s}}\sum_{j=1}^m\norm{w\mathfrak{r}_{hj}^{nx}}_\infty^\frac{s-2}{s}\\
&+C\Big[\Big(\frac{c_1}{(a-a_0)\wedge (b_0-b)}+c_1^\frac{1}{2}c_2+c_3+c_4+c_1^\frac{1}{2}c_7+c_8\Big)(M_{1,h,k}^x)^{\frac{2}{s}}\\
&\qquad\;+\Big(c_1^\frac{1}{2}c_6+c_2c_6\Big)(M_{2,h,k}^x)^{\frac{2}{s}}\Big]\norm{w\mathfrak{r}^{nx}_{hk}}_\infty^\frac{s-2}{s} \\
&+\Big[(c_2+c_7)(M_{1,h,k}^x)^{\frac{1}{s}}+c_6(M_{2,h,k}^x)^{\frac{1}{s}}\Big]\norm{w\mathfrak{r}^{nx}_{hk}}_\infty^\frac{s-1}{s}.
\end{align*}
Summing over $k=1,\dots,m$, we get
\begin{align*}
\sum_{k=1}^m\norm{w \mathfrak{r}^{nx}_{hk}}_\infty\leq & C c_1^\frac{s}{4} \sum_{k=1}^m (M_{1,h,k}^x)^{\frac{1}{2}} \bigg (\sum_{k=1}^m\norm{w\mathfrak{r}^{nx}_{hk}}_\infty\bigg )^\frac{1}{2}\\
&+C\bigg [\bigg (\frac{c_1}{(a-a_0)\wedge (b_0-b)}+c_1^\frac{1}{2}c_2+c_3+c_4
+c_1^\frac{1}{2}c_7+c_8\bigg )\sum_{k=1}^m (M_{1,h,k}^x)^{\frac{2}{s}}\\
&\qquad+\Big(c_1^\frac{1}{2}c_6+c_2c_6+c_5\Big)\sum_{k=1}^m (M_{2,h,k}^x)^{\frac{2}{s}}\bigg ]\bigg (\sum_{k=1}^m \norm{w\mathfrak{r}^{nx}_{hk}}_\infty\bigg )^\frac{s-2}{s}\\
&+C\bigg [(c_2+c_7)\sum_{h=1}^m(M_{1,h,k}^x)^{\frac{1}{s}}
+c_6\sum_{h=1}^m(M_{2,h,k}^x)^{\frac{1}{s}}\bigg ]\bigg (\sum_{k=1}^m \norm{w\mathfrak{r}^{nx}_{hk}}_\infty\bigg )^\frac{s-1}{s}.
\end{align*}
We set 
\begin{align*}
X=& \bigg(\sum_{k=1}^m\norm{w \mathfrak{r}^{nx}_{hk}}_\infty\bigg )^\frac{1}{s},\qquad
\alpha=C c_1^\frac{s}{4} \sum_{k=1}^m(M_{1,h,k}^x)^{\frac{1}{2}}, \\
\beta=&C\bigg [(c_2+c_7)\sum_{k=1}^m (M_{1,h,k}^x)^{\frac{1}{s}} +c_6\sum_{k=1}^m(M_{2,h,k}^x)^{\frac{1}{s}}\bigg ],\\
\gamma=&C\bigg [\bigg (\frac{c_1}{(a-a_0)\wedge (b_0-b)}+c_1^\frac{1}{2}c_2+c_3+c_4+c_1^\frac{1}{2}c_7
+c_8\bigg )\sum_{k=1}^m (M_{1,h,k}^x)^{\frac{2}{s}}\notag\\
&\;\;\;\;+\Big(c_1^\frac{1}{2}c_6+c_2c_6+c_5\Big)\sum_{k=1}^m(M_{2,h,k}^x)^{\frac{2}{s}}\bigg ].
\end{align*}

The above notation yields
$X^s\leq \alpha X^\frac{s}{2} +\beta X^{s-1}+\gamma X^{s-2}$.
If we apply Young's inequality $\alpha X^\frac{s}{2}\leq \frac{1}{4} X^s+\alpha^2$, then we get $f(X)\le 0$, where
\begin{align*}
f(r):=& r^s-\frac{4}{3}\beta r^{s-1}-\frac{4}{3}\gamma r^{s-2}-\frac{4}{3}\alpha^2
=: r^{s-2} g(r)-\frac{4}{3}\alpha^2,\qquad\;\,r\in (0,\infty).
\end{align*}
We claim that it leads to
$X\leq X_0:=\frac{4}{3}\beta +\sqrt{\frac{4}{3}\gamma} +\bigg(\frac{4}{3}\alpha^2\bigg)^\frac{1}{s}$.
For this purpose, we observe that $f'(r)=(s-2)r^{s-3}g(r)+r^{s-2} g'(r)$ for every $r\in (0,\infty)$
and the function $g$ is positive and increasing in $\left(\frac{4}{3}\beta+ \sqrt{\frac{4}{3}\gamma}+ \left(\frac{4}{3}\alpha^2\right)^\frac{1}{s}, \infty \right)$. Hence, it follows that $f'(r)\geq 0$ in the above interval, so that $f$ is increasing.
Moreover, it holds that
\begin{align*}
&f\bigg(\frac{4}{3}\beta+ \sqrt{\frac{4}{3}\gamma}+ \left(\frac{4}{3}\alpha^2\right)^\frac{1}{s}\bigg)
>  \bigg ( \frac{4}{3} \alpha^2\bigg )^\frac{s-2}{s}\bigg ( \frac{4}{3} \alpha^2\bigg )^\frac{2}{s}-\frac{4}{3} \alpha^2
=0.
\end{align*}
Therefore, $f(r)>0$ if $r>\frac{4}{3}\beta+ \sqrt{\frac{4}{3}\gamma}+\left(\frac{4}{3}\alpha^2\right)^\frac{1}{s}$. 
Since $f(X)\leq 0$, the inequality $X\le X_0$ is proved. 
We have so shown that there exists a positive constant $K_1$ such that
\begin{eqnarray*}
\sum_{k=1}^m\norm{w \mathfrak{r}^{nx}_{hk}}_\infty
\leq K_1\big (\alpha^2+\beta^s+\gamma^\frac{s}{2}\big ).
\end{eqnarray*}

Estimate \eqref{eq: stima-nucleo} now follows by plugging in the previous inequality the definition of $\alpha, \beta,\gamma$ and exploiting Theorem \ref{thm:int_time_lyap_fct} to estimate, for every $h=1,\ldots,m$ and $i=1,2$,
\begin{align*}
\sum_{k=1}^mM_{i,h,k}^x=\int_{a_0}^{b_0}((\T^{\mathcal D,P,n}(t)\bm\nu_i)(x))_hdt\leq \nu_i(0,x)\int_{a_0}^{b_0}e^{G_i(t)}dt, \qquad x\in \R^d. \qquad\qquad\qquad \qedhere
\end{align*}
\end{proof}

\subsection{Kernel estimates for the adjoint operator}
Here, we show that, under suitable assumptions, the results obtained in Sections \ref{sec:prliminaries} and \ref{sec:lyapunov} and in the first part of this section hold true for the (formal) adjoint operator $\bm\calA^{P,*}$ of $\bm\calA^P$, defined on smooth functions $\uu$ by
\begin{align*}
(\boldsymbol{\calA}^{P,*} \uu)_k
= & \mathrm{div} (Q^k\nabla  u_k) -\langle b^k,\nabla u_k\rangle -{\rm div}(b^k)u_k-((V^P)^*{\boldsymbol u})_k, \qquad\;\, k=1,\ldots,m.    
\end{align*}
Here, $(V^P)^*$ is the transpose matrix of $V^P$. We assume the following conditions on the coefficients of the operator $\bm\calA$.
\begin{hyp}
\label{hyp-base_adj}
\begin{enumerate}[\rm (i)]
\item 
Conditions $(i)$ and $(ii)$ in Hypotheses $\ref{hyp-base}$ are satisfied. Moreover, $b^h$ belong to $C^{1+\alpha}_{\rm loc}(\R^d;\R^d)$ for every $h\in\{1,\ldots,m\}$ and some $\alpha\in (0,1)$;
\item 
there exist a positive function $\varphi_*\in C^2(\R^d)$, blowing up as $|x|$ tends to infinity, and $\lambda_*\geq0$ such that $\bm\calA^{P,*}\bm\varphi_*\leq\lambda_*\bm\varphi_*$, where $(\bm\varphi_*)_h=\varphi_*$ for every $h=1,\ldots,m$;
\item 
the sum of the elements of each row of $(V^P)^*+{\rm div}(b^k){\rm Id}_{\R^m}$ is bounded from below on $\R^d$, i.e., there exists $M_*\in\R$ such that
\begin{align}
\sum_{k=1}^m v_{kh}^P(x)+{\rm div}(b^h(x))\geq M_*,\qquad\;\, x\in\R^d,\;\, h=1,\ldots,m.   
\label{segno_V_P_*}
\end{align}
\end{enumerate}    
\end{hyp}

Under Hypotheses \ref{hyp-base_adj}, for the operator $\bm\calA^{P,*}$ we recover analogous results to those obtained for the operator $\bm\calA^P$ in Section \ref{sec:prliminaries}. In particular, for every $\f\in C_b(\R^d)$ the Cauchy problem
\begin{align}
\label{prob_cauchy_2_*}
\left\{
\begin{array}{ll}
D_t{\uu}=\boldsymbol{\calA}^{P,*}\uu,   & {\rm in}~(0,\infty)\times \R^d, \\[1mm]
\uu(0,\cdot)=\f,     & {\rm in}~\R^d, 
\end{array}
\right.
\end{align}
admits a unique solution $\uu^P_*\in C([0,\infty)\times \R^d)\cap C^{1+\frac{\alpha}{2},2+\alpha}_{\rm loc}((0,\infty)\times \R^d)$, which is
bounded in each strip $[0,T]\times\R^d$, and, for every $t\geq0$, fulfills the estimate 
\begin{align*}
\|\uu^P_*(t,\cdot)\|_\infty\leq \sqrt me^{-M_*t}\max_{k=1,\ldots,m}\|f_k\|_{\infty}.
\end{align*}
Further, we denote by $(\bm S^P(t))$ the semigroup associated to problem \eqref{prob_cauchy_2_*} and, for every $h,k=1,\ldots,m$ and $t>0$, we denote by $p_{hk}^{P,*}(t,\cdot,\cdot)$ and $p_{hk}^{\mathcal D,P,n,*}(t,\cdot,\cdot)$ the nonnegative kernels associated to the operators $\bm S^P(t)$ and $\bm S^{\mathcal D,P,n}(t)$, respectively, where 
($\bm S^{\mathcal D,P,n}(t))$ is the semigroup associated to the realization of the operator $\bm\calA^{P,*}$ in $C_b(B(n))$ with homogeneous Dirichlet boundary conditions. 

\begin{remark}
\label{rmk:adjoint_op_lyap}
{\rm We stress that, under Hypotheses \ref{hyp-base_adj}, all the results in Sections \ref{sec:prliminaries} and \ref{sec:lyapunov} hold true replacing $\bm\calA^P$ with $\bm\calA^{P,*}$, $\T^P(t)$ with $\bm S^P(t)$, $\T^{\mathcal D,P,n}(t)$ with $\bm S^{\mathcal D,P,n}(t)$, for every $t\geq0$ and $n\in\N$, $\bm\varphi$ replaced by $\bm \varphi^*$ and $\bm\nu$ replaced by $\bm \nu^*$, where $\bm\nu^*$ is a time-dependent Lyapunov function for the operator $D_t+\bm\calA^{P,*}$ with respect to $\varphi_*$ and a suitable $g_*\in L^1(0,T)$.}
\end{remark}

We introduce another set of conditions, which allows us to prove the analogous of Theorem \ref{thm: stima-nucleo_n} for the kernels associated to the operator $\bm\calA^{P,*}$.
\begin{hyp}
\label{hyp-stime-nucleo_*}
Fix $T>0$, $x\in\R^d$ and $0<a_0<b_0<T$. Let ${\bm\nu^*_{1}}=(\nu^*_{1},\dots,\nu^*_{1})$ and ${\bm\nu^*_{2}}=(\nu^*_{2},\dots,\nu^*_{2})$ be two time-dependent Lyapunov functions for the operator ${\bm{\mathcal  L}^*}:=D_t+{\bm\calA}^{P,*}$ with respect to $\varphi^*$ and $g_1^*$ and $g_2^*$, respectively, such that $\bm\nu^*_{1}\leq\bm\nu^*_{2}$. Let $1\leq w\in C^{1,2}((0,T)\times\R^d)$ be a weight function such that
there exist $s>d+2$ and constants $c_{*,1},\ldots,c_{*,8}$, possibly depending on the interval $(a_0,b_0)$, with
\begin{tabbing}
\= {\rm (i)} $w\le c_{*,1}^{\frac{s}{2}}\nu^*_{1}$,
\qquad\qquad \qquad\qquad\qquad\qquad\qquad\= {\rm (ii)} $|Q^h\nabla w|\le c_{*,2}w^{\frac{s-1}{s}}(\nu^*_{1})^{\frac{1}{s}}$,\\[0.5em]
\> {\rm (iii)} $|{\rm div}(Q^h\nabla w)|\le c_{*,3}w^{\frac{s-2}{s}}(\nu^*_{1})^{\frac{2}{s}}$,
\> {\rm (iv)} $|D_t w|\le c_{*,4}w^{\frac{s-2}{s}}(\nu^*_{1})^{\frac{2}{s}}$,\\[0.5em]
\>{\rm (v)} $\displaystyle\bigg (\sum_{k=1}^mv_{hk}^2\bigg )^{\frac{1}{2}}+|{\rm div}(b^h)|\le c_{*,5} w^{-\frac{2}{s}}(\nu^*_{2})^{\frac{2}{s}}$, \>{\rm (vi)} $|b^h|\leq c_{*,6} w^{-\frac{1}{s}}(\nu^*_{2})^{\frac{1}{s}}$,\\[0.5em]
\> {\rm (vii)} $|Q^h|\leq c_{*,7} w^{-\frac{1}{s}}(\nu^*_{1})^{\frac{1}{s}}$ , \> {\rm (viii)} $|R^h|\le c_{*,8}w^{-\frac{2}{s}}(\nu^*_{1})^{\frac{2}{s}}$
\end{tabbing}
on $[a_0,b_0]\times \R^d$, for all $n\in\N$ and $h=1,\dots,m$.
\end{hyp}

Arguing as in the proof of Theorem \ref{thm: stima-nucleo_n} and taking Remark \ref{rmk:adjoint_op_lyap} into account, we deduce the following result.
\begin{corollary}
\label{coro:stima_peso_*}
Under Hypotheses $\ref{hyp-base_adj}$ and $\ref{hyp-stime-nucleo_*}$, there exists a positive constant $C$, depending only on $\eta^0_1,\ldots,\eta^0_m$, $d$ and $s$, such that for every $h,k=1,\ldots,m$ and every $(t,x,y)\in(a,b)\times \R^d\times \R^d$, with $a_0<a<b<b_0$, we get
\begin{align}
w(t,y)p^{P,*}_{hk}(t,x,y)
\le & C \mathscr H^*_{a,b}(x),
%\bigg [ \Big(c_{*,1}^\frac{s}{2}+\frac{c_{*,1}^\frac{s}{2}}{[(a-a_0)\wedge (b_0-b)]^\frac{s}{2}}+c_{*,2}^s+c_{*,3}^\frac{s}{2}+c_{*,4}^\frac{s}{2}+c_{*,1}^\frac{s}{4} c_{*,2}^\frac{s}{2}+c_{*,1}^\frac{s}{4}c_{*,7}^\frac{s}{2}+c_{*,7}^s+c_{*,8}^\frac{s}{2}\Big)\notag\\
% &\qquad\times\nu^*_{1}(0,x)\int_{a_0}^{b_0} e^{G^*_{1}(t)} dt+\Big(c_{*,1}^\frac{s}{4}c_{*,6}^\frac{s}{2}+c_{*,2}^\frac{s}{2}c_{*,6}^\frac{s}{2}+c_{*,5}^\frac{s}{2}+c_{*,6}^s\Big) \nu^*_{2}(0,x)\int_{a_0}^{b_0} e^{G^*_{2}(t)} dt\bigg ],
\label{post_*P}
\end{align}
where the function ${\mathscr H}_{a,b}^*$ is defined as the function ${\mathscr H}_{a,b}$ in \eqref{funzione_stima} with the constant $c_j$ being replaced by $c_j^*$ $(j=1,\ldots,8)$ and the functions $\nu_i$, $g_i$ and $G_i$ being respectively replaced by $\nu_i^*$, $g_i^*$ and $G_i^*$ $(i=1,2)$.

% \begin{align}
% &w(t,y)p^{P,*}_{hk}(t,x,y)\notag\\
% \le & C \bigg [ \Big(c_{*,1}^\frac{s}{2}+\frac{c_{*,1}^\frac{s}{2}}{(b_0-b)^\frac{s}{2}}+c_{*,2}^s+c_{*,3}^\frac{s}{2}+c_{*,4}^\frac{s}{2}+c_{*,1}^\frac{s}{4} c_{*,2}^\frac{s}{2}+c_{*,1}^\frac{s}{4}c_{*,7}^\frac{s}{2}+c_{*,7}^s+c_{*,8}^\frac{s}{2}\Big) \nu^*_{1}(0,x)\int_{a_0}^{b_0} e^{G^*_{1}(t)} dt\notag\\
% &\quad\;+\Big(c_{*,1}^\frac{s}{4}c_{*,6}^\frac{s}{2}+c_{*,2}^\frac{s}{2}c_{*,6}^\frac{s}{2}+c_{*,5}^\frac{s}{2}+c_{*,6}^s\Big) \nu^*_{2}(0,x)\int_{a_0}^{b_0} e^{G^*_{2}(t)} dt\bigg ],
% \label{post_*}
% \end{align}
% for any $h,k=1,\dots, m$, $(t,y)\in (a,b)\times \R^d$ and fixed $x\in\R^d$, where $G^*_{j}(t)=\displaystyle\int_0^tg^*_{j}(s)ds$ for every $t\in [0,T]$ and $j=1,2$.
\end{corollary}

If Hypotheses \ref{hyp-base}, \ref{hyp-stime-nucleo}, \ref{hyp-base_adj} and \ref{hyp-stime-nucleo_*} hold true, then we obtain the following relation between the kernels $p_{hk}^P$ and $p^{P,*}_{hk}$.

\begin{proposition}
\label{prop-2.7}
Assume that Hypotheses $\ref{hyp-base}$, $\ref{hyp-stime-nucleo}$, $\ref{hyp-base_adj}$ and $\ref{hyp-stime-nucleo_*}$ are satisfied. Then, for every $h,k=1,\ldots,m$, every $t\in [0,\infty)$ and $x,y\in\R^d$, it holds that $p_{hk}^{P}(t,x,y)=p_{kh}^{P,*}(t,y,x)$.
In particular, there exists a positive constant $C$, depending only on $\eta^0_1,\ldots,\eta^0_m$, $d$ and $s$, such that for every $h,k=1,\ldots,m$ and every $(t,x,y)\in(a,b)\times \R^d\times \R^d$, with $a_0<a<b<b_0$, we get
\begin{align}
w(t,y)^{\frac12}w(t,x)^{\frac12}p_{hk}^P(t,x,y)
\le & C(\mathscr H_{a,b}(x))^\frac12(\mathscr H^*_{a,b}(x))^{\frac12}.
%\bigg [ \Big(c_1^\frac{s}{2}+\frac{c_1^\frac{s}{2}}{[(a-a_0)\wedge (b_0-b)]^\frac{s}{2}}+c_2^s+c_3^\frac{s}{2}+c_4^\frac{s}{2}+c_1^\frac{s}{4} c_2^\frac{s}{2}+c_1^\frac{s}{4}c_7^\frac{s}{2}+c_7^s+c_8^\frac{s}{2}\Big)\nu_1(0,x)\int_{a_0}^{b_0} e^{G_1(t)} dt\notag\\
%&\quad\;\,+\Big(c_1^\frac{s}{4}c_6^\frac{s}{2}+c_2^\frac{s}{2}c_6^\frac{s}{2}+c_5^\frac{s}{2}+c_6^s\Big) \nu_2(0,x)\int_{a_0}^{b_0} e^{G_2(t)} dt\bigg ]^{\frac12}\notag\\
%& \times \bigg [ \Big(c_{*,1}^\frac{s}{2}+\frac{c_{*,1}^\frac{s}{2}}{[(a-a_0)\wedge (b_0-b)]^\frac{s}{2}}+c_{*,2}^s+c_{*,3}^\frac{s}{2}+c_{*,4}^\frac{s}{2}+c_{*,1}^\frac{s}{4} c_{*,2}^\frac{s}{2}+c_{*,1}^\frac{s}{4}c_{*,7}^\frac{s}{2}+c_{*,7}^s+c_{*,8}^\frac{s}{2}\Big)\notag\\
%&\qquad\times\nu^*_{1}(0,x)\int_{a_0}^{b_0} e^{G^*_{1}(t)} dt+\Big(c_{*,1}^\frac{s}{4}c_{*,6}^\frac{s}{2}+c_{*,2}^\frac{s}{2}c_{*,6}^\frac{s}{2}+c_{*,5}^\frac{s}{2}+c_{*,6}^s\Big) \nu^*_{2}(0,x)\int_{a_0}^{b_0} e^{G^*_{2}(t)} dt\bigg ]^{\frac{1}{2}},
\label{stima-agropoli}
\end{align}
%where $G_j,G_j^*$, $j=1,2$, are as in Corollaries $\ref{coro:stima_peso}$ and $\ref{coro:stima_peso_*}$.
\end{proposition}
\begin{proof}
From \cite[Theorem 9.5.5]{Frie64} and the arguments exploited in the proof of \cite[Theorem 3.4]{AngLorPal16}, it can be checked that $p_{hk}^{\mathcal D,P,n}(t,x,y)=p_{kh}^{\mathcal D,P,n,*}(t,y,x)
$ for every $t\geq0$, $h,k=1,\ldots,m$ and $x,y\in B(n)$. Since the sequences  $(p_{hk}^{\mathcal D,P,n}(t,x,y))$ and $(p_{kh}^{\mathcal D,P,n,*}(t,x,y))$ monotonically converge to $p_{hk}^{P}(t,x,y)$ and $p_{kh}^{P,*}(t,x,y)$, respectively, for every $(t,x,y)\in (0,\infty)\times\R^d\times \R^d$,
the first part of the statement follows at once.

Finally, estimate \eqref{stima-agropoli} follows from  \eqref{eq: stima-nucleo} and \eqref{post_*P}.
\end{proof}

\section{Examples}
\label{sec:examples}
In this section, we provide two classes of systems of parabolic elliptic operators to which the main results of this paper apply.
Before entering into details, we fix some notation.

If $\varrho^k_{ij}$ ($i,j=1,\ldots,d$, $k=1,\ldots,m$) are real constants, then we set
\begin{eqnarray*}
\varrho^k_{\min}=\min_{i=1,\ldots,d}\varrho_{ii}^k,\qquad\;\,\varrho^k_{\max}=\max_{i=1,\ldots,d}\varrho_{ii}^k,\qquad\;\,\underline{\varrho}=\min_{k=1,\ldots,m}\varrho^k_{\max},\qquad\;\,\overline{\varrho}=\max_{k=1,\ldots,m}\varrho^k_{\max}.
\end{eqnarray*}

Similarly, we denote by
$\xi^k_{\min}$ (resp. $\xi^k_{\max}$) the minimum (resp. the maximum) entry of the vector $\xi^k\in\R^d$ $(k=1,\ldots,m)$.  
Finally, we denote by $\widetilde c$ a universal constant, which may vary from line to line.

\subsection{Kernel estimates in case of polynomially growing coefficients}

Let us consider the operator $\bm{\calA}$, defined by \eqref{operatore}, where
\begin{align*}
q_{ij}^k(x)=\zeta_{ij}^k(1+|x|^2)^{\alpha_{ij}^k}, \qquad\;\, b_i^k(x)=-\eta_i^kx_i(1+|x|^2)^{\beta_i^k}, \qquad\;\, v_{hk}(x)=\theta_{hk}(1+|x|^2)^{\gamma_{hk}}
\end{align*}
for every $x\in\R^d$, $i,j=1,\ldots,d$ and $h,k=1,\ldots,m$, under the following conditions on the coefficients of the operator $\boldsymbol{\calA}$. 

\begin{hyp}\label{hp-polinomiale}
\begin{enumerate}[\rm(i)]
\item 
For every $i,j=1,\ldots,d$ and $h,k=1,\ldots,m$, $\alpha_{ij}^k=\alpha_{ji}^k$, $\beta_i^k$ and $\gamma_{hk}$ are nonnegative constants, whereas $\eta_i^k$ and $\theta_{kk}$ are positive constants; 
\item 
$\gamma_{hh}>\gamma_{hk}$ for every $h,k=1,\ldots,m$ with $h\neq k$;
\item 
for every $k=1,\ldots,m$ and $i,j=1,\ldots,d$, with $i\neq j$, $\alpha_{\min}^k>\alpha_{ij}^k$ and the matrix $Z^k$, whose entries are 
$z_{ii}=\zeta_{ii}$ and $z_{ij}=-|\zeta_{ij}|$ for every $i,j\in\{1,\ldots,d\}$, with $i\neq j$, is symmetric and positive definite for every $k=1,\ldots,m$;
\item 
for every $k=1,\ldots,m$, it holds that $
\max\{\gamma_{kk},\beta_{\min}^k\}>\alpha^k_{\max}-1$. %, \qquad \textrm{if $\alpha^{k}_{\max}\ge 1$ or $\beta^k_{\min}>0$}.
% and
% $\gamma_{kk}>1-\alpha_{\max}^k$, otherwise.
% $($Note that, in this second situation, condition \eqref{bayer-roma} is automatically satisfied.$)$.
\end{enumerate}
\end{hyp}

Under the above assumptions, Hypothesis \ref{hyp-base} are fulfilled. Indeed, for every $k=1,\ldots,m$ it follows that
$\langle Q^k(x)\xi,\xi\rangle\ge\eta_k(x)|\xi|^2$ for every $x,\xi\in\R^d$,
where $\eta_k(x)=\lambda_{Z^k}(1+|x|^2)^{\alpha^k_{\min}}$ and $\lambda_{Z^k}$ denotes the minimum eigenvalue of the matrix $Z^k$, which is positive due to Hypothesis \ref{hp-polinomiale}(iii). Further, the conditions $\theta_{hh}>0$ and $\gamma_{hh}>\gamma_{hk}$ for every $h,k\in\{1.\ldots,m\}$ with $h\neq k $ imply that also Hypothesis \ref{hyp-base}(iv) is satisfied.

We set $\varphi(x)=e^{\hat\varepsilon(1+|x|^2)^{\rho}}$ for every $x\in\R^d$, with $\rho>0$ such that 
\begin{align}
(i)~\max\{\gamma_{kk},\beta_{\min}^k+\rho\}+1\ge
2\rho+\alpha^k_{\max}, \qquad\;\, (ii)~\max\{\gamma_{kk},\beta_{\min}^k+\rho\}>\rho, \quad k=1,\ldots,m
\label{cond_lyap_fct_temp}
\end{align}
(such a value of $\rho$ exists
thanks to Hypothesis \ref{hp-polinomiale}(iv) and since $\gamma_{kk}>0$ for every $k\in\{1,\ldots,m\}$),
% \begin{align}
% \rho=\min_{k=1,\ldots,m}\left\{\max\left\{\frac12\big(\gamma_{kk}+1-\alpha^k_{\max}\big),\beta_{\min}^k+1-\alpha^k_{\max}\right\}\right\},    
% \label{cond_lyap_fct}
% \end{align}
%which is positive due to Hypothesis \ref{hp-polinomiale}(iv), 
and $\hat\varepsilon>0$ satisfying
\begin{align*}
%\label{cond_coeff_lyp_1}
\begin{cases}
\displaystyle
\hat\varepsilon<\bigg (\frac{\theta_{kk}}{4\rho^2\zeta^{k}_{\max}}\bigg )^{\frac{1}{2}}, & 
\gamma_{k k}>2\beta^k_{\min}+1-\alpha^k_{\max}, \\[3mm] 
\displaystyle\hat\varepsilon<\frac{\eta^{k}_{\min}}{2\rho\zeta^{k}_{\max}}, & 
\gamma_{k k}<2\beta^k_{\min}+1-\alpha^k_{\max}, \\[3mm]
\displaystyle \hat\varepsilon<\frac{
\eta^{k}_{\min}+\big((\eta^{k}_{\min})^2+4\zeta^{k}_{\max}\theta_{k k}\big )^{\frac{1}{2}}}{4\rho\zeta^{k}_{\max}}, & 
\gamma_{k k}=2\beta^k_{\min}+1-\alpha^k_{\max}, 
\end{cases}
\end{align*}
% where $\eta^k_{\min}=\min_{i=1,\ldots,d}\eta^k_{i}$,
% $\zeta^k_{\max}=\max_{i=1,\ldots,d}\zeta^k_{ii}$
% if $\max\{\gamma_{kk},\min_{i=1,\ldots,d}\beta_i^k+1\}=\gamma_{kk}$, or
% \begin{align*}
% 2\beta{\hat\varepsilon} \bigg(\max_{i=1,\ldots,d} \zeta_{ii}^k+\Big(\sum_{i\neq j}(\zeta_{ij}^k)^2\Big)^{1/2}\bigg)<\min_{i=1,\ldots,d}\eta_i^k,
% \end{align*}
% if $\max\{\gamma_{kk},\min_{i=1,\ldots,d}\beta_i^k+1\}=\min_{i=1,\ldots,d}\beta_i^k+1$, with 
%for every $k=1,\ldots,m$, 
for every index $k$ such that the equality holds in \eqref{cond_lyap_fct_temp}(i), and show that the function $\bm{\varphi}$, with all the components equal to $\varphi$, is a Lyapunov function for the operator $\bm\calA^P$. 

A simple computation reveals that
% For every $i,j=1,\ldots,d$ and $x\in\R^d$ we get
% \begin{align*}
% D_i\varphi(x)= & 2\beta{\hat\varepsilon} x_i(1+|x|^2)^{\beta-1}\varphi(x);\\[1mm]
% D_{ij}\varphi(x) = & 2\beta{\hat\varepsilon}  \delta_{ij}(1+|x|^2)^{\beta-1}\varphi(x)
% +4\beta(\beta-1){\hat\varepsilon} x_ix_j(1+|x|^2)^{\beta-2}\varphi(x) \\
% & + 4\beta^2{\hat\varepsilon}^2 x_ix_j(1+|x|^2)^{2\beta-2}\varphi(x).
% \end{align*}
% Computing $({\bm\calA}^P\bm{\varphi})_k$, $k=1,\ldots,m$, we get
$({\bm\calA}^P\bm{\varphi})_k
= (\mathfrak{a}_{1,k}
+\mathfrak{a}_{2,k}+\mathfrak{a}_{3,k})\varphi$ for every $k=1,\ldots,m$,
where
\begin{align*}
&\mathfrak{a}_{1,k}(x,\hat\varepsilon,\rho)=4\rho^2\hat\varepsilon^2\sum_{i,j=1}^d\zeta_{ij}^kx_ix_j(1+|x|^2)^{\alpha^k_{ij}+2\rho-2}+2\rho\hat\varepsilon\sum_{i=1}^d\zeta_{ii}^k(1+|x|^2)^{\alpha_{ii}^k+\rho-1}
\notag\\
&\phantom{\mathfrak{a}_1(x,\hat\varepsilon,\rho)=\;\;}+4\rho\hat\varepsilon\sum_{i,j=1}^d(\rho-1+\alpha_{ij}^k)\zeta_{ij}^kx_ix_j(1+|x|^2)^{\alpha_{ij}^k+\rho-2},
\\[1mm]    
&\mathfrak{a}_{2,k}(x,\hat\varepsilon,\rho)=-2\rho\hat\varepsilon\sum_{i=1}^d\eta_i^kx_i^2(1+|x|^2)^{\beta_i^k+\rho-1};
\notag 
\\[1mm]
&\mathfrak{a}_{3,k}(x)=
-\theta_{kk}(1+|x|^2)^{\gamma_{kk}}\bigg(1-\theta_{kk}^{-1}\sum_{h= 1, h\neq k}^m|\theta_{kh}|(1+|x|^2)^{\gamma_{kh}-\gamma_{kk}}\bigg)
\end{align*}
for every $x\in\R^d$. 
We fix $k\in\{1,\ldots,m\}$ and observe that
\begin{align}
\sum_{i,j=1}^d\zeta_{ij}^kx_ix_j(1+|x|^2)^{\alpha^k_{ij}+2\rho-2}=& \sum_{i=1}^d\zeta_{ii}^kx_i^2(1+|x|^2)^{\alpha^k_{ii}+2\rho-2}
+\sum_{i\neq j}\zeta_{ij}^kx_ix_j(1+|x|^2)^{\alpha^k_{ij}+2\rho-2}\notag\\
% \le &\zeta^k_{\max}|x|^2(1+|x|^2)^{\alpha^k_{\max}+2\rho-2}+C\sum_{i\neq j}|x|^2(1+|x|^2)^{\alpha^k_{ij}+2\rho-2}\\
\le
&\zeta^k_{\max}|x|^2(1+|x|^2)^{\alpha^k_{\max}+2\rho-2}+o((1+|x|^2)^{\alpha^k_{\max}+2\rho-1},\infty)
\label{a1k-a}
\end{align}
since, by Hypothesis \ref{hp-polinomiale}(iii), $\alpha^k_{\min}>\alpha^k_{ij}$ for every $i\neq j$.
Similarly,
\begin{align}
\sum_{i,j=1}^d(\rho\!-\!1\!+\!\alpha^k_{ij})\zeta_{ij}^kx_ix_j(1\!+\!|x|^2)^{\alpha_{ij}^k+\rho-2}
\le &\widetilde c(1\!+\!|x|^2)^{\alpha^k_{\max}+\rho-1}\!+\!o((1\!+\!|x|^2)^{\alpha^k_{\max}+\rho-1},\infty).
\label{a1k-b}
\end{align}
Hence, we can finally estimate
\begin{align}
\sum_{j=1}^3\mathfrak{a}_{j,k}(x,\hat\varepsilon,\rho)
\le  & 4\rho^2{\hat\varepsilon}^2\zeta^k_{\max}|x|^2(1+|x|^2)^{\alpha^k_{\max}+2\rho-2}-2\rho\hat\varepsilon\eta^k_{\min}|x|^2(1+|x|^2)^{\beta^k_{\min}+\rho-1}\notag\\
&-\theta_{kk}(1+|x|^2)^{\gamma_{kk}}+o((1+|x|^2)^{\alpha^k_{\max}+2\rho-1},\infty)
+o((1+|x|^2)^{\gamma_{kk}},\infty)
\label{luka}
\end{align}
for every $x\in\R^d$.
The right-hand side of \eqref{luka} is bounded from above in $\R^d$ due to the choice of $\rho$ and $\hat\varepsilon$.
% provided that 
% \begin{align}
% \max\{\gamma_{kk},\beta_{\min}^k+\rho\}+1\ge 2\rho+\alpha^k_{\max}, \qquad\;\, k=1,\ldots,m.
% \label{suzanne-vega}
% \end{align}
% Inequality \eqref{suzanne-vega} is solved by any 
% $\rho\le\rho_k:=\max\left\{\frac12\big(\gamma_{kk}+1-\alpha^k_{\max}\big),\beta_{\min}^k+1-\alpha^k_{\max}\right\}$ (note that $\rho_k$ is positive due to condition \eqref{bayer-roma}). 
% This is clear if $\rho<\rho_k$. On the other hand, if $\rho=\rho_k$ and $\alpha^k_{\max}+2\rho-1=\gamma_{kk}$ (which is equivalent to saying that $\gamma_{kk}>2\beta^k_{\min}+1-\alpha^k_{\max}$), the first and the third terms in the right-hand side of \eqref{luka} have the same asymptotic behaviour at $\infty$. The condition on $\hat\varepsilon$ shows that
% their sum diverges to $-\infty$ as $|x|$ tends to $\infty$. Similar arguments and the choice of $\hat\varepsilon$ imply that, also in the case when $\gamma_{kk}\le 2\beta^k_{\min}+1-\alpha^k_{\max}$, the right-hand side of \eqref{luka} diverges to $-\infty$ as $|x|$ tends to $\infty$.
% Summing up, if we take $\rho$ as in \eqref{cond_lyap_fct}, then we can determine $\lambda>0$ such that $({\bm\calA}^P\bm{\varphi})_k\leq \lambda \varphi$ in $\R^d$ for every $k=1,\ldots,m$.
We have so proved that $\bm{\varphi}$ is a Lyapunov function for the operator $\bm{\calA}^P$.

Next, we fix $T>0$ and $\sigma>0$ such that
\begin{align}
\frac{\rho(\sigma+1)}{\sigma}<\max\{\gamma_{kk}, \beta^k_{\rm min}+\rho\}, \qquad k=1,\ldots,m. 
\label{piove-ven}
\end{align}
Such a $\sigma$ exists thanks to condition \eqref{cond_lyap_fct_temp}(ii).
% \begin{equation}
% \sigma>\frac{\rho}{\rho+\underline \alpha-1}
% \label{cond-sigma}
% \end{equation}
We show that the function $\bm{\nu}:\R^d\to\R$, with 
$\nu_h(t,x)=\nu(t,x)=e^{\hat\varepsilon T^{-\sigma} t^\sigma(1+|x|^2)^\rho}$ for every  $t\in[0,T]$, $x\in\R^d$ and $h\in\{1,\ldots,m\}$, 
is a time-dependent Lyapunov function for the operator $D_t+\bm\calA^P$ with respect to $\varphi$ and a suitable function $g\in L^1(0,T)$. 
%We stress that Hypothesis \ref{hp-polinomiale}(iv) and the definition of $\rho$ imply that $\sigma$ is positive.

We fix $k\in\{1,\ldots,m\}$ and observe that 
\begin{align*}
D_t\nu(t,x)+(\bm\calA^P\bm\nu)_k(t,x)
= & \big (\sigma\varepsilon_T(1+|x|^2)^\rho t^{\sigma-1}
+\widetilde{\mathfrak{a}}_{1,k}(t,x,\rho)
+\widetilde{\mathfrak{a}}_{2,k}(t,x,\rho)
+\mathfrak{a}_{3,k}(x)\nu(t,x)
\end{align*}
for every $t\in(0,T]$ and $x\in\R^d$, where
$\varepsilon_T=\hat\varepsilon T^{-\sigma}$, $\widetilde{\mathfrak{a}}_{j,k}(t,x,\rho)=
\mathfrak{a}_{j,k}(x,\varepsilon_Tt^{\sigma},\rho)$ ($j=1,2$).

We write
$\sigma  (1+|x|^2)^\rho t^{\sigma-1}=\sigma  (1+|x|^2)^\rho t^{\sigma-\delta}t^{\delta-1}$, 
with $\delta\in(0,\sigma)$ to be properly fixed later, and apply the Young inequality $ab\leq \frac{a^p}{p}+\frac{b^q}{q}$, which holds true for $a,b\geq0$ and $\frac{1}{p}+\frac{1}{q}=1$, with
$a= (1+|x|^2)^\rho t^{\sigma-\delta}$, $b= t^{\delta-1}$, $p=\frac{\sigma}{\sigma-\delta}$ and $q=\frac\sigma\delta$.    
It follows that
\begin{align}
\sigma (1+|x|^2)^\rho t^{\sigma-1}
\leq (\sigma-\delta)t^\sigma(1+|x|^2)^\frac{\sigma\rho}{\sigma-\delta} +\delta t^{\frac{\sigma(\delta-1)}{\delta}},
\qquad\;\,t\in (0,T],\;\,x\in\R^d.
\label{young}
\end{align}

If we take $\delta>\frac{\sigma}{\sigma+1}$, then it follows that $\frac{\sigma(\delta-1)}{\delta}>-1$. Moreover, since 
$\frac{\sigma\rho}{\sigma-\delta}=\frac{\rho(\sigma+1)}{\sigma}$ when
$\delta=\frac{\sigma}{\sigma+1}$, from \eqref{piove-ven} it follows that we can fix  $\delta\in\left(\frac{\sigma}{\sigma+1},\sigma\right)$ such that
\begin{align*}
\frac{\rho\sigma}{\sigma-\delta}<\max\{\gamma_{kk},\beta_{\rm min}^k+\rho\}, \qquad k=1,\ldots,m.
\end{align*}
With this choice of $\delta$, arguing as for the Lyapunov function ${\bm\varphi}$ we get
%can determine a positive constant $C$ such that
\begin{align*}
& \varepsilon_T(\sigma-\delta)t^\sigma(1+|x|^2)^\frac{\sigma\rho}{\sigma-\delta}+\widetilde{\mathfrak{a}}_{1,k}(t,x,\rho)+ \widetilde{\mathfrak{a}}_{2,k}(t,x,\rho)
+\mathfrak{a}_{3,k}(x)\leq \widetilde c,\qquad\;\,(t,x)\in [0,T]\times\R^d.
\end{align*}
Thus, we conclude that $D_t\nu+(\bm\calA^P\bm\nu)_k
\leq g\nu$ in $(0,T]\times\R^d$, where 
$g(t)=\widetilde c+\varepsilon_T\delta t^{\frac{\sigma(\delta-1)}{\delta}}$ for every $t\in(0,T]$, and $\bm{\nu}$ is a Lyapunov function for the operator $D_t+\bm{\calA}^P$ with respect to $\varphi$ and $g$.

Finally, we fix $0<a_0<a<b<b_0<T$, $s>d+2$ and compute the constants $c_j$ $(j=1,\ldots,8)$ in Hypotheses \ref{hyp-stime-nucleo}, with
the weight function $w$ and the Lyapunov functions $\bm{\nu}_1$ and $\bm{\nu}_2$ given by
\begin{align*}
w(t,x)=e^{\varepsilon t^\sigma(1+|x|^2)^\rho} ,\qquad\;\,\nu_1(t,x)=e^{{\varepsilon}_1 t^\sigma(1+|x|^2)^\rho}, \qquad\;\,\nu_2(t,x)=e^{\varepsilon_2 t^\sigma(1+|x|^2)^\rho} 
\end{align*}
for every $(t,x)\in[0,T]\times \R^d$, with $0<\varepsilon<\varepsilon_1<\varepsilon_2\le\varepsilon_T$.

Fix $(t,x)\in [a_0,b_0]\times\R^d$.
Since $\varepsilon<\varepsilon_1$, we can take $c_1=1$.
To determine the remaining constants $c_j$, we take advantage of the estimate
$z^\gamma e^{-\tau z^\rho}
% \leq \tau^{-\frac{\gamma}{\rho}} \left(\frac{\gamma}{\rho}\right)^\frac{\gamma}{\rho} e^{-\frac{\gamma}{\rho}}=:
\le C(\gamma,\rho) \tau^{-\frac{\gamma}{\rho}}$ which holds true for every $z\in [0,\infty)$, every $\tau,\gamma>0$ and some positive constant
$C(\gamma,\rho)$.

Using such an estimate we get
\begin{align*}
\frac{|Q^h(x) \nabla w(t,x)|}{w(t,x)^\frac{s-1}{s}\nu_1(t,x)^\frac{1}{s}}
&\leq 2\varepsilon\rho t^\sigma (1+|x|^2)^{\rho-\frac{1}{2}} \sum_{i,j=1}^d |\zeta_{ij}^h|(1+|x|^2)^{\alpha_{ij}^h} e^{-\frac{\varepsilon_1-\varepsilon}{s}
%e^{-\frac{\varepsilon-\varepsilon_1}{s}
t^\sigma (1+|x|^2)^\rho}\\
&\leq \widetilde c\varepsilon\rho t^\sigma (1+|x|^2)^{\overline{\alpha}+\rho-\frac{1}{2}}  e^{-\frac{\varepsilon_1-\varepsilon}{s}
%e^{-\frac{\varepsilon-\varepsilon_1}{s}
t^\sigma (1+|x|^2)^\rho}\\
%\leq \widetilde c t^{- \frac{\sigma}{2\rho}(2\overline{\alpha}-1)} \\
\leq  &\widetilde c t^{\sigma-\frac{\sigma(2\overline \alpha+2\rho-1)^+}{2\rho}}
\le \widetilde c a_0^{- \frac{\sigma}{2\rho}(2\overline{\alpha}-1)^+}:=c_2.
\end{align*}
% Hence, we choose $c_2=\widetilde c a_0^{- \frac{\sigma}{2\rho}(2\overline{\alpha}-1)}$ if $\overline{\alpha}\geq \frac{1}{2}$ and 
% $c_2= \widetilde cT^{- \frac{\sigma}{2\rho}(2\overline{\alpha}-1)}$, otherwise.

To go further, we observe that the previous computations show that
$\mathrm{div}(Q^h(x)\nabla w(t,x))=\mathfrak{a}_{1,h}(x,\varepsilon t^{\sigma},\rho)$. Hence, taking \eqref{a1k-a} and \eqref{a1k-b} into account, we can estimate
\begin{align*}
|\mathrm{div}(Q^h(x)\nabla w(t,x))|
% =&2\varepsilon\rho t^\sigma \bigg | 2\sum_{i,j=1}^d (\alpha_{ij}^h\!+\!\rho\!-\!1)\zeta_{ij}^hx_ix_j(1+|x|^2)^{\alpha_{ij}^h+\rho-2}
% +\sum_{i=1}^d
% \zeta_{ii}^h(1\!+\!|x|^2)^{\alpha_{ii}^h+\rho-1}\\
% &\qquad\quad +2\varepsilon\rho t^{\sigma}\sum_{i,j=1}^d\zeta_{ij}^hx_ix_j(1+|x|^2)^{\alpha_{ij}^h+2\rho-2}\bigg |w(t,x)\\
% \le &\widetilde c t^\sigma \bigg (\sum_{i,j=1}^d |x|^2(1+|x|^2)^{\alpha_{ij}^h+\rho-2}
% +\sum_{i=1}^d
% (1+|x|^2)^{\alpha_{ii}^h+\rho-1}\\
% &\qquad\;+t^{\sigma}\sum_{i,j=1}^d|x|^2(1+|x|^2)^{\alpha_{ij}^h+2\rho-2}\bigg )w(t,x)\\
\le &\widetilde c(t^{\sigma}(1+|x|^2)^{\overline\alpha+\rho-1}+t^{2\sigma}(1+|x|^2)^{\overline\alpha+2\rho-1})w(t,x)
\end{align*}

Therefore, we obtain that
\begin{align*}
\frac{|\mathrm{div}(Q^h(x)\nabla w(t,x))|}{w(t,x)^\frac{s-2}{s}\nu_1(t,x)^\frac{2}{s}}
% \leq &\widetilde c\big (t^{\sigma} (1+|x|^2)^{\overline{\alpha}+\rho-1}+t^{2\sigma} (1+|x|^2)^{\overline{\alpha}+2\rho-1}\big )e^{-\frac{2(\varepsilon_1-\varepsilon)}{s}t^\sigma (1+|x|^2)^\rho} \\
\leq \widetilde c t^{\sigma-\frac{\sigma(\overline{\alpha}+2\rho-1)^+}{\rho}}
\le \widetilde c a_0^{-\frac{\sigma}{\rho}(\overline{\alpha}-1)^+}=:c_3.
\end{align*}
% Thus, we choose $c_3=\widetilde ca_0^{-\frac{\sigma}{\rho}(\overline{\alpha}-1)}$ if $\overline{\alpha}\geq 1$ and $c_3=\widetilde cT^{-\frac{\sigma}{\rho}(\overline{\alpha}-1)}$ otherwise.
Arguing similarly, we can show that (here, we set $\gamma_{\rm max}:=\max\{\gamma_{kk}:k=1,\ldots,m\}$)
\begin{align*}
%%4
\frac{|D_t w(t,x)|}{w(t,x)^\frac{s-2}{s}\nu_1(t,x)^\frac{2}{s}}
=& \sigma \varepsilon t^{\sigma-1} (1+|x|^2)^\rho e^{-\frac{2(\varepsilon_1-\varepsilon)}{s}t^\sigma (1+|x|^2)^\rho}
\leq \widetilde c a_0^{-1}=:c_4,\\
%%5
\frac{|V^h(x)|}{w(t,x)^{-\frac{2}{s}}\nu_2(t,x)^\frac{2}{s}}
% \leq &\sum_{k=1}^m |\theta_{kh}|(1+|x|^2)^{\gamma_{kh}} e^{-\frac{2(\varepsilon_2-\varepsilon)}{s}t^\sigma (1+|x|^2)^\rho}\\
\leq &\widetilde c (1+|x|^2)^{\gamma_{\mathrm{max}}} e^{-\frac{2(\varepsilon_2-\varepsilon)}{s}t^\sigma (1+|x|^2)^\rho}
\leq \widetilde c a_0^{-\frac{\sigma}{\rho}\gamma_\mathrm{max}}=:c_5,\\
%%6
\frac{|b^h(x)|}{w(t,x)^{-\frac{1}{s}}\nu_2(t,x)^\frac{1}{s}}
% \leq &\sum_{i=1}^d \eta_i^h |x_i| (1+|x|^2)^{\beta_i^h}e^{-\frac{\varepsilon_2-\varepsilon}{s}t^\sigma (1+|x|^2)^\rho}\\
\leq &\widetilde c (1+|x|^2)^{\overline{\beta}+\frac{1}{2}}e^{-\frac{\varepsilon_2-\varepsilon}{s}t^\sigma (1+|x|^2)^\rho}
\le\widetilde c a_0^{-\frac{\sigma}{2\rho}(2\overline{\beta}+1)}=:c_6,\\
%%7
\frac{|Q^h(x)|}{w(t,x)^{-\frac{1}{s}}\nu_1(t,x)^\frac{1}{s}}
% \leq &\sum_{i,j=1}^d |\zeta_{ij}^h| (1+|x|^2)^{\alpha_{ij}^h}e^{-\frac{\varepsilon_1-\varepsilon}{s}t^\sigma (1+|x|^2)^\rho}\\
\leq &\widetilde c(1+|x|^2)^{\overline{\alpha}}e^{-\frac{\varepsilon_1-\varepsilon}{s}t^\sigma (1+|x|^2)^\rho}
\leq \widetilde c a_0^{-\frac{\sigma}{\rho}\overline{\alpha}}=:c_7,\\
%%8
%\frac{|\nabla Q^h(x)|}{w(t,x)^{-\frac{2}{s}}\nu_1(t,x)^\frac{2}{s}}
\frac{|R^h(x)|}{w(t,x)^{-\frac{2}{s}}\nu_1(t,x)^\frac{2}{s}}
% \leq & 2\sum_{i,j=1}^d |\zeta_{ij}^h| \alpha_{ij}^h |x_i| (1+|x|^2)^{\alpha_{ij}^h-1} e^{-\frac{2(\varepsilon_1-\varepsilon)}{s}t^\sigma (1+|x|^2)^\rho}\\
\leq &\widetilde c (1+|x|^2)^{\overline{\alpha}-\frac{1}{2}}e^{-\frac{2(\varepsilon_1-\varepsilon)}{s}t^\sigma (1+|x|^2)^\rho}
\leq\widetilde c a_0^{-\frac{\sigma}{2\rho}(2\overline{\alpha}-1)^+}=:c_8,
\end{align*}
where $V^h$ denotes the $h$-th column of the matrix $V$ and $R^h$ is the matrix whose entries are $D_iq^h_{ij}$.
% To sum up, we set
% \begin{align*}
% &c_1=1, 
% &&c_3=\widetilde ca_0^{-\frac{\sigma}{\rho}(\overline{\alpha}-1)},
% &&c_4=\widetilde c a_0^{-1},
% && c_5=\widetilde c a_0^{-\frac{\sigma}{\rho}\gamma_\mathrm{max}}, \\
% &c_6=\widetilde c a_0^{-\frac{\sigma}{2\rho}(2\overline{\beta}+1)},
% && c_7=\widetilde ca_0^{-\frac{\sigma}{\rho}\overline{\alpha}},
% &&c_8=\widetilde c a_0^{-\frac{\sigma}{2\rho}(2\overline{\alpha}-1)^+}
% \end{align*}
% and  $c_2=a_0^{- \frac{\sigma}{2\rho}(2\overline{\alpha}-1)}$ if $\overline{\alpha}\geq \frac{1}{2}$, $c_2= b_0^{- \frac{\sigma}{2\rho}(2\overline{\alpha}-1)}$ otherwise.

We now apply Theorem \ref{thm: stima-nucleo_n}.
For this purpose, we preliminarily 
set
$a_0=\frac{1}{8}t$, $a=\frac{1}{4}t$, $b=\frac{1}{2}t$ and $b_0=\frac{3}{4}t$ and note that
$\displaystyle
\int_{a_0}^{b_0}e^{G_j(t)}dt
\leq (b_0-a_0)r_j= \frac{5t}{8}r_j$
for $j=1,2$ and some positive constants $r_1$ and $r_2$. 
Combining the above estimate with the values of the constants $c_1,\dots, c_8$, from \eqref{eq: stima-nucleo} we deduce that for every $(t,x,y)\in(0,T]\times \R^d\times \R^d$,
\begin{align*}
|p_{hk}(t,x,y)|
\leq Ct^{1-\lambda s}e^{\displaystyle -\varepsilon t^\sigma(1+|y|^2)^\rho}, \qquad h,k=1,\ldots,m,
\end{align*}
where $s>d+2$ and
\begin{align}
\lambda:=\max\left\{\frac12, \frac{\sigma}{\rho}\overline\alpha,\frac{\sigma}{2\rho}\gamma_{\max}, \frac{\sigma}{2\rho}(2\overline\beta+1)\right\}, \qquad \textrm{if $\overline\alpha\le\frac{1}{2}$},
\label{lambda}
\end{align}
whereas
\begin{align}
\lambda:=\max\left\{\frac12, \frac{\sigma}{\rho}\overline\alpha, \frac{\sigma}{2\rho}(\overline\alpha+\overline\beta), \frac{\sigma}{2\rho}\gamma_{\max}, \frac{\sigma}{2\rho}(2\overline\beta+1)\right\}, \qquad \textrm{if $\overline\alpha>\frac{1}{2}$}.  
\label{lambda-1}
\end{align}

To get a decay estimate of the function $p_{hk}(t,\cdot,y)$, we consider the operator $\bm\calA^{P,*}$ and
assume the following conditions.

\begin{hyp}\label{hp-polinomiale-1}
\begin{enumerate}[\rm(i)]
\item 
Assumptions $(i)$ and $(iii)$ in Hypotheses $\ref{hp-polinomiale}$ are satisfied;
\item 
for every $k\in\{1,\ldots,m\}$ it holds that
$\gamma_{kk}>\max\{\beta^k_{\max},\gamma_{hk},\alpha^k_{\max}-1: h\in\{1,\ldots,m\}\setminus\{k\}\}$.
\end{enumerate}
\end{hyp}

Since Hypothesis \ref{hp-polinomiale-1}(ii) implies that $\gamma_{kk}>\max_{h=1,\ldots,m, h\neq k}\{\beta^k_{\max},\gamma_{hk}\}$  for every $k=1,\ldots,m$, it follows that condition \eqref{segno_V_P_*} is satisfied. 
The same assumption guarantees that
the function $\varphi_*:\R^d\to\R$, defined by $\varphi_*(x):=e^{\hat\varepsilon_*(1+|x|^2)^{\rho_*}}$ for  every $x\in\R^d$, is a Lyapunov function for $\bm\calA^{P,*}$ with $\rho_*$ being a positive constant satisfying the conditions
\begin{align}
(i)~\gamma_{kk}\geq \max\{\alpha_{\rm max}^k+2\rho_*-1,\beta_{\rm max}^k+\rho_*\}, \qquad\;\,(ii)~\gamma_{kk}>\rho_*, \quad k=1,\ldots,m,
\label{cond_1_lyap_adj_pol_2}
\end{align}
(such a $\rho_*$ exists due to Hypothesis \ref{hp-polinomiale-1}(ii))
% \begin{align}
% \label{cond_1_lyap_adj_pol}
% \rho_*=\min_{k=1,\ldots,m}\left\{\frac12(\gamma_{kk}+1-\alpha^k_{\max}), \gamma_{kk}-\beta^k_{\max}\right\},    
% \end{align}
and $\hat\varepsilon_*>0$ fulfills
\begin{align*}
\begin{cases}
\displaystyle
\hat\varepsilon_*<\left(\frac{\theta_{kk}}{4\rho_*^2\zeta^k_{\max}}\right)^{\frac{1}{2}}, & 
\gamma_{kk}>2\beta^k_{\max}-\alpha^k_{\max}+1, \\[3mm]
\displaystyle 
\hat\varepsilon_*<\frac{\theta_{kk}}{2\rho_*\eta^k_{\max}}, &  
\gamma_{kk}<2\beta^k_{\max}-\alpha^k_{\max}+1, \\[3mm]
\displaystyle \hat\varepsilon_*<\frac{-\eta^k_{\max}+\big((\eta^k_{\max})^2+4\zeta^k_{\max}\theta_{kk}\big )^{\frac{1}{2}}}{4\rho_*\zeta^k_{\max}}, &  
\gamma_{kk}=2\beta^k_{\max}-\alpha^k_{\max}+1\end{cases}
\end{align*}
for every $k\in\{1,\ldots,m\}$ such that the equality in \eqref{cond_1_lyap_adj_pol_2}(i) is realized. Indeed, in this case we get $({\bm\calA}^{P,*}\bm{\varphi})_k
= (\mathfrak{a}_{1,k}(\cdot,\hat\varepsilon_*,\rho_*)
+\widehat{\mathfrak{a}}_{2,k}(\cdot,\hat\varepsilon_*,\rho_*)+\widehat{\mathfrak{a}}_{3,k})\varphi$ for every $k=1,\ldots,m$,
where $\widehat{\mathfrak{a}}_{2,k}(\cdot,\hat\varepsilon_*,\rho_*)$ is defined as $\mathfrak{a}_{2,k}(\cdot,\hat\varepsilon_*,\rho_*)$, with the second term being replaced by its opposite, and
\begin{align*} 
\widehat{\mathfrak{a}}_{3,k}(x)=
-\theta_{kk}(1+|x|^2)^{\gamma_{kk}}\bigg(&1-\theta_{kk}^{-1}\sum_{h= 1, h\neq k}^m|\theta_{hk}|(1+|x|^2)^{\gamma_{hk}-\gamma_{kk}}-\sum_{i=1}^d\eta^k_i(1+|x|^2)^{\beta^k_i-\gamma_{kk}}\\
&-
2\sum_{i=1}^d\beta_i^k\eta_i^kx_i^2(1+|x|^2)^{\beta_i^k-\gamma_{kk}-1}
\bigg)
\end{align*}
for every $x\in\R^d$. It turns out that, for every $x\in\R^d$,
\begin{align*}
(\bm\calA^{P,*}\bm\varphi)_k(x)
\le  & \big (4\rho_*^2\hat\varepsilon_*^2\zeta^k_{\max}(1+|x|^2)^{\alpha^k_{\max}+2\rho_*-1}+2\rho_*\hat\varepsilon_*\eta^k_{\max}(1+|x|^2)^{\beta^k_{\max}+\rho_*}-\theta_{kk}(1+|x|^2)^{\gamma_{kk}}\notag\\
&+o((1+|x|^2)^{\alpha^k_{\max}+2\rho_*-1},\infty)
+o((1+|x|^2)^{\gamma_{kk}},\infty)\big )\varphi(x).
\end{align*}
The choice of $\hat\varepsilon_*$ and  condition \eqref{cond_1_lyap_adj_pol_2}(i) imply that
$\bm{\varphi}_*$ is a Lyapunov function for $\bm\calA^{P,*}$. 

Finally, arguing as above, we deduce that the function $\bm{\nu}_*:[0,T]\times\R^d\to\R$, with all the components equal to $\nu_{*}(t,x)=e^{\hat\varepsilon_{*}T^{-\sigma_*}t^{\sigma_*}(1+|x|^2)^{\rho_*}}$ for every $(t,x)\in[0,T]\times \R^d$, is a time-dependent Lyapunov function for the operator $D_t+\bm\calA^{P,*}$ with respect to $\varphi_*$ and some $g_*\in L^1(0,T)$, if we assume that   
$\sigma_*>\frac{\rho_*}{\gamma_{\min}-\rho_*}$,
where $\gamma_{\min}=\min_{k=1,\dots,m}\gamma_{kk}$.
Note that condition \eqref{cond_1_lyap_adj_pol_2}(ii) implies that $\rho_*<\gamma_{\min}$.
Choosing 
\begin{align*}
w^*(t,x)=e^{\varepsilon_* t^{\sigma_*}(1+|x|^2)^{\rho_*}} ,\qquad\;\,\nu_1^*(t,x)=e^{{\varepsilon}_1 t^{\sigma_*}(1+|x|^2)^{\rho_*}}, \qquad\;\,\nu_2^*(t,x)=e^{\varepsilon_2 t^{\sigma_*}(1+|x|^2)^{\rho_*}} 
\end{align*}
for every $(t,x)\in[0,T]\times \R^d$, with $0<\varepsilon_*<\varepsilon_1<\varepsilon_2\leq T^{-\sigma_*}\hat\varepsilon_*$, and recalling that $\gamma_{kk}>\beta^k_{\max}$ for every $k=1,\ldots,m$, same computations as in the first part of this example show that we can take 
the constant $c_{*,j}$ defined
as the corresponding constant
$c_j$ $(j=1,\ldots,8)$, with $(\sigma,\rho)$ being replaced by
$(\sigma_*,\rho_*)$.

% \begin{align*}
% &c_{*,1}=1, 
% && c_{*,2}=\widetilde ca_0^{- \frac{\sigma_*}{2\rho_*}(2\overline{\alpha}-1)^+}, 
% && c_{*,3}=\widetilde ca_0^{-\frac{\sigma_*}{\rho_*}(\overline{\alpha}-1)^+}, 
% && c_{*,4}=\widetilde c a_0^{-1}, \\
% & c_{*,5}=\widetilde c a_0^{-\frac{\sigma_*}{\rho_*}\gamma_{\max}}, 
% && c_{*,6}=\widetilde c a_0^{-\frac{\sigma_*}{2\rho_*}(2\overline{\beta}+1)},
% && c_{*,7}=\widetilde ca_0^{-\frac{\sigma_*}{\rho_*}\overline{\alpha}},
% &&c_{*,8}=\widetilde c a_0^{-\frac{\sigma_*}{2\rho_*}(2\overline{\alpha}-1)^+}.
% \end{align*}
% $c_{*,2}=\widetilde ca_0^{- \frac{\sigma_*}{2\rho_*}(2\overline{\alpha}-1)}$, if $\overline{\alpha}\geq \frac{1}{2}$, and $c_{*,2}= \widetilde cT^{- \frac{\sigma_*}{2\rho_*}(2\overline{\alpha}-1)}$ , otherwise, whereas we choose 
% $c_{*,3}=\widetilde ca_0^{-\frac{\sigma_*}{\rho_*}(\overline{\alpha}-1)}$ if $\overline{\alpha}\geq 1$ and $c_{*,3}=\widetilde cT^{-\frac{\sigma_*}{\rho_*}(\overline{\alpha}-1)}$, otherwise.

From Corollary \ref{coro:stima_peso_*} we conclude that, for every $(t,x,y)\in(0,T]\times \R^d\times \R^d$, 
\begin{align*}
|p_{hk}^{P,*}(t,x,y)|
\leq Ct^{1-\lambda_* s}e^{\displaystyle -\varepsilon_* t^{\sigma_*}(1+|y|^2)^{\rho_*}},
\qquad h,k=1,\ldots,m,
\end{align*}
where $s$ is any number larger than $d+2$ and $\lambda_*$ is defined as $\lambda$ (see \eqref{lambda} and \eqref{lambda-1}), with $\sigma$ and $\rho$ being replaced, respectively, by $\sigma_*$ and $\rho_*$.
Summing up, under the conditions
\begin{hyp}
%\label{hp-polinomiale-fin}
\begin{enumerate}[\rm(i)]
\item 
for every $i,j=1,\ldots,d$ and $h,k=1,\ldots,m$,  $\alpha_{ij}^k=\alpha_{ji}^k$, $\beta_i^k$ and $\gamma_{kh}$ are nonnegative constants, whereas $\eta_i^k$ and $\theta_{kk}$ are positive constants;
\item 
for every $k=1,\ldots,m$ it holds that $\alpha_{\min}^k> \max_{i\neq j}\alpha_{ij}^k$ and the matrix $Z^k$, whose entries are 
$z_{ii}=\zeta_{ii}$ and $z_{ij}=-|\zeta_{ij}|$ for every $i,j\in\{1,\ldots,d\}$, with $i\neq j$, is symmetric and positive definite;
% the matrix $Z^k=(\zeta_{ij}^k)$ is symmetric and positive definite;
\item 
% \begin{eqnarray*}
% \gamma_{kk}>
% \left\{
% \begin{array}{ll}
% \max\{\gamma_{kh},\gamma_{hk},\beta^k_{\max},\alpha^k_{\max}-1\} & {\rm if}~\underline{\alpha}>1,\\[1mm] 
% \max\{\gamma_{kh},\gamma_{hk},\beta^k_{\max},\alpha^k_{\max}-2\underline{\alpha}+1\} & {\rm if}~\underline{\alpha}\le 1,\\[1mm] 
% \end{array}
% \right.
% \end{eqnarray*}
$\gamma_{kk}>\max\{\gamma_{kh},\gamma_{hk},\beta^k_{\max},\alpha^k_{\max}-1: h\in\{1,\ldots,m\}\setminus\{k\}\}$,
%if $\min\{\alpha^k_{\max}-1,\beta^k_{\min}\}>0$,
%, whereas 
% $\gamma_{kk}>\max\{\gamma_{kh},\gamma_{hk},\beta^k_{\max},1-\alpha^k_{\max}: h\in\{1,\ldots,m\}\setminus{k}\}$, otherwise,
%\item 
%there exist no nontrivial subsets $K\subset\{1,\ldots,m\}$ such that $\theta_{kh}=0$ for every $k\in K$ and $h\notin K$.
\end{enumerate}
\end{hyp}
\noindent
from Proposition \ref{prop-2.7}, the following kernel estimates are satisfied:
\begin{eqnarray*}
|p_{hk}(t,x,y)|
\leq Ct^{1-(\lambda+\lambda_*)\frac{s}{2}}e^{-\frac{\varepsilon}{2} t^\sigma(1+|y|^2)^\rho}e^{-\frac{\varepsilon_*}{2} t^{\sigma_*}(1+|x|^2)^{\rho_*}}, \qquad\;\, (t,x,y)\in(0,T]\times \R^d\times \R^d,
%\qquad\;\,\textrm{a.e. }y\in\R^d
\end{eqnarray*}
for every $h,k\in\{1,\ldots,m\}$, every $s>d+2$ and some positive constant $C=C(T)$. 
% Here, $\lambda$ and $\lambda_*$ are given by \eqref{lambda} and \eqref{lambda-star}, $\rho$ and $\rho_*$ are given by \eqref{cond_lyap_fct} and \eqref{cond_1_lyap_adj_pol}, $\sigma$ and $\sigma_*$ are given by \eqref{piove-ven} and \eqref{cond-sigma-star}, respectively,
% whereas $\varepsilon$ and $\varepsilon_*$ are arbitrarily fixed in the intervals $(0,T^{-\sigma}\hat\varepsilon)$ and $(0,T^{-\sigma_*}\hat\varepsilon_*)$, respectively, where $\hat\varepsilon$ and $\hat\varepsilon_*$ are defined by \eqref{cond_coeff_lyp_1} and 
% \eqref{cond_2_lyap_adj_pol}, respectively.

\subsection{Kernel estimates in case of exponentially growing coefficients}

We consider the operator $\bm{\calA}$ defined by \eqref{operatore} with
\begin{align*}
q_{ij}^k(x)&=\zeta_{ij}^ke^{(1+|x|^2)^{\alpha_{ij}^k}},\qquad\;\, b_i^k(x)=-\eta_i^kx_ie^{(1+|x|^2)^{\beta_i^k}},\qquad\;\, v_{hk}(x)=\theta_{hk}e^{(1+|x|^2)^{\gamma_{hk}}}
\end{align*}
for every $x\in\R^d$, $i,j=1,\dots,d$ and $h,k=1,\dots,m$, and we assume the following conditions.

\begin{hyp}
\label{hyp-exp-1}
\begin{enumerate}[\rm (i)]
\item 
Conditions $(i)$ to $(iii)$ in Hypotheses $\ref{hp-polinomiale}$ are satisfied;
\item 
for every $k=1,\ldots,m$, it holds that
$\max\{\beta^k_\mathrm{min},\gamma_{kk}\}>\alpha^k_{\max}$.
\end{enumerate}    
\end{hyp}

As a consequence of Hypothesis \ref{hp-polinomiale}(iii), we can estimate 
\begin{align*}
\langle Q^k(x)\xi,\xi\rangle
\geq  e^{(1+|x|^2)^{\alpha_{\min}^k}}\langle Z^k\xi,\xi\rangle\ge\lambda_{Z^k}e^{(1+|x|^2)^{\alpha_{\min}^k}}
|\xi|^2 =:  \eta_k(x)|\xi|^2,\qquad\;\,x,\xi\in\R^d.
\end{align*}
Hence, Hypothesis \ref{hyp-base}(ii) is verified. It is easy to check Hypothesis \ref{hyp-base}(iv) too.

We consider the function $\varphi:\R^d\to\R$, defined by
$\varphi(x)=\displaystyle\exp\bigg (\hat\varepsilon \int_0^{1+|x|^2} e^\frac{\tau^\rho}{2}d\tau\bigg )$ for every $x\in\R^d$,
with $\hat{\varepsilon}>0$ and
$0<\rho<\max\{\beta^k_\mathrm{min},\gamma_{kk}\}$ for every $k=1,\ldots,m$, and prove that the function $\bm{\varphi}:\R^d\to\R^m$, with all the components equal to $\varphi$, is a Lyapunov function for the operator $\bm\calA^P$. 
Since
$({\bm\calA}^P\bm{\varphi})_k
=(\mathfrak{a}_{1,k}+ \mathfrak{a}_{2,k}+\mathfrak{a}_{3,k})\varphi$ for all $k=1,\ldots,m$, where
\begin{align*}
&\mathfrak{a}_{1,k}(x,\hat\varepsilon)=
2\hat\varepsilon\sum_{i=1}^d\zeta_{ii}^ke^{(1+|x|^2)^{\alpha_{ii}^k}+\frac12(1+|x|^2)^{\rho}}+4\hat\varepsilon^2\sum_{i,j=1}^d\zeta_{ij}^kx_ix_je^{(1+|x|^2)^{\alpha_{ij}^k}+(1+|x|^2)^{\rho}}
\notag\\
&\phantom{\mathfrak{a}_{1,k}(x,\hat\varepsilon)=}+2\hat\varepsilon\sum_{i,j=1}^d\zeta_{ij}^kx_ix_j[\rho(1+|x|^2)^{\rho-1}+2\alpha^k_{ij}
(1+|x|^2)^{\alpha^k_{ij}-1}]e^{(1+|x|^2)^{\alpha_{ij}^k}+\frac12(1+|x|^2)^{\rho}};
\\[1mm]    
&\mathfrak{a}_{2,k}(x,\hat\varepsilon)=-2\hat\varepsilon\sum_{i=1}^d\eta_i^kx_i^2e^{(1+|x|^2)^{\beta_{i}^k}+\frac12(1+|x|^2)^{\rho}};
%\label{a2k-exp}
\notag
\\[1mm]
&\mathfrak{a}_{3,k}(x)=-\sum_{h=1}^m\theta_{kh}e^{(1+|x|^2)^{\gamma_{kh}}}
\end{align*}
for every $x\in\R^d$,
the conditions on the parameters imply that 
\begin{align}
\sum_{j=1}^2\mathfrak{a}_{j,k}(x,\hat\varepsilon)
+\mathfrak{a}_{3,k}(x)
\le & 4\hat\varepsilon^2\zeta_\mathrm{max}^k|x|^2e^{(1+|x|^2)^{\alpha_\mathrm{max}^k}+(1+|x|^2)^{\rho}}
-2\hat\varepsilon\eta^k_{\min}|x|^2e^{(1+|x|^2)^{\beta_{\min}^k}+\frac{1}{2}(1+|x|^2)^{\rho}}\notag\\
&-\theta_{kk}e^{(1+|x|^2)^{\gamma_{kk}}}+o(e^{(1+|x|^2)^{\alpha_{\max}^k}+(1+|x|^2)^{\rho}},\infty)
+o(e^{(1+|x|^2)^{\gamma_{kk}}},\infty)
\label{roxette}
\end{align}
for every $x\in\R^d$. Hypothesis \ref{hyp-exp-1}(ii) and the choice of $\rho$ show that the
right-hand side of \eqref{roxette} is bounded from above over $\R^d$, so that there exists a positive constant $\lambda$ such that $({\bm\calA}^P\bm{\varphi})_k\leq \lambda\varphi$ on $\R^d$ for every $k=1,\ldots,m$.

Next, we fix $T>0$, consider the function $\nu:[0,T]\times\R^d\to\R$, defined by
\begin{eqnarray*}
\nu(t,x)=\exp\bigg (\varepsilon_T t^\sigma \int_0^{1+|x|^2} e^\frac{\tau^\rho}{2}d\tau\bigg ),\qquad\;\,(t,x)\in [0,T]\times\R^d,
\end{eqnarray*}
with $\sigma>0$ and $\varepsilon_T= T^{-\sigma}\hat\varepsilon$,
and show that the function ${\bm \nu}:[0,T]\times\R^d\to\R^m$, with all the components equal to $\nu$, is a time-dependent Lyapunov function for the operator $D_t+{\bm\calA}^P$ with respect to $\varphi$ and $g$, for a suitable function $g\in L^1(0,T)$. For this purpose, we fix $k\in\{1,\ldots,m\}$ and observe that
\begin{align*}
D_t\nu(t,x)+({\bm\calA}^P\bm{\nu})_k(t,x)
=\bigg(\sigma\varepsilon_T t^{\sigma-1}\int_0^{1+|x|^2}e^{\frac{\tau^\rho}{2}}d\tau+\widetilde{\mathfrak{a}}_{1,k}(t,x)+\widetilde{\mathfrak{a}}_{2,k}(t,x)+\mathfrak{a}_{3,k}(x)\bigg)\nu(t,x)  
\end{align*}
for every $(t,x)\in (0,T]\times\R^d$, where $\widetilde{\mathfrak a}_{j,k}(t,x)=\mathfrak{a}_{j,k}(x,\varepsilon_Tt^{\sigma})$ ($j=1,2$).

In view of \eqref{roxette} and observing that
$\int_0^{1+|x|^2}e^{\frac{\tau^\rho}{2}}d\tau\le (1+|x|^2)e^{\frac{1}{2}(1+|x|^2)^{\rho}}$ for every $x\in\R^d$, we can estimate 
\begin{align*}
D_t\nu(t,x)\!+\!({\bm\calA}^P\bm{\nu})_k(t,x)
\le \bigg (&\sigma \varepsilon_Tt^{\sigma-1}(1+|x|^2)e^{\frac{1}{2}(1+|x|^2)^{\rho}}
\!+\!4\varepsilon_T^2T^{\sigma}t^{\sigma}\zeta_\mathrm{max}^k|x|^2e^{(1+|x|^2)^{\alpha_\mathrm{max}^k}+(1+|x|^2)^{\rho}}\\
&-2\varepsilon_Tt^{\sigma}\eta^k_{\min}|x|^2e^{(1+|x|^2)^{\beta_{\min}^k}+\frac{1}{2}(1+|x|^2)^{\rho}}-\theta_{kk}e^{(1+|x|^2)^{\gamma_{kk}}}\notag\\
&+t^{\sigma}o(e^{(1+|x|^2)^{\alpha_{\max}^k}+(1+|x|^2)^{\rho}},\infty)
+o(e^{(1+|x|^2)^{\gamma_{kk}}},\infty)\bigg )\nu(t,x)
\end{align*}
for every $(t,x)\in (0,T]\times\R^d$. The same arguments as in the proof of \eqref{young} can be used to show that 
\begin{align*}
\sigma t^{\sigma-1}(1+|x|^2)e^{\frac{1}{2}(1+|x|^2)^{\rho}}\le &
 t^{\sigma}(1+|x|^2)^{\frac{\sigma}{\sigma-\delta}}e^{\frac{\sigma}{2(\sigma-\delta)}(1+|x|^2)^{\rho}}+\delta t^{\frac{\sigma(\delta-1)}{\delta}}\\
= &\delta t^{\frac{\sigma(\delta-1)}{\delta}}+t^{\sigma}(o(e^{(1+|x|^2)^{\max\{\gamma_{kk},\beta^k_{\min}\}}},\infty)
\end{align*}
for every $\delta\in (0,\sigma)$. If we take $\delta\in \left (\frac{\sigma}{\sigma+1},\sigma\right )$, then we can estimate 
$D_t\nu+({\bm\calA}^P\bm{\nu})_k\le g\nu$, where
$g(t)=\widetilde c+\varepsilon_T\delta t^{\frac{\sigma(\delta-1)}{\delta}}$ for every $t\in (0,T]$, and the function $g$ is integrable in $(0,T)$.

\medskip

We now pass to check Hypotheses \ref{hyp-stime-nucleo}. For this purpose, we fix 
$0<a_0<a<b<b_0<T$ and
$(t,x)\in [a_0,b_0]\times\R^d$
and consider the functions $w,\nu_1,\nu_2:[0,T]\times\R^d\to\R$, defined by
\begin{align*}
w(t,x)=\exp\bigg (\varepsilon t^\sigma\int_0^{1+|x|^2}
e^{\frac{\tau^{\rho}}{2}}d\tau\bigg ) ,\qquad\;\,\nu_j(t,x)=\exp\bigg (\varepsilon_j t^\sigma\int_0^{1+|x|^2}
e^{\frac{\tau^{\rho}}{2}}d\tau\bigg ),\;\,j=1,2,
\end{align*}
for every $(t,x)\in [0,T]\times\R^d$, where $\varepsilon<\varepsilon_1<\varepsilon_2<\varepsilon_T$.

Clearly, we can take $c_1=1$. To estimate the constant $c_2$ and the remaining constants, we 
note that, for every positive constants $a$, $b$, $\rho$, $c$, $\delta$, there exists a positive constant $C$, depending on the above constants but independent of $t\in(0,T]$, such that
\begin{equation}
(1+|x|^2)^ae^{b(1+|x|^2)^{\rho}}
e^{(1+|x|^2)^c}\le  
Ce^{\frac{1}{4}t^{-\sigma}}\exp\bigg (\delta t^{\sigma}\int_0^{1+|x|^2}e^{\frac{\tau^{\rho}}{2}}d\tau\bigg )  
\label{cibus}
\end{equation}
for every $t\in (0,T]$ and $x\in\R^d$.
To prove \eqref{cibus}, we begin by estimating
\begin{eqnarray*}
\exp\bigg (\delta t^{\sigma}\!\int_0^{1+|x|^2}e^{\frac{\tau^{\rho}}{2}}d\tau\bigg )\ge\exp\bigg (\delta t^{\sigma}\!
\int_{|x|^2}^{1+|x|^2}
e^{\frac{\tau^{\rho}}{2}}d\tau\bigg )\ge
\exp\Big (\delta t^{\sigma}
e^{\frac{|x|^{2\rho}}{2}}\Big ),\qquad\;\,(t,x)\in (0,T]\times\R^d.
\end{eqnarray*}
Hence, if we set $r=1+|x|^2$ and take the above estimate into account, then we easily realize that \eqref{cibus} is satisfied if the function $g:[1,\infty)\to\R$, defined by 
$g(r)=r^a\exp\left (br^{\rho}+r^c-\delta t^{\sigma}e^{\frac{1}{2}(r-1)^{\rho}}\right )$ for every $r\in [1,\infty)$, is bounded from above by a positive constant times $e^{\frac{1}{4}t^{-\sigma}}$. This follows from observing that
$g(r)\le \widetilde C\exp \left (br^{\rho}+2r^c-\delta t^{\sigma}e^{\frac{(r-1)^{\rho}}{2}}\right )$ for every $r\in [1,\infty)$ and a positive constant $\widetilde C$, which depends on $a$ and $c$, and 
$br^{\rho}+2r^c\le \delta e^{\frac{(r-1)^{\rho}}{4}}$ for every $r\in [r_0,\infty)$ for some $r_0>1$, which depends on $b$, $\rho$, $c$ and $\delta$. For such values of $r$ and recalling that $t\in (0,T]$, we can estimate $g(r)\le \widetilde C\exp \left (\delta\left (e^{\frac{(r-1)^{\rho}}{4}}-
t^{\sigma}e^{\frac{(r-1)^{\rho}}{2}}\right )\right )$. Since the maximum of the function $z\mapsto z-t^{\sigma}z^2$ is $(4t^{\sigma})^{-1}$, we conclude that $g(r)\le \widetilde C\exp\left (\frac{1}{4t^{\sigma}}\right )$ for every 
$r\in [r_0,\infty)$. Clearly, we can extend the previous estimate to every $r\in [1,\infty)$ up to replacing $\widetilde C$ with a larger constant, if needed. Inequality \eqref{cibus} follows.

We can now determine the constants $c_2,\ldots,c_8$. Since
\begin{align*}
(Q^h(x) \nabla w(t,x))_i
=2\varepsilon t^{\sigma}\bigg (\sum_{j=1}^d\xi_{ij}^hx_je^{(1+|x|^2)^{\alpha^h_{ij}}}e^{\frac{(1+|x|^2)^{\rho}}{2}}\bigg )w(t,x)
\end{align*}
for every $t\in (0,T]$, $x\in\R^d$ and $i=1,\ldots,d$, it follows that 
\begin{align*}
|Q^h(x) \nabla w(t,x)|
\le \widetilde c t^{\sigma}|x|e^{(1+|x|^2)^{\alpha^h_{\max}}}e^{\frac{(1+|x|^2)^{\rho}}{2}}w(t,x),\qquad\;\,
(t,x)\in (0,T]\times\R^d.
\end{align*}
Applying estimate \eqref{cibus}, with
$a=b=\frac{1}{2}$, $c=\alpha_{\max}^h$, $\delta=\frac{\varepsilon_1-\varepsilon}{s}$ and observing that
$t^{\sigma}\le T^{\sigma}$ for every $t\in [0,T]$, we conclude that
\begin{align*}
|Q^h(x) \nabla w(t,x)|\le \widetilde ct^{\sigma}e^{\frac{1}{4}t^{-\sigma}}
(w(t,x))^{\frac{s-1}{s}}(\nu_1(t,x))^{\frac{1}{s}},\qquad\;\,(t,x)\in (0,T]\times\R^d.
\end{align*}
Hence, we choose $c_2=\widetilde cb_0^{\sigma}e^{\frac{1}{4}a_0^{-\sigma}}$.

As far as $c_3$ is concerned, we observe that
$\mathrm{div}(Q^h(x)\nabla w(t,x))
=\mathfrak{a}_{1,k}(x,\varepsilon t^{\sigma})$
for every $(t,x)\in (0,T]\times\R^d$. 
Taking
$\delta=\frac{2(\varepsilon_1-\varepsilon)}{s}$ and choosing
properly $a$, $b$, $c$ in \eqref{cibus}, from \eqref{roxette} it follows that
\begin{align*}
|\mathrm{div}(Q^h(x)\nabla w(t,x))|
\leq \widetilde ct^{\sigma}e^{\frac14t^{-\sigma}}{w(t,x)^\frac{s-2}{s}\nu_1(t,x)^\frac{2}{s}}, \qquad (t,x)\in(0,T]\times \R^d.
\end{align*}
Hence, we can put $c_3=\widetilde c b_0^{\sigma}e^{\frac{1}{4}a_0^{-\sigma}}$.
Furthermore, $|D_tw(t,x)|\le \overline ct^{-1}$ for
every $(t,x)\in (0,T]\times\R^d$, so that 
we can take $c_4=\widetilde c a_0^{-1}$.

Finally, we observe that
\begin{eqnarray*}
\begin{array}{ll}
|V^h(x)|
\leq \widetilde c e^{(1+|x|^2)^{\gamma_{\max}}},\quad
&|b^h(x)|\le\widetilde c|x|e^{(1+|x|^2)^{\beta^h_{\max}}},\\[1mm]
|Q^h(x)|
\le\widetilde ce^{(1+|x|^2)^{\alpha^h_{\max}}},
&|R^h(x)|\le 
\widetilde c|x|(1+|x|^2)^{\alpha_{\max}^h-1}e^{(1+|x|^2)^{\alpha^h_{\max}}}
\end{array}
\end{eqnarray*}
for every $x\in\R^d$.
Hence, choosing properly $a$, $b$, $c$ and $\delta$ in \eqref{cibus}, it follows easily that we can take $c_5=c_6=c_7=c_8=\widetilde ce^{\frac{1}{4}a_0^{-\sigma}}$.

Applying Theorem \ref{thm: stima-nucleo_n}, with $s>d+2$, $a_0=\frac{t}{1+4\phi}$, $a=\frac{t}{1+3\phi}$, $b=\frac{t}{1+2\phi}$ and $b_0=\frac{t}{1+\phi}$ and some $\phi>0$, we can infer that, for every $(t,x,y)\in(0,T]\times \R^d\times \R^d$,
\begin{align*}
|p_{hk}(t,x,y)|
\leq C_{\phi}t\exp\bigg (\frac{s(1+4\phi)^{\sigma}}{4}t^{-\sigma}-\varepsilon t^{\sigma}\int_0^{1+|y|^2}e^{\frac{\tau^{\rho}}{2}}d\tau\bigg ), \qquad h,k=1,\ldots,m,
\end{align*}
where $C_{\phi}$ is a positive constant which only depends on $\phi$ and blows up as $\phi$ tends to $0$.
Since $s$ is any arbitrary fixed number greater than $d+2$, properly choosing $\phi>0$ we can replace
$4^{-1}s(1+4\phi)^{\sigma}$ by any constant $\hat c>\frac{d+2}{4}$ and infer that
for every $(t,x,y)\in(0,T]\times \R^d\times \R^d$,
\begin{align*}
|p_{hk}(t,x,y)|
\leq Ct\exp\bigg (\hat ct^{-\sigma}-\varepsilon t^{\sigma}\int_0^{1+|y|^2}e^{\frac{\tau^{\rho}}{2}}d\tau\bigg ), \qquad h,k=1,\ldots,m.
\end{align*}

Now we consider the operator $\bm\calA^{P,*}$ and assume the following
\begin{hyp}
\label{hyp-exp-2}
\begin{enumerate}[\rm (i)]
\item 
Conditions $(i)$ and $(iii)$ in Hypotheses $\ref{hp-polinomiale}$ are satisfied;
\item 
for every $k=1,\ldots,m$ it holds that
$\gamma_{kk}>\max\{\alpha^k_{\max},\beta^k_{\max},\gamma_{hk}: h\in\{1,\ldots,m\}\setminus\{k\}\}$.
\end{enumerate}    
\end{hyp}

We take $\varphi_*(x)=\displaystyle\exp\bigg (\hat\varepsilon_*\int_0^{1+|x|^2}e^{\frac{\tau^{\rho_*}}{2}}d\tau\bigg )$ for every $x\in\R^d$, some $\hat\varepsilon_*>0$ and $0<\rho_*<\gamma_{kk}$ for every $k=1,\ldots,m$, and observe that
\begin{align}
(\bm{\calA}^{P,*}\bm{\varphi}_*)_k(x)\le
\Big (&4\hat\varepsilon^2_*\zeta_{\max}^k|x|^2e^{(1+|x|^2)^{\alpha_\mathrm{max}^k}+(1+|x|^2)^{\rho_*}}
+2\hat\varepsilon_*\eta^k_{\max}|x|^2e^{(1+|x|^2)^{\beta_{\max}^k}+\frac{1}{2}(1+|x|^2)^{\rho_*}}\notag\\
&-\theta_{kk}e^{(1+|x|^2)^{\gamma_{kk}}}
+o(e^{(1+|x|^2)^{\gamma_{kk}}},\infty)\Big )\varphi_*(x)
\label{roxette_2}
\end{align}
for every $x\in\R^d$. 
Hypothesis \ref{hyp-exp-2}(ii) and condition $\rho^*<\gamma_{kk}$ for every $k=1,\ldots,m$ guarantee that the term in brackets in the right-hand side of \eqref{roxette_2} is bounded from above in $\R^d$, so that the function $\bm{\varphi}_*$ is a Lyapunov function for the operator $\bm{\calA}^{P,*}$. Moreover, again from Hypotheses \ref{hyp-exp-2} it follows that the function $\bm{\nu}_*:[0,T]\times\R^d\to\R^m$, with all components given by  
$\displaystyle\nu_*(t,x)=\exp\bigg (\varepsilon_{*,T}t^{\sigma_*}\int_0^{1+|x|^2}e^{\frac{\tau^{\rho_*}}{2}}d\tau\bigg )$ for every $(t,x)\in[0,T]\times \R^d$, with $\varepsilon_{*,T}=T^{-\sigma_*}\hat\varepsilon_*$ and some $\hat\varepsilon_*>0$, is a time-dependent Lyapunov function for the operator $D_t+\bm\calA^{P,*}$ with respect to $\varphi_*$ and a suitable $g_*\in L^1(0,T)$. Note that also Hypotheses \ref{hyp-stime-nucleo_*} are satisfied with constants $c_{*,j}$ ($j=1,\ldots,8)$ which have the same expression of the corresponding constants  $c_j$ ($j=1,\ldots,8)$. From Corollary \ref{coro:stima_peso_*}, it follows that, for every $\hat c>\frac{d+2}{4}$, there exists a positive constant $C$ such that
\begin{align*}
|p_{hk}^{P,*}(t,x,y)|
\leq Ct\exp\bigg (\hat ct^{-\sigma_*}-\varepsilon t^{\sigma_*}\int_0^{1+|x|^2}e^{\frac{\tau^{\rho_*}}{2}}d\tau\bigg ),
\qquad (t,x,y)\in(0,T]\times \R^d\times \R^d,
%\qquad\;\,\textrm{a.e. }y\in\R^d
\end{align*}
for every $h,k=1,\ldots,m$ and $\varepsilon>0$.
From Proposition \ref{prop-2.7}, we thus conclude that for every $\hat c>\frac{d+2}{4}$ there exists a positive constant $C$ such that, for every $(t,x,y)\in(0,T]\times \R^d\times \R^d$,
\begin{align*}
|p_{hk}(t,x,y)|
\leq Ct\exp\bigg [\hat ct^{-\sigma}-\frac{1}{2}\varepsilon t^{\sigma}\bigg (\int_0^{1+|y|^2}e^{\frac{\tau^{\rho}}{2}}d\tau+\int_0^{1+|x|^2}e^{\frac{\tau^{\rho}}{2}}d\tau\bigg )\bigg ]
\end{align*}
%for a.e. $y\in\R^d$, 
for every $h,k=1,\ldots,m$, every  $\rho\in (0,\gamma_{\min})$, provided that we assume the following
\begin{hyp}
\begin{enumerate}[\rm (i)]
\item 
Conditions $(i)$ and $(iii)$ in Hypotheses $\ref{hp-polinomiale}$ are satisfied;
\item 
for every $k=1,\ldots,m$ it holds that
$\gamma_{kk}>\max\{\alpha^k_{\max},\beta^k_{\max},\gamma_{hk},\gamma_{kh}: h\in\{1,\ldots,m\}\setminus\{k\}\}$.
\end{enumerate}    
\end{hyp}

\section{Appendix}
In this section, we provide some Schauder estimates which have been exploited in Proposition \ref{prop:conv_SMGR}.
Let $\Omega\subseteq\R^d$ be an open and connected domain and set $\eta_0=\min_{k=1,\ldots,m}\eta^0_k$.

\begin{theorem}
\label{teo-Schauder}
Let $\uu\in C^{1,2}([0,T]\times \Omega)$ be a solution to the system of differential equations $D_t\uu=\bm\calA^P\uu+\bm g$ in $[0,T]\times \Omega$ with $\bm g\in C^{\frac{\alpha}{2},\alpha}_{\rm loc}([0,T]\times \Omega)$ and such that $\uu(0,\cdot)\in C^{2+\alpha}_{\rm loc}(\Omega)$. Then, $\uu\in C^{1+\frac{\alpha}{2},2+\alpha}_{\rm loc}([0,T]\times \Omega)$ and for every compact sets $\Omega_1\Subset\Omega_2\Subset\Omega$ such that ${\rm dist}(\Omega_1,\Omega_2^c)>0$, there exists a positive constant $C=C(T,\eta_0,\Omega_1,\Omega_2)$ such that
\begin{align*}
\|\uu\|_{C^{1+\frac{\alpha}{2},2+\alpha}(\Omega_{1,T})}
\leq  C\big(\|\uu(0,\cdot)\|_{C^{2+\alpha}(\Omega_2)}+\|\uu\|_{C_b(\Omega_{2,T})}
+\|\bm g\|_{C^{\frac{\alpha}{2},\alpha}(\Omega_{2,T})}\big),
\end{align*}
where $\Omega_{j,T}=[0,T]\times\Omega_j$ $(j=1,2)$.
\end{theorem}
\begin{proof}
The proof follows the lines of that of \cite[Theorem 6.2.10]{LorRhabook21} (see also \cite[Theorem A.2]{AddAngLorTes17}). We stress that it is enough to prove the statement for $\Omega_1=B(x_0,r)$ and $\Omega_2=B(x_0,2r)$, where $x_0\in\Omega$ and $r\in(0,\frac12{\rm dist}(x_0,\Omega^c))$. Indeed, the general case, then will follow by a covering $\Omega_1$ with a finite number of balls such that the union of the double balls is contained in $\Omega_2$.

To enlighten the notation, we set $\R^d_T:=(0,T)\times\R^d$ and $B_T(x_0,R):=(0,T)\times B(x_0,R)$ for every $R>0$. Finally, along the proof, $C$ denotes a positive constant which may vary line by line.

Let $x_0,r,\Omega_1,\Omega_2$ be as above. We introduce the sequence $(r_n)$ given by $r_n=(2-2^{-n})r$ for every $n\in\N\cup\{0\}$. Let $\vartheta\in C^\infty(\R)$ be such that $\chi_{(-\infty,1]}\leq \vartheta\leq \chi_{(-\infty,2]}$, and for every $n\in\N\cup\{0\}$ let us set
$\vartheta_n(x)=\vartheta\left(1+\frac{|x-x_0|-r_n}{r_{n+1}-r_n}\right)$ for every $x\in\R^d$.
It follows that $\vartheta_n(x)=1$ if $|x-x_0|\leq r_n$ and $\vartheta_n(x)=0$ if $|x-x_0|\geq r_{n+1}$. We introduce the operator $\widetilde{\bm\calA}^P$, which acts on smooth functions $\f:\R^d\to \R^m$ as 
\begin{align*}
(\widetilde {\bm\calA} ^P \f)_i=\rho\ \!{\rm div}(Q^i\nabla\f)+(1-\rho)\Delta\f+\rho\langle b^i,\nabla\f\rangle-\rho(V^P\f)_i, \qquad\;\, i=1,\ldots,m,    
\end{align*}
where $\rho\in C^\infty_c(\R^d)$ satisfies $\rho(x)=1$ for every $x\in \overline{B(x_0,2r)}$.

We set $\uu_n:=\vartheta_n\uu$ for every $n\in\N$. Then, $\uu_n\in C^{1,2}([0,T]\times \R^d)$ and it satisfies the equation $D_t\uu_n=\widetilde{\bm\calA}\uu_n+\bm g_n$ in $(0,\infty)\times \R^d$, where $g_{n,i}=\vartheta_n g_i-u_i({\rm div}(Q^i\nabla\vartheta_n)+\langle b^i,\nabla\vartheta_n)\rangle-2Q^i\langle \nabla u_i,\nabla\theta_n\rangle$ for every $i=1,\ldots,m$.
It follows that, for every $n\in\N\cup\{0\}$,
\begin{align*}
\|\uu_n\|_{C^{1+\frac\alpha2,2+\alpha}_b(\R^d_T)}\leq C\big(\|\uu_n(0,\cdot)\|_{C_b^{2+\alpha}(\R^d)}+\|\bm g_n\|_{C^{\frac\alpha2,\alpha}_b(\R^d_T)}\big).
\end{align*}
Taking $n=0$ we deduce that $\uu\in C^{1+\frac\alpha2,2+\alpha}(B_T(x_0,r))$ and the arbitrariness of $r$ and $x_0$ implies that $\uu\in C^{1+\frac\alpha2,2+\alpha}_{\rm loc}([0,T]\times \Omega)$.

We notice that, for every $n\in\N$,
\begin{align*}
\|\vartheta_n\uu(0,\cdot)\|_{C^{2+\alpha}_b(\R^d)}
\leq \|\vartheta_n\|_{C^{3}_b(\R^d)}\|\uu(0,\cdot)\|_{C_b^{2+\alpha}(B(x_0,2r))}
\leq C 2^{3n}\|\uu(0,\cdot)\|_{C_b^{2+\alpha}(B(x_0,2r))}.
\end{align*}
Hence, taking the definition of $\bm g_n$ into account, we deduce that
\begin{align*}
\|\uu_n\|_{C^{1+\frac\alpha2,2+\alpha}_b(\R^d_T)}
\leq & 2^{3n}C\Big(\|\uu(0,\cdot)\|_{C_b^{2+\alpha}(B(x_0,2r))}+\|\uu\|_{C^{\frac\alpha2,1+\alpha}_b(B_T(x_0,r_{n+1}))}+ \|\bm g\|_{C_b^{\frac\alpha2,\alpha}(B_T(x_0,2r))}\Big) \\
\leq &  2^{3n}C\Big(\|\uu_{n+1}\|_{C^{\frac\alpha2, \alpha}_b(\R^d_T)}+\sum_{j=1}^d\|D_j\uu_{n+1}\|_{C^{\frac\alpha2, \alpha}_b(\R^d_T)} \\
&\qquad\quad + \|\uu(0,\cdot)\|_{C_b^{2+\alpha}(B(x_0,2r))}+\|\bm g\|_{C_b^{\frac\alpha2,\alpha}(B_T(x_0,2r))}\Big).
\end{align*}
From well-known estimates (see for instance \cite[Chapter 1]{LorRhabook21}, which we apply componentwise) it follows that
\begin{align*}
\|\uu_n\|_{C^{1+\frac\alpha2,2+\alpha}_b(\R^d_T)}
\leq & 2^{3n}C\Big(\varepsilon\|\uu_{n+1}\|_{C^{1+\frac\alpha2,2+\alpha}_b(\R^d_T)}+\varepsilon^{-(1+\alpha)}\|\uu_{n+1}\|_\infty \\
&\qquad\quad+\|\uu(0,\cdot)\|_{C_b^{2+\alpha}(B(x_0,2r))}+\|\bm g\|_{C_b^{\frac\alpha2,\alpha}(B_T(x_0,2r))}\Big)
\end{align*}
for every $\varepsilon>0$. Now we fix $\eta\in(0,2^{-3(2+\alpha)})$ and choose $\varepsilon=\varepsilon_n=2^{-3n}C^{-1}\eta$. Setting $\zeta_n:=\|\uu_n\|_{C^{1+\frac\alpha2,2+\alpha}_b(\R^d_T)}$ for every $n\in\N\cup\{0\}$, the above estimate implies that
\begin{align*}
\zeta_n
\leq \eta\zeta_{n+1}+2^{3n(2+\alpha)}C\|\uu\|_{C_b(B_T(x_0,2r))} +2^{3n}C\big( \|\uu(0,\cdot)\|_{C_b^{2+\alpha}(B(x_0,2r))} +\|\bm g\|_{C_b^{\frac\alpha2,\alpha}( B_T(x_0,2r))}\big).    
\end{align*}
Multiplying both the sides by $\eta^n$ and summing from $0$ to $m\in\N$, it follows that
\begin{align*}
\zeta_0-\eta^{m+1}\zeta_{m+1}
\leq & C\|\uu\|_{C_b(B_T(x_0,2r))}\sum_{n=0}^m(2^{3(2+\alpha)}\eta)^n\\
&+ C\Big ( \|\uu(0,\cdot)\|_{C_b^{2+\alpha}(B(x_0,2r))}+\|\bm g\|_{C_b^{\frac\alpha2,\alpha}( B_T(x_0,2r))}\Big )\sum_{n=0}^m(2^{3}\eta)^n \\ 
\leq & C\Big(\|\uu\|_{C_b( B_T(x_0,2r))}+\|\uu(0,\cdot)\|_{C_b^{2+\alpha}(B(x_0,2r))}+\|\bm g\|_{C_b^{\frac\alpha2,\alpha}(B_T(x_0,2r))}\Big),
\end{align*}
since both the series in the right-hand side converge. Finally, we claim that $(\eta^{m}\zeta_{m})$ vanishes as $m$ tends to $0$. Indeed, we have proved that $\uu\in C^{1+\frac\alpha2,2+\alpha}([0,T]\times K)$ for every $K\Subset \Omega$. Hence, for every $n\in\N$ we get
\begin{align*}
\zeta_n
\leq \|\theta_n\|_{C^3_b(\R^d)}\|\uu\|_{C_b^{1+\frac\alpha2,2+\alpha}(B_T(x_0,2r))}\leq 2^{3n}C\|\uu\|_{C_b^{1+\frac\alpha2,2+\alpha}(B_T(x_0,2r))}.
\end{align*}
The choice of $\eta$ gives the claim and so
\begin{align*}
\|\uu\|_{C^{1+\frac\alpha2,2+\alpha}_b( B_T(x_0,r))}
\leq \zeta_0
\leq C\Big(\|\uu\|_{C( B_T(x_0,2r))}+\|\uu(0,\cdot)\|_{C_b^{2+\alpha}(B(x_0,2r))}+\|\bm g\|_{C_b^{\frac\alpha2,\alpha}(B_T(x_0,2r))}\Big).
\end{align*}

\end{proof}

\end{document}